\documentclass[12pt]{amsart}
\usepackage{amssymb,xspace,url,amscd,euscript,mathrsfs,stmaryrd,
epic,eepic,longtable}
\usepackage[centering,text={14.5cm,23cm}]{geometry}
\usepackage{enumitem}
\usepackage{graphicx}
\usepackage[dvipsnames]{xcolor}
\usepackage[all]{xy}
\usepackage{mathrsfs}
\usepackage{marvosym}
\usepackage{stmaryrd}
\usepackage{hyperref}

\definecolor{shadecolor}{rgb}{1,0.9,0.7}

\let\oldtocsection=\tocsection
 
\let\oldtocsubsection=\tocsubsection
 
\let\oldtocsubsubsection=\tocsubsubsection
 
\renewcommand{\tocsection}[2]{\hspace{0em}\oldtocsection{#1}{#2}}
\renewcommand{\tocsubsection}[2]{\hspace{1em}\oldtocsubsection{#1}{#2}}
\renewcommand{\tocsubsubsection}[2]{\hspace{2em}\oldtocsubsubsection{#1}{#2}}

\newtheorem{theorem}{Theorem}[section]
\newtheorem{lemma}[theorem]{Lemma}
\newtheorem{proposition}[theorem]{Proposition}
\newtheorem{corollary}[theorem]{Corollary}

\theoremstyle{definition}
\newtheorem{definition}[theorem]{Definition}
\newtheorem{construction}[theorem]{Construction}
\newtheorem{discussion}[theorem]{Discussion}

\newtheorem{example}[theorem]{Example}

\theoremstyle{remark}
\newtheorem{remark}[theorem]{Remark}

\numberwithin{equation}{section}
\numberwithin{figure}{section}

\newcommand{\ChB}[1]{{#1}}

\newcommand {\lfor} {\llbracket}
\newcommand {\rfor} {\rrbracket}

\newcommand{\BB} {\mathbb{B}}
\newcommand{\DD} {\mathbb{D}}

\newcommand{\NN} {\mathbb{N}}
\newcommand{\ZZ} {\mathbb{Z}}

\newcommand{\RR} {\mathbb{R}}
\newcommand{\CC} {\mathbb{C}}

\newcommand{\PP} {\mathbb{P}}
\renewcommand{\AA} {\mathbb{A}}

\newcommand {\shA}  {\mathcal{A}}

\newcommand {\shE}  {\mathcal{E}}
\newcommand {\shF}  {\mathcal{F}}

\newcommand {\shHom} {\mathcal{H}\!\text{\textit{om}}}

\newcommand {\shM}  {\mathcal{M}}

\newcommand {\shP}  {\mathcal{P}}

\newcommand {\shX}  {\mathcal{X}}
\newcommand {\shY}  {\mathcal{Y}}
\newcommand {\shZ}  {\mathcal{Z}}

\newcommand {\Arg}  {\operatorname{Arg}}

\newcommand {\cl}  {\operatorname{cl}}

\newcommand {\conv} {\operatorname{Conv}}

\newcommand {\Ext}  {\operatorname{Ext}}
\newcommand {\ev}  {\operatorname{ev}}

\newcommand {\GL}  {\operatorname{GL}}

\newcommand {\gp}  {{\operatorname{gp}}}

\newcommand {\Hom}  {\operatorname{Hom}}

\newcommand {\id}  {\operatorname{id}}
\newcommand {\im}  {\operatorname{im}}

\newcommand {\Int}  {\operatorname{Int}}

\renewcommand {\ker } {\operatorname{ker}}
\newcommand {\KN} {\mathrm{KN}}

\newcommand {\lra}  {\longrightarrow}
\newcommand {\ls}  {\dagger}

\newcommand {\M} {\mathcal{M}}

\renewcommand {\max} {{\operatorname{max}}}
\newcommand {\maxid} {\mathfrak{m}}

\renewcommand{\O}  {\mathcal{O}}

\newcommand {\ol} {\overline}

\renewcommand{\P}  {\mathscr{P}}

\newcommand {\PM} {\operatorname{PM}}

\newcommand {\shLS} {\mathcal{LS}}

\newcommand {\sing} {\mathrm{sing}}

\newcommand {\Spec} {\operatorname{Spec}}
\newcommand {\Specan} {\operatorname{Specan}}

\newcommand {\ul} {\underline}

\newcommand {\X} {\mathfrak X}

\def\mydate{\ifcase\month \or January\or February\or March\or
April\or May\or June\or July\or August\or September\or October\or 
November\or December\fi \space\number\day,\space\number\year}

\newcommand\restr[2]{{
  \left.\kern-\nulldelimiterspace 
  #1 
  \vphantom{\big|} 
  \right|_{#2} 
  }}

\begin{document}

\title[Real locus of Kato-Nakayama spaces]{Real loci in (log-) Calabi-Yau manifolds via Kato-Nakayama spaces of toric degenerations}

\author{H\"ulya Arg\"uz}
\address{Laboratoire de Math\'ematiques, Universit\'e de
Versailles St Quentin en Yvelines, France}
\email{nuromur-hulya.arguz@uvsq.fr}


\date{\today}

\begin{abstract}
We study the real loci of toric degenerations of complex varieties with
reducible central fibre. We show that the topology of such degenerations can be
explicitly described via the Kato-Nakayama space of the central fibre as a log
space. We furthermore provide generalities of real structures in log geometry and
their lift to Kato-Nakayama spaces. A key point of this paper is a description of the Kato-Nakayama space
of a toric degeneration and its real locus, both as bundles determined by
tropical data. We provide several examples including real toric degenerations of K3-surfaces and a
toric degeneration of local $\PP^2$.
\end{abstract}

\maketitle
\tableofcontents

\section*{Introduction.}
The Strominger–Yau–Zaslow \cite{SYZ} conjecture in mirror symmetry postulates that mirror pairs of Calabi--Yau threefolds $(X,\breve{X})$ admit dual special Lagrangian torus fibrations 
\[ f\colon X\to B, \,\ \,\ \mathrm{and} \,\ \,\  \breve{f} \colon \breve{X}\to B \]
over a real three-dimensional affine
manifold $B$ with singular fibers over a discriminant locus $\Delta \subset B$. After weakening the condition on the torus fibration, Ruan \cite{Ru1,Ru2,Ru3},
Gross \cite{GrossInv,GrossTopology,GrossGeometry}, and Haase–Zharkov \cite{HZ1,HZ2,HZ3} have constructed integral affine manifolds together
with topological or Lagrangian torus fibrations on Calabi–Yau toric complete intersections. In this article, we provide new tools to construct topological torus fibrations on Calabi--Yau manifolds in bigger generality, by using log geometry and Kato--Nakayama spaces \cite{KN}, which are topological spaces resembling some features of exploded manifolds studied in \cite{P}. In particular, we consider \emph{toric degenerations} of Calabi--Yau varieties, introduced
by Gross–Siebert \cite{invitation}, and describe how to construct torus fibrations on the general fibre of such degenerations, by investigating the \emph{Kato--Nakayama space} of the central fiber. We further study the topology of the Kato--Nakayama space of the central fiber, as well as the real loci in it, and provide canonical descriptions for both the Kato--Nakayama space and its real loci as bundles determined by tropical data.

A toric degeneration is a
degeneration of algebraic varieties $\delta:\X\to T=\Spec R$ with $R$ a discrete
valuation ring and with central fibre $X_0=\delta^{-1}(0)$ a union of toric
varieties, glued pairwise along toric divisors. Here $0\in \Spec R$ is the
closed point. We also require that $\delta$ is toroidal at the zero-dimensional
toric strata, that is, \'etale locally near these points, $\delta$ is given by a
monomial equation in an affine toric variety. Probably the most remarkable aspect of toric degenerations is that the homogeneous coordinate ring for the total space of such degenerations can be
produced \emph{canonically} from the central fibre $X_0$ and some residual
information on the family $\X$ \cite{affinecomplex}, captured by what is called a \emph{log
structure}. By residual information we mean the following: By Proposition \ref{Prop: description of shX in codim=1} we can assume at a general point of the singular locus of $X_0$, the local equation of $\shX$ is given as 
\begin{equation}
    \label{Eq: local eqn}
xy=f\cdot t^e    
\end{equation}
where $t\in R$
generates the maximal ideal.
Then, the log structure captures $e\in \mathbb{N}$ and
the restriction of $f$ to $x=y=0$. The information of the log
structure on $X_0$ is encoded in the \emph{gluing data}
\[ s\in
H^1(B,\iota_* \check\Lambda\otimes\CC^\times)\]
where $\iota:B_0=B\setminus \Delta \to B$ is the inclusion and $\check\Lambda$ denotes the sheaf of integral cotangent vectors on $B_0$. For a review of integral affine manifolds, log structures, and gluing data see \cite{affinecomplex,theta}. 

One of the key results of \cite{affinecomplex} is a reconstruction algorithm of toric degenerations from the data of a \emph{toric log Calabi--Yau space} $(X_0,\M_{X_0})$ \cite[Defn.4.3]{logmirror1}, which plays the role of the central fiber. Here, $X_0$ is a proper reduced algebraic space, and $\M_{X_0}$ stands for the log structure on it. 
While the reconstruction of a toric degeneration is done by an inductive procedure involving
a \emph{wall structure} \cite{affinecomplex}, and is typically impossible to
carry through in practice, many features of the family are already contained in
the log structure. In this paper we study topological features of toric degenerations by investigating the log structure on $X_0$.

We first show that the
topology of the toric degeneration reconstructed from a toric log Calabi--Yau $(X_0,\M_{X_0})$ can be read off canonically by considering the \emph{Kato--Nakayama space}, or \emph{Betti realization} \cite{KN} of $(X_0,\M_{X_0})$. We review the construction of the Kato--Nakayama space along with several examples in \S\ref{Sect: Kato-Nakayama space}.
One of the major outcomes of \cite{NO} in this context, displayed as Theorem \ref{Thm: Nakayama/Ogus}, shows that under certain assumptions a morphism  $f:(X,\mathcal{M}_X)\to (Y,\mathcal{M}_Y)$ induces a topological fiber bundle, that is, a continuous surjection which satisfies local triviality,
\[ f^{KN}:(X,\mathcal{M}_X)^{KN} \lra (Y,\mathcal{M}_Y)^{KN}. \]
Following this intuition we obtain the following result.
\begin{theorem}(= Theorem \ref{Thm: KN for X0 in simple case})
Let $\delta:\shX\to \DD=\{z\in\CC\,|\, |z|<R\}$ be a toric degeneration with
strongly semi-simple singularities (Definition~\ref{Def: simple singularities}).
Then the restriction $\shX\setminus X_0\to \DD\setminus\{0\}$ is a topological
fiber bundle isomorphic to
\[
\delta^{KN}\times\id: X_0^\KN\times(0,r)\lra (\O^{\dagger})^{KN} \times(0,r) = U(1)\times(0,r).
\]
In particular, for $t= re^{i\theta}\in\DD$, the restriction of this isomorphism
to the fibre over $(e^{i\theta},r)$ induces a homeomorphism
\begin{equation}
\label{Eq: restriction to a phase}
\shX_t \simeq X_0^\KN(e^{i\theta})
\end{equation}
between the general fibre of $\delta$ and the fiber $X_0^\KN(e^{i\theta})$ of $\delta^{KN}$ over $e^{i\theta}\in U(1)$.
\end{theorem}
In \S\ref{sec: The Topological Classification of Torus Bundles} we further study the topology of the Kato--Nakayama space of a toric log Calabi--Yau, and investigate topological torus fibrations on the general fibre of the associated degeneration. Such fibrations for specific classes of Calabi--Yau threefolds were constructed in \cite{GrossInv}, by describing suitable Calabi--Yau compactifications of smooth torus fibrations over the discriminant locus on the base. Constructing topological torus fibrations via Kato--Nakayama spaces offers a number of advantages. In particular, it replaces a single Calabi--Yau compactification with a moduli space of compactifications parametrised by gluing data. Moreover, it holds at any dimension and is not restricted to three dimensional Calabi--Yau's. Furthermore, this construction is intrinsic to the Gross--Siebert program in mirror symmetry and  though for simplicity of the exposition we concentrate on the Calabi--Yau case throughout this article, all results have straight forward generalizations to the context of log Calabi--Yau varieties which also play important role in this program \cite{GHK, theta}. 

To provide an explicit description of topological torus fibrations on the Kato--Nakayama space $X_0^{\mathrm{KN}}$ we first define momentum maps from $X_0$ onto a topological manifold $B$, which is realised as the intersection complex of $X_0$. This description does not require reference to a symplectic structure. We then obtain the following result.
\begin{theorem}(= Theorem \ref{Thm: (X_0)^KN->B is a torus bundle})
The composition of the generalized momentum map $\mu: X_0 \to B$, and the natural retraction map $r: X_0^{\mathrm{KN}} \to X_0$,
\[ \mu \circ r: X_0^{\mathrm{KN}} \to B, \]
 is a
bundle of $(n+1)$-tori over $B\setminus\shA$, where $n = \dim B$, and $\shA$ is the amoeba image of the discriminant locus in $B$. Similarly, $X_0^\KN(\xi)$ is a bundle of $n$-tori over $B\setminus\shA$.
\end{theorem}
  Theorem \ref{Thm: (X_0)^KN->B is a torus bundle} provides a description of a topological torus fibration on the general fibre of a torus degeneration away from the discriminant locus. Studying singular fibres over the discriminant locus amounts to analysing the topology of the Kato--Nakayama space over points in $X_0$ which do not carry a \emph{fine} log structure. For a concrete analysis in a two dimensional case, where the monodromy around the discriminant locus is a Dehn twist, see Examples \ref{Expl: Focus-focus singularity I} and \ref{Expl: Focus-focus singularity II}. 

We note that the families obtained by smoothing toric log Calabi--Yau spaces can always be
chosen projective under weak assumptions on the gluing data \cite[Thm 2.34]{logmirror1}, and in fact, each family can be
deformed into a projective one, see \cite{theta}, Proposition~5.12 and
Theorem~A.7. The choice of a polarization is encoded in a multivalued piecewise linear (MPL) function $\varphi$ on $B$ (for a short overview on this see \cite[\S1.1]{RS2}).
 From the data of the MPL function $\varphi$ encoding a choice of a polarization, along with the gluing data $s$, we obtain an extension class in \eqref{extension-class} given by
\begin{equation}
\nonumber
\big(c_1(\varphi), \Arg(s)\big) \in \Ext^1(\Lambda,\ul\ZZ\oplus \ul U(1))=
H^1(B\setminus\shA, \check\Lambda)\oplus H^1(B\setminus\shA,
\check\Lambda\otimes U(1)).
\end{equation}
where $\Lambda$ and $\check\Lambda$ respectively denote the sheaf of integral tangent and cotangent vectors on $B$, and $c_1(\varphi)$ is the first Chern class of $\varphi$, as in Definition \ref{Def: The chern class}. Taking this extension and an
extension of $\Lambda$ by $\ul\ZZ$ as in \eqref{Eqn: Aff^* extension} as we obtain the following extension of $\Lambda$, fitting into the commutative diagram in \eqref{Eqn: X_0^KN as an extension}:
\begin{equation}
0 \lra \ul\ZZ\oplus\ul U(1) \lra \widehat\shP \lra \Lambda \lra 0
\end{equation}
Use the extension $\widehat\shP$ of $\Lambda$, we arrive at a canonical description of the Kato--Nakayama space $X^{KN}_0$, over a subset $B'\subset B$ covered by the interiors of the maximal cells and vertices of $\P$:
\begin{proposition}(= Proposition~\ref{Prop: X_0^KN versus Hom(widehat shP,U(1))})
There is a
canonical homeomorphism
\[
\Hom^\circ(\widehat\shP, \ul U(1))|_{B'} \stackrel{\simeq}{\lra}
X_0^\KN\big|_{B'}
\]
of topological fibre bundles over $B'$,  where $\Hom^\circ(\widehat\shP,\ul U(1)) \subset \Hom (\widehat\shP, \ul
U(1))$ is the space of fibrewise homomorphisms restricting to the
identity on $\ul U(1)\subset \widehat\shP$.
\end{proposition}

In the second part of the paper we study real structures in toric degenerations, and also obtain an analogous canonical description of real loci. In \S\ref{Sect: Real log spaces} we first introduce the straightforward notion
of a real structure on a log space $(X,\M_X)$, where $X$ is a real scheme with an anti-holomorphic involution
$\iota_{X} :X\to X$, as an
involution
\[
\tilde\iota_{X}= (\iota_X,\iota_X^\flat): (X,\M_X)\lra (X,\M_X)
\]
of log schemes over $\RR$ with underlying scheme-theoretic morphism
$\iota_{X}$. After treating some basic properties of real log schemes, we show that the real involution on a real log scheme lifts canonically to its
Kato-Nakayama space in \S~\ref{Sect: Real structure on KN}. Our primary interest in this paper are real structures in a toric log Calabi--Yau space $X_0$ and their lift
to $X_0^\KN$. One of the reasons for being interested in real structures in this
context is that the real locus produces natural Lagrangian submanifolds on any
complex projective manifold defined over $\RR$. Assuming the analytic model (or  the analytification of)
$\shX$ is defined over $\RR$, it comes with a natural family of degenerating
Lagrangian submanifolds. Again we can study these Lagrangians by means of their
analogues in $X_0^\KN$. We show that the real locus $X^\KN_\RR\subset X^\KN$ defines real structures in toric degenerations, by showing that the fibre bundle
structure for $\delta^\KN: X_0^\KN\to S^1$ can be chosen to respect the real
locus $X_{0,\RR}^\KN\subset X_0^\KN$:
\begin{proposition}(= Proposition \ref{Prop: pairs})
The fibre bundle
structure for $\delta^\KN: X_0^\KN\to S^1$ can be chosen to respect the real
locus $X_{0,\RR}^\KN\subset X_0^\KN$ in the sense that the local trivializations
identify the real subset with a submanifold of the fibres.
\end{proposition}

In \S\ref{Subsect: X_0^KN real}, we further investigate standard real structures on toric log Calabi--Yau spaces, and obtain a necessary and sufficient condition for the existence of a real structure on a toric log Calabi Yau space:
\begin{proposition}(= Proposition \ref{Prop: Gluing data and slab function for the standard real structure})
Let $(X_0,\M_{X_0})$ be a polarized toric log Calabi-Yau space with intersection complex
$(B,\P)$. Then there is a standard real structure on
$(X_0,\M_{X_0})$ if and only if 
there exist lifted open gluing data $s=(s_v)_{v\in \tau}$ with
$X_0\simeq X_0(B,\P,s)$ such that $s_v \in \check{\Lambda}_\sigma \otimes_\ZZ \mathbb{R}^*$.
\end{proposition}
We refer to the gluing data $s\in \check{\Lambda}_\sigma \otimes_\ZZ \mathbb{R}^*$ as \emph{real lifted open gluing data}. 
One of our main results in this article is the following theorem, which provides a canonical description of the Kato--Nakayama space of toric log Calabi--Yau space defined by real lifted open gluing data.
\begin{theorem}(=  Theorem \ref{Thm: X_{0,RR}^KN in terms of extensions})
Let $(X_0,\M_{X_0})$ be a toric log Calabi-Yau space defined by real lifted open gluing data. Then the real locus in $X_0^\KN$ over $B'\subset B$ is given by
\[
X_{0,\RR}^\KN = \Hom^{\circ}(\widehat\shP,\{\pm1\})|_{B'} \subset
\Hom^{\circ}(\widehat\shP,U(1))|_{B'}.
\] 
where $\Hom^{\circ}(\widehat\shP,U(1))|_{B'}$ is defined as in Proposition~\ref{Prop: X_0^KN versus Hom(widehat shP,U(1))}.
\end{theorem} 
Theorem \ref{Thm: X_{0,RR}^KN in terms of extensions} apriori describes the real locus as a section over $B'\subset B$. We discuss the  necessary criteria to extend the real locus as a continuous section over $B$ in Proposition \ref{Prop: Topological torus fibrations}. Analysing further the topology of the Kato--Nakayama space $X_0^\KN$, for $X_0$ is defined by real lifted open gluing data, in Proposition \ref{Prop: real locus is a branched cover of B} we prove that the restriction of $\mu^\KN: X_0^\KN\to B$ to the real locus
exhibits $X_{0,\RR}^\KN$ as a surjection with finite fibres. Over
$B\setminus \shA$, this map is a topological covering map with
fibres of cardinality $2^{n+1}$. As a concrete example we study in \S\ref{Subsect: quartic
K3} real loci in a real toric degeneration of quartic $K3$ surfaces. For a specific choice of the gluing data we reproduce a
result of Casta\~no-Bernard and Matessi \cite{Castano/Bernard-Matessi} on the
topology of the real locus of an SYZ-fibration with compatible real involution
in our setup. We also investigate a non-trivial real structure on the local $\PP^2$.

\textbf{Related work:} 
The real locus in algebraic varieties is studied extensively in the mathematics literature from different perspectives  \cite{Delaunay,MJT,DK,Huisman,N,NS}. Of particular interest is to understand the cohomology groups of real loci. In \cite{Castano/Bernard-Matessi}, \S3 the mod $2$ cohomology of real loci is also studied via a spectral sequence relating it to Hodge groups, see  in
connection with \cite{logmirror2}. In \cite{AP}, in our joint work with Thomas Prince, we continue this study from the point of view
developed here, in the case of Calabi--Yau threefolds. Among other things, the rank of the mod $2$ first cohomology
group for a toric degeneration of quintic threefolds turns out to be $101$,
agreeing with the Hodge number $h^{1,1}$. These observations indicate
deeper meaning of our real Lagrangians in the study of Hodge theoretic mirror symmetry, which is investigated in \cite{AP2}. The study of real loci is also essential from the point of view of Gromov--Witten theory and enumerative real algebraic geometry. Once we have a real Lagrangian $L\subset \shX_t$, a holomorphic disc with
boundary on $L$ glues with its complex conjugate to a rational curve $C\subset
\shX_t$ with a real involution. Real rational curves are amenable to techniques
of algebraic geometry and notably of log Gromov-Witten theory of the central
fibre $X_0$. Thus real Lagrangians provide an algebraic-geometric path to open
Gromov-Witten invariants and the Fukaya category. See \cite{Solomon,PSW,FOOO,georgieva} for some previous work in this direction
without degenerations. We also note that another very influential method of producing algebraic varieties with
real structures combinatorially is Viro's patchworking method \cite{Viro}. This
method produces real algebraic hypersurfaces inside toric varieties with
prescribed topology. Bounds on the mod 2 Betti numbers of the real locus for such varieties have
been studied by Bihan, Itenberg, and Viro \cite{B,I,IV} and, very recently, by Renaudineau–
Shaw \cite{RS} via the introduction of a real analog of tropical homology groups \cite{IKMZ}. Our method here is substantially different in that our
varieties typically do not embed as hypersurfaces in toric varieties. On the
downside, our setup is restricted to the toric degenerations treated in
\cite{affinecomplex} and hence only produces varieties with effective
anti-canonical divisor. A certain overlap of the two methods concerns
hypersurfaces in toric varieties with effective anticanonical bundle. 
\noindent 

It is natural to ask if the Kato–Nakayama space of the central fiber of a toric degeneration, can be
obtained as a topological Calabi–Yau compactification as in \cite{RZ1}. This is further investigated in \cite{RZ2}. 

\emph{Acknowledgements}: This work was part of my PhD thesis, which would not be possible without the support of my adviser Bernd Siebert. I am also indebted to Mark Gross for many useful conversations; a major part of this paper was written during a visit to the University of Cambridge hosted by him. I am grateful to Tom Coates and Dimitri Zvonkine for their useful suggestions which improved the exposition of the paper. Finally, many thanks to an anonymous referee for their very useful feedback and corrections. This project has
received funding from the ``Research Training Group 1670 Mathematics inspired by String Theory'' at Universit\"at Hamburg funded by
the Deutsche Forschungsgemeinschaft (DFG), the European Research Council (ERC) under the European Union’s
Horizon 2020 research and innovation programme (grant agreement No. 682603), and from Fondation Math\'ematiques Jacques Hadamard.

\emph{Conventions.}
We work in the category of complex analytic spaces.\footnote{
The discussions in \S\ref{Sect: Real log spaces} on real structures in log
geometry and in \S\ref{Subsect: X_0^KN real} on real toric degenerations also
hold in the category of schemes over $\RR$.} For $R$ a finitely generated
$\CC$-algebra we write $\Specan R$ for the analytic space associated to the
complex scheme $\Spec R$.

For $a= r e^{i\varphi}\in\CC\setminus\{0\}$ we denote by $\arg(a)=\varphi\in
\RR/2\pi i\ZZ$ and by $\Arg(a)= e^{i\varphi}= a/|a|$. The unit disc in $\CC$ is
denoted $D$.
\bigskip


\section{{Kato-Nakayama spaces with a view toward toric degenerations}}
\label{Sect: Kato-Nakayama space}

\subsection{Generalities on Kato-Nakayama spaces}

Throughout this section we assume basic familiarity with log geometry at the level
of \cite{Fkato}. Logarithmic structures in the analytic category have been
introduced in \cite[Def.1.1.1]{KN}. For more details we encourage the
reader to also look at \cite{Kato,Ogus}. The structure homomorphism of a
log space $(X,\M_X)$ is denoted $\alpha_X:\M_X\to \O_X$, or just $\alpha$ if $X$
is understood. The standard log point $(\Spec \CC, \NN\oplus\CC^\times)$ is
denoted $O^\ls$.

To any log analytic space $(X,\M_X)$, Kato and Nakayama in \cite{KN}
have associated functorially a topological space $(X,\M_X)^\KN$, its
\emph{Kato-Nakayama space} or \emph{Betti-realization}. We review this
definition and its basic properties first before discussing the additional
properties coming from a real structure. Denote by $\Pi^\ls= (\Specan\CC,
\M_\Pi)$ the \emph{polar log point}, with log structure
\[
\alpha_\Pi:\M_{\Pi,0}= \RR_{\ge 0}\times U(1)\lra \CC,\quad
(r,e^{i\varphi})\longmapsto r\cdot e^{i\varphi}.
\]
The map $\Pi^\ls\to \Specan\CC$ \ChB{forgetting the logarithmic structure} makes
$\Pi^\ls$ into a log space over $\CC$. Note $\ol\M_{\Pi,0} = U(1)$, so this log
structure is not fine. As a set define
\[
(X,\M_X)^\KN:= \Hom\big(\Pi^\ls, (X,\M_X)\big),
\]
the set of morphisms of complex analytic log spaces $\Pi^\ls\to (X,\M_X)$. Note
that a log morphism $f: \Pi^\ls \to (X,\M_X)$ is given by (a) its
set-theoretic image, a point $x=\varphi(0)\in X$, and (b) a monoid
homomorphism $f^\flat:\M_{X,x}\to \RR_{\ge 0}\times U(1)$. Forgetting the monoid
homomorphism thus defines a map of sets
\begin{equation}\label{Eqn: pi}
\pi: (X,\M_X)^\KN\lra X.
\end{equation}
We endow $(X,\M_X)^\KN$ with the following topology. A local section
\ChB{$s\in\Gamma(U,\M_X)$,} $U\subset X$ open, defines a map
\begin{equation}\label{Eqn: ev_s}
\ev_s: \pi^{-1}(U)\lra \RR_{\ge 0}\times U(1),\quad
f\longmapsto f^\flat\circ s. 
\end{equation}
As a subbasis of open sets on $(X,\M_X)^\KN$ we take
$\ev_s^{-1}(V)$, for any $U\subset X$ open,
\ChB{$s\in\Gamma(U,\M_X)$} and $V\subset \RR_{\ge0}\times U(1)$ open.
The forgetful map $\pi$ is then clearly continuous.

If the log structure is understood, we sometimes write $X^\KN$
instead of $(X,\M_X)^\KN$ for brevity.

\begin{remark}
The following more explicit set-theoretic description of
$(X,\M_X)^\KN$ is sometimes useful. A log morphism
$f:\Pi^\ls\to (X,\M_X)$ with $f(0)=x$ is equivalent to a choice of
monoid homomorphism $f^\flat$ fitting into the commutative
diagram
\[
\begin{CD}
\M_{X,x}@>{f^\flat=(\rho,\theta)}>> \RR_{\ge0}\times
U(1)\\
@V{\alpha_{X,x}}VV @VV{\alpha_\Pi}V\\
\O_{X,x}@>>{\ev_x}>\CC
\end{CD}
\]
Here $\alpha_{X,x}$ is the stalk of the structure morphism
$\alpha_X:\M_X\to\O_X$ of $X$ and $\ev_x$ takes the value of a
holomorphic function at the point $x$. This diagram implies that the
first component $\rho$ of $f^\flat$ is determined by $x$ and the structure
homomorphism by \ChB{the equation}
\[
\rho(s)= \big|\big(\alpha_{X,x}(s)\big)(x)\big|.
\]
Thus giving $f$ is equivalent to selecting the point $x\in X$ and a
homomorphism $\theta:\M_{X,x} \to U(1)$ with the property that for
any $s\in \M_{X,x}$ it holds
\[
(\alpha_{X,x}(s))(x) = \big|(\alpha_{X,x}(s))(x)\big|\cdot
\theta(s).
\]
Since both sides vanish unless $s\in \O_{X,x}^\times\subset \M_{X,x}$, this
last property needs to be checked only on invertible elements. Note also that a
homomorphism $\M_{X,x}\to U(1)$ extends to $\M_{X,x}^\gp$ since $U(1)$ is an
abelian group. Summarizing, we have a canonical identification
\begin{equation}\label{Eqn: (x,theta) description of KN space}
(X,\M_X)^\KN= \Big\{(x,\theta)\in\textstyle \coprod_x
\Hom(\M^\gp_{X,x},U(1)) \,\Big|\, \forall h\in\O_{X,x}^\times:
\theta(h)=\frac{h(x)}{|h(x)|}\Big\}
\end{equation}
In this description we adopt the occasional abuse of notation of
viewing $\O_{X,x}^\times$ as a submonoid of $\M_{X,x}$ by means of
the structure homomorphism $\M_{X,x}\to \O_{X,x}$. \ChB{Given $s\in\M_{X,x}$ then by} \eqref{Eqn: (x,theta) description of KN space} any point
$f\in X^\KN$ over $x\in X$ defines an element $\theta(s)\in U(1)$. We
refer to this element of $U(1)$ as the \emph{phase} of $s$ at $f$. If
$s\in\O^\times_{X,x}$ then the phase of any point of $X^\KN$ over
$x$ agrees with $\Arg(s)= e^{i\arg(s)}$.
\end{remark}
Next we give an explicit description of $(X,\M_X)^\KN$ assuming the log
structure has a global chart with a fine (i.e. finitely generated and integral) monoid. For a fine and saturated monoid
$P$, we have $P^\gp\simeq T\oplus\ZZ^r$ with $T$ finite. Thus the set
$\Hom(P^\gp, U(1))$ is in bijection with $|T|$ copies of the real torus $U(1)^r$
by means of choosing \ChB{a basis of $\ZZ^r$.} This identification is compatible with the
topology on $\Hom(P^\gp, U(1))$ defined by the subbasis of topology consisting
of the sets
\begin{equation}\label{Eqn: V_p}
V_p:=\big\{ \varphi\in\Hom(P^\gp,U(1))\,\big|\, \varphi(p)\in V\big\},
\end{equation}
for all $V\subset U(1)$ open and $p\in P^\gp$.

\begin{proposition}
\label{Prop: KN space from chart}
Let $P$ be a fine monoid and  let $X$ be an analytic space endowed
with the log structure defined by a holomorphic map $g: X\to\Specan\CC[P]$.
Then there is a canonical closed \ChB{topological} embedding of
$(X,\M_X)^\KN$ into $X\times \Hom(P^\gp, U(1))$ with image
\[
\Big\{ (x,\lambda)\in X\times \Hom(P^\gp,U(1))\,\Big|\,
\forall p\in P:\, g^\sharp(z^p)_x\in \O^\times_{X,x} \Rightarrow
\lambda(p)=\Arg(g^\sharp(z^p)(x)) \Big\}.
\]
\end{proposition}

\begin{proof}
Denote by $\beta: P\to \Gamma(X,\M_X)$ the chart given by $g$ and by
$\beta_x: P\to \M_{X,x}$ the induced map to the stalk at $x\in X$. 
Recall the description \eqref{Eqn: (x,theta) description of KN
space} of $(X,\M_X)^\KN$ by pairs $(x,\theta)$ with $x\in X$ and
$\theta:\M_{X,x}^\gp\to U(1)$ a group homomorphism extending
$h\mapsto h(x)/|h(x)|$ for $h\in \O_{X,x}^\times$. With this
description, the canonical map in the statement is
\[
\Psi: (X,\M_X)^\KN\lra X\times \Hom(P^\gp, U(1)),\quad
(x,\theta)\longmapsto \big(x, \theta\circ\beta^\gp_x \big).
\]
Here $\beta_x^\gp: P^\gp\to \M_{X,x}^\gp$ is the map induced by
$\beta_x$ on the associated groups.

To prove continuity of $\Psi$, let $p\in P^\gp$ and $V\subset U(1)$ be open.
Then $\Psi^{-1}(X\times V_p)$ with $V_p\subset \Hom(P^\gp,U(1))$ the basic
open set from \eqref{Eqn: V_p}, equals $\ev_{\beta^\gp(p)}^{-1}(\RR_{\ge
0}\times V)$, with $\ev_{\beta^\gp(p)}$ defined in \eqref{Eqn: ev_s}. By the
definition of the topology, $\Psi^{-1}(X\times V_p)\subset (X,\M_X)^\KN$ is thus
open. Continuity of the first factor $\pi$ of $\Psi$ being trivial, this shows
that $\Psi$ is continuous.

We next check that $\im(\Psi)$ is contained in the closed subset of
$X\times\Hom(P^\gp, U(1))$ stated in the assertion. Let $(x,\theta)\in
(X,\M_X)^\KN$. For $p\in P$, the required equation $g^\sharp(z^p)(x)=
\lambda(p)\cdot \big|g^\sharp(z^p)(x)\big|$ for $\lambda=
\theta\circ\beta_x^\gp$ is non-trivial only if
$h:=g^\sharp(z^p)\in\O^\times_{X,x}$. In this case, $\beta_x(p)$ maps to $h$
under the structure homomorphism $\M_{X,x}\to \O_{X,x}$ and hence
\[
\big(\theta\circ\beta_x(p)\big)(x)
=\frac{h(x)}{\big|h(x) \big|} = \frac{g^\sharp(z^p)(x)}{\big|
g^\sharp(z^p)(x)\big|},
\]
verifying the required equality.

Conversely, assume $(x,\lambda)\in X\times\Hom(P^\gp, U(1))$ fulfills
\begin{equation}\label{Eqn: condition for lambda}
g^\sharp(z^p)(x)= \lambda(p)\cdot \big|g^\sharp(z^p)(x)\big|,
\end{equation}
for all $p\in P$. Denote by $\alpha: \M_X\to \O_X$ the structure
homomorphism. Then $\M_{X,x}$ fits into the cocartesian diagram of monoids
\[
\xymatrix{
\beta_x^{-1}(\O_{X,x}^{\times}) \ar@{^{(}->}[r] \ar[d]_{\alpha_x\circ \beta_x}
& P \ar[d]^{\beta_x} \\
\O_{X,x}^{\times} \ar[r]  & \shM_{X,x}}.
\]
By the universal property of the fibered coproduct there exists a unique map $\phi \in \Hom(\shM_{X,x},S^1)$, with $\phi(h)=\frac{h(x)}{|h(x)|}$ for $h \in \mathcal{O}_{X,x}^{\times}$, fitting into the following commutative diagram:
\[
\xymatrix@C=30pt
{\beta^{-1}(\mathcal{O}_{X,x}^{\times}) \ar@{^{(}->}[r] \ar[d]^{\alpha_x \circ \beta_x}& P\ar[d]\ar@/^/[ddr]^{\lambda}&\\
\mathcal{O}_{X,x}^{\times} \ar[r]\ar@/_/[drr]^{\Arg \circ \operatorname{ev}_x}&\shM_{X,x}
\ar@{-->}[rd]^{\phi}&\\
&& U(1)
}
\]
The map $\phi$ defines a unique element in $\theta:\M_{X,x}^\gp\to U(1)$,
as the induced map on associated
groups. For $h\in \O_{X,x}^\times$ it holds $\theta(h)= h(x)/|h(x)|$
and hence $(x,\theta)\in (X,\M_X)^\KN$. It is now not hard to see
that the map $(x,\lambda)\mapsto (x,\theta)$ is inverse to $\Psi$
and continuous as well.
\end{proof}


\subsection{Examples of Kato-Nakayama spaces}

We \ChB{now} discuss a few examples of Kato-Nakayama spaces, geared
toward toric degenerations. Unless otherwise stated, $N$ denotes a
finitely generated free abelian group, $M=\Hom(N,\ZZ)$ its dual and
$N_\RR$, $M_\RR$ are the associated real vector spaces. If
$\sigma\subset N_\RR$ is a cone then the set of monoid homomorphisms
$\sigma^\vee= \Hom(\sigma,\RR_{\ge0}) \subset M_\RR$ denotes its
dual cone. A \emph{lattice polyhedron} is the intersection of
rational half-spaces in $M_\RR$ with an integral point on each
minimal face.

The basic example is a canonical description of the Kato-Nakayama space of a
toric variety defined by a momentum polytope. We use a rather liberal definition
of a momentum map \ChB{in Definition~\ref{Def: Momentum map} below that makes}
no reference to a symplectic structure.

\ChB{We first recall the notion of \emph{positive real locus} of a toric
variety, see e.g.\ \cite[\S4.1]{Fulton}.} Let $X$ be the complex projective toric
variety associated to a full-dimensional, convex lattice polyhedron $\Xi\subset
M_\RR$. A basic fact of toric geometry states that the fan of $X$ agrees with
the normal fan $\Sigma_\Xi$ of $\Xi$. From this description, $X$ is covered by
affine toric varieties $\Specan (\CC[\sigma^\vee\cap M])$, for $\sigma\in
\Sigma_\Xi$. Since the patching is monomial, it preserves the real structure of
each affine patch. Hence the real locus $\Hom(\sigma^\vee,\RR)\subset
\Hom(\sigma^\vee,\CC)$ of each affine patch glues to the real locus
$X_\RR\subset X$. Unlike in the definition of $\sigma^\vee$, here $\RR$ and
$\CC$ are multiplicative monoids. Moreover, inside the real locus of each affine
patch there is the distinguished subset
\[
\sigma=\Hom(\sigma^\vee,\RR_{\ge 0})\subset\Hom(\sigma^\vee,\RR),
\]
with ``$\Hom$'' referring to homomorphisms of monoid. These also
patch via monomial maps to give the \emph{positive real locus}
$X_{\ge0}\subset X_\RR$.

\ChB{With the positive real locus defined,} we are now in position to define abstract
momentum maps.

\begin{definition}
\label{Def: Momentum map}
Let $X$ be the complex toric variety defined by a full-dimensional lattice
polyhedron $\Xi\subset M_\RR$. Then a continuous map
\[
\mu: X\lra \Xi
\]
is called an \emph{(abstract) momentum map} if the following holds.
\begin{enumerate}
\item
$\mu$ is invariant under the action of $\Hom(M,U(1))$ on $X$.
\item
The restriction
of $\mu$ maps $X_{\ge0}$ homeomorphically \ChB{onto} $\Xi$, thus defining a
section $s_0: \Xi\to X$ of $\mu$ with image $X_{\ge0}$.
\item
The map
\begin{equation}
\label{Eqn: torus times section}
\Hom(M,U(1))\times \Xi\lra X,\quad
(\lambda,x)\longmapsto \lambda\cdot s_0(x)
\end{equation}
induces a homeomorphism $\Hom(M,U(1))\times \Int(\Xi) \simeq
X\setminus D$, where $D\subset X$ is the toric boundary
divisor.
\end{enumerate}
\end{definition}

Projective toric varieties have a momentum map, see e.g.\ \cite[\S4.2]{Fulton}.
For an affine toric variety $\Specan(\CC[P])$, momentum maps also exist. One
natural construction discussed in detail in \cite[\S1]{NO}, is a simple formula
in terms of generators of the toric monoid $P$ \cite[Def.1.2 and Thm.1.4]{NO}.
Some work is however needed to show that if
$P=\sigma^\vee\cap M$, then the image of this momentum map is the cone
$\sigma^\vee$ spanned by $P$. We give here another, easier but somewhat ad hoc
construction of a momentum map in the affine case.
\medskip

Our next result concerns the announced canonical description of the
Kato-Nakayama space of a toric variety with a momentum map.

\begin{proposition}
\label{Prop: KN of toric variety}
Let $X$ be a complex toric variety with a momentum map
$\mu: X\to \Xi\subset M_\RR$ and let $\M_X$ be the toric log
structure on $X$. Then the map \eqref{Eqn: torus times section}
factors through a canonical homeomorphism
\[
\Phi: \Xi\times \Hom(M, U(1))\lra (X,\M_X)^\KN.
\]
\end{proposition}

\begin{proof}
The toric variety $X$ is covered by open affine sets of the form $\Specan(\CC[P])$
with $P^\gp= M$, and these define charts for the log
structure. Thus the local description of $X^\KN$ in Proposition~\ref{Prop: KN
space from chart} as a closed subset globalizes to define a closed \ChB{topological}
embedding
\[
\iota: X^\KN\lra X\times \Hom(M, U(1)).
\]
With $s_0:\Xi\to X$ the section of $\mu$ with image $X_{\ge0}\subset X$,
consider the continuous map
\begin{eqnarray}
\nonumber
\Phi: \Xi\times \Hom(M,U(1)) & \lra & X\times\Hom(M,U(1)) \\
\nonumber
(a,\lambda) & \longmapsto & \big(\lambda\cdot s_0(a), \lambda )
\nonumber
\end{eqnarray}
Here $\lambda\in \Hom(M, U(1))$ acts on $X$ as an element of the
algebraic torus $\Hom(M,\CC^\times)$. The map $\Phi$ has the
continuous left-inverse \ChB{$\mu\times\id$.} Thus to finish the proof it
remains to show $\im(\Phi)=\im(\iota)$. \ChB{We may check this equation on
$U\times\Hom(M,U(1))$ for $U=\Specan \CC[\sigma^\vee\cap M]$ toric affine open
subsets of $X$. Note that $\Phi^{-1}\big(U\times \Hom(M,U(1))\big)$ is the
product with $\Hom(M,U(1))$ of the open dense subset $\mu(U)\subset \Xi$.}

Indeed, according to Proposition~\ref{Prop: KN space from chart},
$(x,\lambda)\in X\times \Hom(M, U(1))$ lies in $\iota(X^\KN)$ iff for all \ChB{$m\in
\sigma^\vee\cap M$} it holds $z^m(x)=\lambda(m)\cdot |z^m(x)|$. But
this equation holds if and only if $x=\lambda \cdot s_0(a)$ for $a=\mu(x)$
since $ s_0(a)\in X_{\ge0}$ implies
\[
z^m(\lambda\cdot s_0(a)) =
\lambda(m) \cdot z^m( s_0(a)) =
\lambda(m)\cdot|z^m(x)|.
\]
Thus $(x,\lambda)\in X^\KN$ iff $(x,\lambda) = (\lambda\cdot s_0(\mu(x)),
\lambda)$, that is, iff $(x,\lambda)\in \im(\Phi)$.
\end{proof}

\begin{remark}
The left-hand side in the statement of Proposition~\ref{Prop: KN of
toric variety} can also be written $T^*_\Xi/\check\Lambda$ where
$\check\Lambda\subset T_\Xi^*$ is the local system of integral cotangent
vectors. Indeed, for any $y\in \Xi$ we have the sequence of
canonical isomorphisms
\[
T^*_{\Xi,y}/ \check\Lambda_y \lra \Hom(M,\RR)/\Hom(M,\ZZ)=
\Hom(M,\RR/\ZZ)= \Hom(M, U(1)).
\]
\end{remark}

\begin{example}
Let $X= \AA^1= \Specan(\CC[\NN])$ be endowed with the divisorial log structure
$\shM_{(X,\{0\})}$. \ChB{For $z$ the toric coordinate, $|z|$} defines a momentum
map $\mu: X\to \RR_{\ge0}$ \ChB{in the sense of Definition~\ref{Def: Momentum
map}}. By Proposition~\ref{Prop: KN of toric variety}, \ChB{there exists a
canonical homeomorphism}
\[
X^{\mathrm{KN}} \cong \RR_{\geq 0} \times S^1.
\]
The map $\pi: X^{\mathrm{KN}}\to X$ is a homeomorphism onto the image over
$\AA^1\setminus\{0\}$ and has fibre $S^1=\Hom(\M_{(X,\{0\})},U(1)) = \Hom(\NN,
U(1))$ over $0$. Thus $X^{\mathrm{KN}}$ is homeomorphic to the \emph{oriented
real blow up} of $\AA^1$ at $0$ \cite{Gillam}.

\begin{figure}
\center{\input{BlowUp.pspdftex}}
\caption{An oriented real blow-up of $\AA^1$ at the origin}
\label{Fig: KN A1}
\end{figure}

More generally, let $(X,\M_{(X,D)})$ be the divisorial log structure on a
complex scheme $X$ with a normal crossings divisor $D\subset X$. Then the
Kato-Nakayama space $X^{\mathrm{KN}}$ of $X$ can be identified with the oriented
real blow up of $X$ along $D$. At a point $x\in X$ the map $X^\KN\to X$ has
fibre $(S^1)^k$ with $k$ the number of irreducible components of $D$ containing
$x$. For details of this construction we refer to \cite{Gillam}.
\end{example}

\begin{example}
\label{P2 divisorial}

Let $X= \mathbb{P}^2$ with the toric log structure. There exists a momentum map
$\mu:\PP^2 \rightarrow \Xi$ with
\[
\Xi=\conv\{(0,0),(1,0),(0,1)\} \subset M_\RR
\]
the $2$-simplex and $M=\ZZ^2$. The momentum map exhibits the algebraic torus
$(\CC^{\times})^2 \subset \PP^2$ as a trivial $(S^1)^2$-bundle over $\Int\Xi$.
Intrinsically, the $2$-torus fibres of $\mu$ over $\Int(\Xi)$ are
$\Hom(M,U(1))$. Over a face $\tau\subset\Xi$, the $2$-torus fibre collapses via
the quotient map given by restriction,
\[
\Hom(M,U(1))\lra \Hom(M\cap T_\tau,U(1)),
\]
where $T_\tau\subset M_\RR$ is the tangent space of $\tau$.
The quotient yields an $S^1$ over the
interior of an edge of $\Xi$ and a point over a vertex.

Now going over to the Kato-Nakayama space simply restores the
collapsed directions, thus yielding the trivial product
\[
X^\KN=\Xi\times(S^1)^2.
\]
The fibre of $X^\KN\to X$ over the interior of a toric
stratum given by the face $\tau\subset\Xi$ are the fibres of
$\Hom(M,U(1))\to \Hom(M\cap T_\tau,U(1))$.

An analogous discussion holds for all toric varieties with a
momentum map.
\end{example}

We finish this section with an instructive non-toric example that
features a non-fine log structure. It discusses the most simple
non-toric example of a toric degeneration, the subject of Section~\ref{Sect:
Toric degenerations}.

\begin{example}
\label{Expl: Focus-focus singularity I}
Let $X= \Specan \big(\CC[x,y,w^{\pm1},t]/(xy-t(w+1))\big)$, considered as a holomorphic
family of complex surfaces
\[
\delta: X\to\CC
\]
via projection by $t$. For fixed
$t\neq0$ we can eliminate $w$, which is invertible, to arrive at $\delta^{-1}(t)\simeq \CC^2 \setminus \{ xy=0\}$. For
$t=0$ we have
\[
\delta^{-1}(0)= (\CC\times\CC^*)\amalg_{\CC^*} (\CC\times\CC^*),
\]
two copies of $\CC\times\CC^*$ with coordinates $x,w$ and
$y,w$, respectively, glued seminormally along $\{0\}\times\CC^*$. Denote
$X_0=\delta^{-1}(0)$, let $\M_X=\M_{(X,X_0)}$ be the log structure defined by
the family and $\M_{X_0}$ its restriction to the fibre over $0$. Then
$(X_0,\M_{X_0})$ comes with a log morphism $f$ to the standard log point $O^\ls=
(\mathrm{pt}, \CC^\times\oplus\NN)$. We want to discuss the Kato-Nakayama space
of $(X_0,\M_{X_0})$ together with the map to $S^1$, the Kato-Nakayama space of
$O^\ls$.

First note that $X$ has an $A_1$-singularity at the point $p_0$ with coordinates
$x=y=t=0$, $w=-1$. Any Cartier divisor at $p_0$ with support contained in $X_0$
is defined by a power of $t$. Hence $\ol\M_{X_0,p_0}=\NN$, while at a general
point $p$ of the double locus $(X_0)_\sing\simeq\CC^*$, the central fibre is a
normal crossings divisor in a smooth space and hence $\ol\M_{X_0,p}=\NN^2$. In
particular, $\ol\M_{X_0}$ is not a fine sheaf at $p_0$. On the other hand
$(X_0,\M_{X_0})$ is a typical example of Ogus' notion of \emph{relative
coherence}. In this category, the main result of \cite{NO} still says that
$X^\KN$ is homeomorphic relative $(\CC,\M_\CC)^\KN = S^1\times\RR_{\ge0}$ to
$X_0^\KN\times\RR_{\ge0}$. In particular, the fibre of $f^\KN: X_0^\KN\to
S^1=(O^\ls)^\KN$ over $e^{i\phi}\in S^1$ is homeomorphic to $\CC^2 \setminus \{ xy=0\}$. We want to
verify this statement explicitly.

As a matter of notation, we write $s_x,s_y,s_t$ for the sections
of $\M_X$ or of $\M_{X_0}$ defined by the monomial functions
indicated by the subscripts. We also use $s_t$ to denote the
generator of the log structure $\M_{O^\ls}$ of $O^\ls$.

Since $s_t$ generates $\M_{X_0,p_0}$ as a log structure,
according to \eqref{Eqn: (x,theta) description of KN space} the
fibre of $\pi:X_0^\KN\to X_0$ over $p_0$ is a copy of $U(1)$, by
mapping $\theta\in \Hom(\M_{X_0,p_0}^\gp,U(1))$ to its value on
$s_t$. The projection to $S^1=(O^\ls)^\KN$ can then be viewed
as the identity.

On the complement of $p_0$ the log structure is fine, but there is
no global chart. We rather need two charts, \ChB{which we take to be}
\begin{eqnarray*}
U&=&X_0\setminus (x=y=0) = \Specan \big(\CC[x^{\pm1}, w^{\pm 1}]\times
\CC[y^{\pm1},w^{\pm 1}]\big)\simeq (\CC^*)^2\amalg (\CC^*)^2\\
V&=&X_0\setminus(w=-1) =
\Specan\big(\CC[x,y,w^{\pm1}]/(xy)\big)_{w+1},
\end{eqnarray*}
respectively. The charts are as follows:
\begin{eqnarray*}
\varphi: \NN&\lra&  \Gamma(U,\M_{X_0}),\quad
\varphi(1)=s_t.\\
\psi: \NN^2&\lra& \Gamma(V,\M_{X_0}),\quad
\varphi(a,b)= s_x^a\cdot s_y^b.
\end{eqnarray*}
Proposition~\ref{Prop: KN space from chart} now exhibits $U^\KN$,
$V^\KN$ as closed subsets of $U\times U(1)$ and $V\times U(1)^2$,
respectively. In each case, the projections to the $U(1)$-factors are
defined by evaluation of $\theta\in \Hom(\M_{X_0,x}^\gp, U(1))$ on
monomials. We write these $U(1)$-valued functions defined on open
subsets of $X_0^\KN$ by $\theta_t,\theta_x,\theta_y,\theta_w$
according to the corresponding monomial. Since $f: (X_0,\M_{X_0})\to
O^\ls$ is strict over $U$, we have $U^\KN= U\times U(1)$ with
$f^\KN= \theta_t$ the projection to $U(1)$. For $V^\KN$, over the
double locus $x=y=0$ the fibre of the projection $V^\KN\to V$ is all
of $U(1)^2$, while for $x\neq 0$ the value of $\theta_x$ is
determined by $\arg x$. An analogous statement holds for $y\neq0$.

To patch the descriptions of $X_0^\KN$ over the two charts amounts
to understanding the map $V^\KN\to (O^\ls)^\KN= U(1)$, the image
telling the value of $\theta\in \Hom(\M_{X_0,x}^\gp,U(1))$ on $s_t$.
Over $V=X_0\setminus(w=-1)$ we have the equation $s_t= (w+1)^{-1}
s_xs_y$. Thus, say over $x\neq 0$, \ChB{we obtain a} description of $V^\KN$
by the value $\theta_y$ of $\theta\in\Hom(\M_{X_0,x}^\gp,U(1))$ on $s_y$. Then 
\begin{equation}
\label{Eqn: Twisted gluing A_1}
\theta_t= \frac{\Arg(x)}{\Arg(w+1)}\cdot\theta_y.
\end{equation}
Thus the identification with $U^\KN$ is twisted both by the phases
of $x$ and of $w+1$. A similar description holds for $y\neq0$.

For $t=\tau e^{i\phi}\neq 0$ denote by
$X_0^\KN(e^{i\phi})$ the fibre over $e^{i\phi}\in U(1)
=(O^\ls)^\KN$ and similarly $U^\KN(e^{i\phi})$,
$V^\KN(e^{i\phi})$. It is now not hard to construct a
homeomorphism between $U^\KN(e^{i\phi})\cup
V^\KN(e^{i\phi})$ and $\ChB{\delta^{-1}(t)}\simeq
\CC^2 \setminus \{ xy=0\}$. For example, there exists a unique such
homeomorphism that on $(x=0)\subset U^\KN(e^{i\phi})$ restricts
to
\[
(\CC^*)^2\ni (y=se^{i\psi},w)\longmapsto
\left( \frac{(w+1)\cdot\tau e^{i(\phi-\psi)}}{
s+|(w+1)\tau|^{1/2}},  (s+|(w+1)\tau|^{1/2}) e^{i\psi} \right) \in
\CC^2,
\]
and to a similar map with the roles of $x$ and $y$ swapped on $(y=0)\subset
U^\KN(e^{i\phi})$. This form of the homeomorphism comes from considering the
degeneration $xy=(w+1)t$ as a family of normal crossing degenerations of curves
parametrized by $w=\textrm{const}$. Details are left to the reader.
\end{example}

\begin{example}
\label{Expl: Focus-focus singularity II}
An alternative and possibly more useful way to discuss the
Kato-Nakayama space of the degeneration $xy=(w+1)t$ in
Example~\ref{Expl: Focus-focus singularity I}, is in terms of closed
strata and the momentum maps of the irreducible components
$Y_1=(y=0)$, $Y_2=(x=0)$ of $X_0$ and of their intersection
$Z=Y_1\cap Y_2$. Endow $Y_1,Y_2, Z$ with the log structures making
the inclusions into $X_0$ strict. Away from $p_0$ we then have
global charts defined by $s_t,s_x$ for $Y_1$, by $s_t,s_y$ for $Y_2$
and by $s_x,s_y$ for $Z$. By functoriality, the fibre of $\pi:
X_0^\KN\to X_0$ over these closed strata $Y_1$, $Y_2$, $Z$ agrees
with $Y_1^\KN$, $Y_2^\KN$, $Z^\KN$, respectively. Therefore, we can
compute $X_0^\KN$ as the fibred sum
\[
X_0^\KN= Y_1^\KN\amalg_{Z^\KN} Y_2^\KN.
\]
Away from the singular point $p_0\in X_0$ of the log structure, $Y_1^\KN$ is the
Kato-Nakayama space of $Y_1$ as a toric variety times an additional $S^1$-factor
coming from $s_t$, and similarly for $Y_2$. Since each $Y_i$ has a momentum map
$\mu_i$ with image the half-plane $\RR_{\ge 0}\times\RR$, Proposition~\ref{Prop:
KN of toric variety} gives a description of $Y_i^\KN$ as
$\RR_{\ge0}\times\RR\times U(1)^3/\sim$ with the $U(1)$-factors telling the
phases of $w$, $s_t$ and of $s_x$ (for $i=1$) or of $s_y$ (for $i=2$),
respectively. We assume that the momentum map maps $p_0$ to $(0,0)\in
\RR_{\ge0}\times\RR$. Note that $w\neq 0$, so the phase of $w$ is already
determined uniquely at any point of $X_0$. The indicated quotient takes care of
the special point $p_0$ by collapsing a $U(1)$ over $(x,y,w)=(0,0,-1)$ as
follows. Restricting the projection
\[
\RR_{\ge0}\times\RR\times U(1)^3\lra Y_1
\]
to $x=0$ yields a $U(1)^2$-bundle over $\CC^*$, the $w$-plane. The two
$U(1)$-factors record the phases of $s_t$ and $s_x$, respectively. Now the
quotient collapses the second $U(1)$-factor over $w=-1$, reflecting the fact
that only $s_t$ survives in $\M_{X_0,p_0}$.

Again by functoriality, the restriction of either $Y_i^\KN$ to $\{0\}\times\RR$
yields $Z^\KN$. Using $s_x,s_y$ as generators for $\M_Z$ over
$Z\setminus\{p_0\}= \CC^*\setminus\{-1\}$ we see
\[
Z^\KN\ =\ \RR\times U(1)^3/\sim\ =\ \CC^*\times U(1)^2/\sim
\]
Now the three $U(1)$-factors tell the phases $\theta_w$, $\theta_x$, $\theta_y$
of $w$, $s_x$, $s_y$. In the description as $\CC^*\times U(1)^2/\sim$,
the equivalence relation collapses the $U(1)$-subgroup
\[
\big\{(\theta_x,\theta_y)\in U(1)^2\,\big|\,
\theta_x\cdot \theta_y=1 \big\}
\]
over $-1\in \CC^*$. Thus over the circle $|w|=a$ inside the double locus
$x=y=0$, $X_0^\KN$ is a trivial $U(1)^2$-bundle as long as $a\neq 1$, hence a
$3$-torus. This $3$-torus fibres as a trivial bundle of $2$-tori over
$(O^\ls)^\KN=S^1$. If $a=1$, one of the $U(1)$-factors collapses to a point over
$w=-1$, leading to a trivial family of pinched $2$-tori over $(O^\ls)^\KN=S^1$.

A nontrival torus fibration arises if we consider a neighbourhood of the double
locus. This is most easily understood by viewing $X_0^\KN$ as a torus fibration
over $\RR^2$ by taking the union of the momentum maps
\[
\mu: X_0\lra \RR^2,\quad
\mu|_{Y_1}=\mu_1,\quad
\mu|_{Y_2}=\kappa\circ\mu_2,
\]
with $\kappa(a,b)=(-a,b)$. Denote by $X_0^\KN(e^{i\phi})$ the fibre of $X_0^\KN\to
(O^\ls)^\KN= S^1$ over $e^{i\phi}\in S^1$. Write $\mu^\KN= \ChB{\mu\circ\pi}: X_0^\KN\to
\RR^2$ and $\mu^\KN(e^{i\phi})$ for the restriction to $X_0^\KN(e^{i\phi})$. For any
$(a,b)\in \RR^2\setminus \{(0,0)\}$ the fibre $(\mu^\KN)^{-1}(a,b)$ is a
$3$-torus trivially fibred by $2$-tori over $(O^\ls)^\KN=S^1$. We also have
trivial torus bundles over the half-spaces
$(\RR_{\ge0}\times\RR)\setminus\{(0,0)\}$ and
$(\RR_{\le0}\times\RR)\setminus\{(0,0)\}$ as well as over
$\RR\times(\RR\setminus\{0\})$. However, the torus bundle is non-trivial over
any loop about $(0,0)\in \RR^2$. The reason is that the equation $xy=t(w+1)$ gives the
identification of torus fibrations over the two half planes via
\begin{equation}\label{Eqn: Patching of KN for focus-focus}
\theta_y= \theta_x^{-1}\cdot\Arg(w+1)\cdot\theta_t. 
\end{equation}
Now consider $\mu_0^{\KN}(e^{i\phi}): X_0^{\KN}(e^{i\phi})\to \RR^2$
over a small circle about the origin in $\RR^2$, yielding a fibration by 2-tori
over $S^1$. The previous descriptions give trivializations of
$\mu_0^{\KN}(e^{i\phi})$ over the right half-plane by $\theta_w, \theta_x$ and
over the left half-plane by $\theta_w, \theta_y$, yielding a half-circle times
$(S^1)^2$ each. These trivialization are identified at their boundaries, which
map to two points on the vertical axis $\{0\}\times\RR\subset\RR^2$, by means of
\eqref{Eqn: Patching of KN for focus-focus}. The two points on the vertical axis
have coordinates $(0,\pm b)$, corresponding to circles $|w|=r$ and $|w|=R$ with
$r<1<R$ in the $w$-plane. Now $\Arg(w+1)$ restricted to the circle $|w|=r<1$ is
homotopic to a constant map, while for $|w|=R>1$ this restriction has winding
number $1$. This means that the topological monodromy of the $2$-torus
fibration $\mu^\KN(e^{i\phi}): X_0^\KN(e^{i\phi})\to \RR^2$ along a
counterclockwise loop about $(0,0)\in\RR^2$ is a (negative) Dehn-twist. Thus
$\mu^\KN(e^{i\phi})$ is homeomorphic to a neighbourhood of an $I_1$-singular
fibre (a nodal elliptic curve) of an elliptic fibration of complex surfaces.
\end{example}


\section{Toric degenerations from gluing data}
\label{Sect: Toric degenerations}


\subsection{Toric degenerations and their intersection complex}
\label{Subsect: Toric log Calabi-Yau spaces}
We now focus on toric degenerations with reducible central fiber, as introduced
in \cite[Def.4.1]{logmirror1}, but following the conventions of
\cite{theta} and interpreted in the complex-analytic category.\footnote{See
\cite[\S A.1 and \S A.2]{theta} for the translation to the earlier conventions of
\cite{logmirror1} and \cite{affinecomplex}.} For much of the present paper we
take the point of view that a toric degeneration is given, thus avoiding to
discuss the canonical reconstruction procedure involving wall structures
\cite{logmirror1},\cite{GHK},\cite{theta}. To further simplify the presentation,
we also restrict to the original case of one-dimensional families.

To define toric degenerations, we start with their central fibers. For this
purpose, we say a map $\iota: X\to Y$ of toric varieties is a \emph{stratified
embedding} if it induces isomorphisms of each closed toric stratum of $X$ with a
closed toric stratum of $Y$. For an \emph{integral polytope} $\tau$, the convex
hull of finitely many integral points in $\RR^n$ for some $n$, we write
$\Lambda_\tau\simeq\ZZ^{\dim\tau}$ for the free abelian group of integral vector
fields on $\tau$, or equivalently, for the group of integral tangent vectors at
any point of $\tau$.

\begin{definition}
\label{Def: broken toric variety}
a) Let $B$ be a topological manifold that is the underlying topological space of
a cell complex $\P$ of integral polytopes with attaching maps inclusions of
faces preserving the integral affine structure. For each $\tau\in\P$ we assume
that the map $\tau\to B$ is injective, that is, that no cell of $\P$
self-intersects. We call $(B,\P)$ an integral polyhedral manifold and write
$\P_\max$ for the set of maximal ($n$-dimensional) cells of $\P$.

b) A \emph{toric log Calabi--Yau} is a reduced complex space $X_0$ of some pure
dimension $n$ with irreducible components $X_\sigma$ the projective toric
varieties with momentum polytopes the maximal cells $\sigma$ of an integral
polyhedral manifold $(B,\P)$. We refer to \cite[Defn.4.3]{logmirror1} for further technical assumptions on toric log Calabi--Yau spaces. For all $\tau \in\P$ we also assume that
\[
X_\tau:=\bigcap_{\sigma\in\P_\max,\, \tau\subseteq\sigma}  X_\sigma
\]
is isomorphic to the toric variety with momentum polytope $\tau$. We call $(B,\P)$ the \emph{intersection complex} of $X_0$ and $X_\tau$ its \emph{toric strata}.
\end{definition}

Thus for any $\tau\in\P$ and face $\omega\subseteq\tau$, we have a closed
embedding $\iota_{\tau\omega}:X_\omega\to X_\tau$. Since $\iota_{\tau\omega}$
preserves toric strata, it is the composition of a toric automorphism of
$X_\omega$ with the inclusion of a closed toric stratum into $X_\tau$. 

For the following definition recall that a normal complex space $\shX$ with a
divisor $X_0\subset\shX$ is called a \emph{toroidal pair} if locally
$(\shX,X_0)$ is isomorphic to $(Y,D)$ with $Y$ an affine toric variety and
$D\subset Y$ the union of toric prime divisors.

\begin{definition}
\label{Def: Toric degeneration}
A \emph{toric degeneration} is a proper and flat holomorphic map
\[
\delta:\shX\lra \DD
\]
of complex spaces, where $\DD\subset\CC$ a disk, $\shX$ is normal, and the central fiber $X_0$ a union of
toric varieties, glued pairwise torically along toric prime divisors. We also assume that there is a closed
analytic subset $\shZ\subset \shX$ with fibers over $\DD$ of codimension two
such that $(\shX,X_0)$ is a toroidal pair outside of $\shZ$. We also assume that
$Z= \shZ\cap X_0$ is contained in the singular locus $(X_0)_\sing$ of $X_0$.
\end{definition}

We endow $\shX$ and $\DD$ with their divisorial log structures $\M_\shX
=\O_\shX\cap \O^\times_{\shX\setminus X_0}$ and $\M_\DD$, and view a toric
degeneration as a morphism of log analytic spaces
\[
\delta: (\shX,\M_{\shX})\lra (\DD,\M_{\DD}).
\]

We are mostly concerned with the log structure $\M_{X_0}$ on $X_0$ obtained by
restriction of $\M_\shX$. The induced morphism
\[
(X_0,\shM_{X_0})\lra O^\ls
\]
to the standard log point $O^\ls$ is then also log smooth away from $Z$. We therefore refer to $Z$ as the \emph{log singular locus} of $X_0$.

There is the following simple description of $(X_0,\shM_{X_0})$ at smooth points of
$(X_0)_\sing$, including points of $Z$. Let $\sigma,\sigma'\in\P_\max$ be
maximal cells and $\rho=\sigma\cap\sigma'$ be of dimension $n-1$. Let $u,v$ be
local holomorphic functions on $\shX$ restricting to toric monomials on
$X_\sigma, X_{\sigma'}\subset X_0$, respectively, such that $X_0$ is locally
defined by $uv=0$. 

\begin{proposition}
\label{Prop: description of shX in codim=1}
There is a unique Laurent polynomial
\[f_\rho\in\CC[\Lambda_\rho]=\CC[z_1^{\pm1},\ldots,z_{n-1}^{\pm 1}]\] and unique
$\kappa\in\NN\setminus 0$ such that in a neighbourhood of the algebraic torus
$(\CC^*)^{n-1}\subset X_\rho$, the log morphism
\[
(X_0,\shM_{X_0})\to O^\ls
\]
is isomorphic to the logarithmic central fiber of the map
\[
\CC^3\times(\CC^*)^{n-1}\supset V(f,\kappa)\lra \CC,\quad (u,v,t,z)\longmapsto t
\]
with $V(f,\kappa)\subset\CC^3\times(\CC^*)^{n-1}$ defined by
\begin{equation}
\label{Eqn: standard equation}
uv=f_\rho(z) \cdot t^\kappa.
\end{equation}
\end{proposition}

\begin{proof}
This is a simple special case of \cite[Thm.3.22 and Prop.4.20]{logmirror1},
which says that considering $u|_{X_\sigma}$ and $v|_{X_{\sigma'}}$ fixed, the
restriction $f(z,0)$ of $f$ to $t=0$ and $\kappa$ uniquely determine the local
isomorphism class of the log structure $\M_{X_0}$ on $X_0$ together with the log
morphism to $O^\ls$.
\end{proof}

\begin{remark}
We emphasize that $f_\rho$ depends on the choice of $u$ and $v$. Indeed,
changing the monomial $u|_{v=0}\in\CC[\Lambda_\sigma]$ by multiplication with
$z^m\in\CC[\Lambda_\rho]\subset\CC[\Lambda_\sigma]$ changes $f_\rho$ to
$z^m\cdot f_\rho$, and similarly for changes of $v$. Canonical choices of
$f_\rho$ can be deduced, however, depending on the choice of a connected
component of $\Int\rho\cap (B\setminus\shA)$ with $\shA\subset B$ introduced in
Definition~\ref{Def: Degenerate momentum map}, as should be clear from the
discussion of \eqref{Def: m_x} below.
\end{remark}

A central role in the theory of toric degenerations is played by a canonical
extension of the integral affine structure (transition functions in
$\ZZ^n\rtimes\GL(n,\ZZ)$) already given on the interiors of maximal cells to a
larger subset of $B$. To define this affine structure, we now add the mild
assumption that the compatible polarizations on the irreducible components
$X_\sigma\subset X_0$ are the restrictions of an ample line bundle on $X_0$. In
particular, $X_0$ is now projective.

Under this projectivity assumption, \cite[Prop.2.1]{RS2} proves the existence of
a continuous map
\begin{equation}
\label{Eqn: mu:X_0->B}
\mu: X_0\lra B,
\end{equation}
with restriction to each irreducible component $X_\sigma\subset X_0$ an abstract
momentum map in the sense of Definition~\ref{Def: Momentum map}.

\begin{definition}
\label{Def: Degenerate momentum map}
We refer to $\mu$ as a \emph{degenerate momentum map} and write
\[
\mu^\KN=\mu\circ\pi: X_0^\KN\lra B
\]
for the composition of $\mu$ with $\pi: X_0^\KN\to X_0$. The image $\shA=\mu(Z)$ of the log singular locus $Z\subset X_0$ is referred to as \emph{amoeba image} of $Z$.
\end{definition}

Note that $\shA$ is contained in the $(n-1)$-skeleton of $\P$. Moreover,
given an $(n-1)$-cell $\rho\in\P$, the intersection of $Z$ with the algebraic
torus $(\CC^*)^{n-1}\cap X_\rho$ is an algebraic hypersurface, and there exists
a diffeomorphism $\RR^{n-1}\to\Int \rho$ mapping the amoeba of this
hypersurface to $\shA\cap\Int(\rho)$, thus motivating the name amoeba image.

We will now show that the integral affine structure on
$\bigcup_{\sigma\in\P_\max} \Int\sigma\subset B$ extends naturally to
$B\setminus\shA$. For the statement, recall that an integral tangent vector
$\xi$ on a momentum polyhedron $\sigma\subset\RR^n$ defines a monomial function
$z^\xi\in \CC[z_1,\ldots,z_n]$. Thus the monomial function $u$ in
Proposition~\ref{Prop: description of shX in codim=1} defines a primitive
integral vector field $\xi\in \Lambda_\sigma$ that generates the quotient
$\Lambda_\sigma/\Lambda_\rho\simeq\ZZ$ and points from $\rho$ into $\sigma$.
Similarly, $v$ defines $\xi'\in\Lambda_{\sigma'}$ pointing from $\rho$ into
$\sigma'$. Extending the affine structure on $\Int\sigma\cup\Int\sigma'$ across
a neighbourhood of $x\in\Int\rho$ amounts to define the parallel transport of
$\xi$ through $x$ as $-\xi'+m_x$ for some $m_x$ in the image of the inclusion
$\Lambda_\rho\to\Lambda_{\sigma'}$. Now if $x\in B\setminus\shA$, the restriction of $f$ to $\mu^{-1}(x)\subset (\CC^*)^{n-1}\subset X_\rho$ defines a continuous map
\begin{equation}
\label{Def: m_x}
\mu^{-1}(x)=\Hom\big(\Lambda_\rho, U(1)\big)\stackrel{f}{\lra} \CC^*\stackrel{\Arg}{\lra} U(1),
\end{equation}
hence an element of $\Lambda_\rho=\Hom\big(\Hom(\Lambda_\rho,U(1)),U(1)\big)$.
We define $m_x$ as the image of this element in $\Lambda_\sigma'$. It is not hard to see that this extension of the integral affine structure over $\bigcup_{\rho,\dim\rho=n-1} \Int\rho$ does not depend on any choices, see \cite[Constr.2.2 and Rem.2.3]{RS2}.

\begin{proposition}
\label{Prop: Affine structure on B minus shA}
There exists a unique integral affine structure on $B\setminus\shA$ satisfying the following conditions:
\begin{enumerate}
\item
The affine structure restricts to the given affine structure on $\Int\sigma$,
for any $\sigma\in\P_\max$.
\item
The extension over a point $x\in \Int\rho\cap(B\setminus\shA)$ for $\rho\in\P$
an $(n-1)$-cell is via \eqref{Def: m_x} as described.
\end{enumerate}
\end{proposition}

\begin{proof}
Uniqueness follows since if $A\subset \RR^n$ is a closed subset with
$\RR^n\setminus A$ connected, then by the identity theorem for affine functions
there is a unique extension of the standard affine structure on $\RR^n\setminus
A$ to $\RR^n$.

To construct an integral affine chart near a point $x\in B\setminus\shA$, let
$\tau\in\P$ be the minimal cell containing $x$, assumed of codimension $r\ge 1$.
Pick $z \in \mu^{-1}(x)$ and
denote by $V\subset X_0$ the complement of all closed toric strata in $X_0$ not
containing $z$, where $\mu$ is as in \eqref{Eqn: mu:X_0->B}. Thus $V$ is the smallest neighbourhood of $z$ that is a union of
torus orbits of the irreducible components of $X_0$. Write $P=\ol\M_{X_0,z}$ and
let $m_0\in P$ be the image of the generator of $\ol \M_{\O^\ls}$. From the fact
that $X_0$ is a toric log Calabi--Yau space it is shown in \cite[\S3.3]{logmirror1}
that there is a closed embedding
\begin{equation}
\label{Eqn: Closed embedding Phi}
\Phi:V\lra\Specan \big(\CC[P\oplus\Lambda_\tau]\big)
\end{equation}
with image the reduced union of toric prime divisors defined by $z^{m_0}=0$. Thus if $P_\sigma\subset P$ denotes the facet corresponding to
$\sigma\in\P_\max$ with $\tau\subset\sigma$, the map $\Phi$ restricts to an isomorphism of $X_\sigma\cap V$ with $\Specan \CC[P_\sigma\oplus\Lambda_\tau]$. In particular, $\Phi$ induces an isomorphism
\begin{equation}
\label{Eqn: Chart on K(sigma,tau)}
P^\gp_\sigma\oplus\Lambda_\tau\stackrel{\simeq}{\lra} \Lambda_\sigma 
\end{equation}
of free abelian groups, whose restriction to $\Lambda_\tau$ is the embedding
$\Lambda_\tau\to \Lambda_\sigma$ coming from the fact that $\tau$ is a face of
$\sigma$. Note that this isomorphism is unique up to adding the composition
of a homomorphism $P_\sigma^\gp\to\Lambda_\tau$ with the projection
$P_\sigma^\gp\oplus\Lambda_\tau\to P_\sigma^\gp$.

Note also that the embedding $P_\sigma^\gp\to P$ induces a canonical isomorphism $P_\sigma^\gp\simeq P^\gp/\ZZ\cdot m_0$. The induced isomorphism
\[
\Lambda_\sigma\lra P_\sigma^\gp\oplus\Lambda_\tau \lra
(P^\gp/\ZZ\cdot m_0)\oplus\Lambda_\tau
\]
now identifies the set of integral points of the cone $\RR_{\ge
0}\cdot(\sigma-\tau)\subset\Lambda_\sigma\otimes_\ZZ\RR$ with a subset of the
image of $P_\sigma \oplus\Lambda_\tau$. Varying $\sigma$, the images of the
facets $P_\sigma\subset P$ in $P^\gp/\ZZ\cdot m_0$ are the integral points of
the maximal cones of a complete fan in the associated real vector space
$(P^\gp/\ZZ\cdot m_0)\otimes_\ZZ\RR$. We thus see that any closed embedding
$\Phi$ as in \eqref{Eqn: Closed embedding Phi} defines an integral affine chart
along $\Int\tau$. Since $\Phi$ is not unique, we now need to restrict to a certain
choice of $\Phi$ depending on the point $x\in \Int\tau\cap
(B\setminus\shA)$ to define the affine chart near $x$ unambiguously.

To this end, observe that $\Phi$ is the chart for a log structure over $O^\ls$
on $V$, which, however can not be isomorphic to $\M_{X_0}|_V$ unless all
$f_\rho$ are monomial. The picture developed in \cite[\S3.3]{logmirror1} rather
shows that a chart for $\M_{X_0}$ on an open subset $U\subset V$ can always be
obtained as follows. For $p\in\partial P$ denote by $z_p=\Phi^*(z^p)$ and by
$A_p\subset V$ the closure of $V\setminus V(z^p)$, that is, the closure of the
subset where $z^p$ does not vanish. The definition of the sheaf $\shF$ in
\cite[p.263]{logmirror1} then shows that for any $p\in\partial P$ there exists
$h_p\in\O^\times(A_p)$ such that
\[
\CC[P]\lra \O(U), \quad z^p\longmapsto h_p\cdot z_p,
\]
is the chart for a log structure over $O^\ls$ isomorphic to $(U,\M_{X_0}|_U)$.

It follows from \eqref{Def: m_x} 
that the restriction of $\Arg(h_p)\in U(1)$ to $\mu^{-1}(x)$ defines a homomorphism
\[
\Hom(\Lambda_\tau,U(1))=\mu^{-1}(x)\lra U(1),
\]
hence an element $m_p\in \Lambda_\tau\subset \Lambda_\sigma$. After replacing
$h_p$ by $z^{-m_p}\cdot h_p$ and $\Phi$ by mapping $z^p$ to $z^{-m_p}z_p$, we
can assume that $m_p=0$ for all $p$. Note that $m_p=0$ is equivalent to the
statement that all maps
\[
h_p|_{\mu^{-1}(x)}: U(1)^r\equiv \Hom(\Lambda_\tau, U(1))\lra \CC^*
\]
are homotopic to constant maps. We say that $\Phi$ \emph{is adapted to $x$} if
this is the case, that is, if all $h_p|_{\mu^{-1}(x)}$ are homotopically
constant. Note that at any $x\in B\setminus\shA$ there is exactly one adapted
chart at $x$. We define the affine chart near $x$ as explained with respect to an adapted chart, generalizing the construction in codimension one.

Compatibility of affine charts follows since moving $x$ to a point $y$ in an
adjacent higher-dimensional toric strata is done by localization of $P$, thus
moving some direct summand of $P^\gp$ to $\Lambda_\tau$. By keeping $y$ close to
$x$, the restriction $h_p$ to $\mu^{-1}(y)$ differs from the restriction to
$\mu^{-1}(x)$ only by a small term, hence remains homotopically contractible.
This shows that the restriction of $\Phi$ to the corresponding neighbourhood of
$y$ is an adapted chart also at $y$. The corresponding affine charts are then
compatible by construction.
\end{proof}


\subsection{Faithful and simple toric degenerations}
\label{Subsect: Faithful and simple toric degenerations}

If the log singular locus $Z\subset X_0$ contains $0$-dimensional strata of
$X_0$, the information captured by the log space $(X_0,\M_{X_0})$ typically is
too limited to tell much about $\shX\to\DD$. This is indeed the situation in
\cite{GHK}, where one obtains a large number of topologically different toric
degenerations with the central fiber a union of affine planes $\AA^2$,
joined cyclically along their coordinate axes and with the same induced log
structure on $X_0$.

In contrast, if $Z=\emptyset$ then $(\shX,\M_{\shX})\to(\DD,\M_{\DD})$ is log smooth and
$(X_0,\M_{X_0})\to O^\ls$ captures the complete topology of
\begin{equation}
\label{Eqn: shX minus X_0 to DD}
\shX\setminus X_0\to \DD\setminus\{0\}.
\end{equation}
Indeed, \cite[Thm.0.3]{NO} implies that \eqref{Eqn: shX minus X_0 to DD} is topologically isomorphic to
\begin{equation*}
\label{Eqn: NO for log smooth case}
X_0^\KN\times(0,r)\lra (O^\ls)^\KN\times (0,r)= U(1)\times(0,r).
\end{equation*}
Here $r>0$ is the radius of $\DD$.

Unfortunately, toric degenerations with $Z=\emptyset$ cover only very restricted
cases such as when $\shX_t$ for $t\neq0$ is a toric or abelian variety. If
$Z\neq\emptyset$ one can sometimes find a small resolution $\tilde\shX\to\shX$
with centers in $Z$ making $\tilde\shX\to\DD$ log smooth and the central fiber
$\tilde X_0$ normal crossings. This is for example always possible in dimension
$2$. A certain tradeoff is that now the toric irreducible components of $X_0$
get replaced by certain modifications, complicating any explicit analysis of
$\tilde X_0^\KN$.

Quite a bit more can still be said if $Z$ at least does not contain zero-dimensional
toric strata of $X_0$, as considered in \cite{logmirror1}, \cite{affinecomplex}. To distinguish this situation from the general case, we introduce the following definition.

\begin{definition}
A toric degeneration $\shX\to\DD$ is \emph{faithful} if the log singular locus
$Z\subset X_0$ does not contain zero-dimensional toric strata.
\end{definition} 

By definition, a faithful toric degeneration $\shX\to\DD$ is locally near
zero-dimensional strata of $X_0$ isomorphic to a map of affine toric varieties
\begin{equation}
\label{Eqn: Toric local model near 0d strata}
\Specan \CC[P]\lra \CC
\end{equation}
with reduced central fiber. Here $P\subset\ZZ^{n+1}$ is a saturated and finitely
generated monoid without invertibles other than $0$, and the map to $\CC$ is
defined by a monomial $z^p\in \CC[P]$. A crucial observation in
\cite{logmirror1} is that in this situation there is a natural extension of the
integral affine structure on the interiors of the maximal cells
$\sigma\in\P_\max$ over neighbourhoods of vertices $v\in\P$ in $B$, as we now explain. At a vertex $v\in B$ we have $\mu^{-1}(v) = X_v$, a zero-dimensional toric
stratum, where $\mu$ is as in \eqref{Eqn: mu:X_0->B}. There is a complete toric fan $\Sigma_v$ which describes $\mathcal{X}_0$ at $X_v$ as a gluing of
affine toric varieties (for details see \cite[Lemma 4.11]{A}). In fact, the
monoid $P$ can be identified with 
\[
P\simeq \big\{(m,h)\in \ZZ^n\times\ZZ\,\big|\, h\ge \varphi_v(m)\big\},
\]
where $\varphi_v: \RR^n\lra \RR$ is a specified piecewise linear function. The choice of such a function is described in relation with a choice of polarization in \cite{affinecomplex}, whereas in \cite{theta} it
canonically occurs in the presented mirror construction. The fact that $X_0$ is reduced is equivalent to integrality of
the slopes of $\varphi_v$. 

Now a maximal cone $C\in\Sigma_v$ corresponds to a
maximal cell $\sigma\in\P$ containing $v$, and $\Spec[C\cap\ZZ^n]$ defines an
affine toric chart for $X_\sigma$ containing the zero-dimensional toric stratum
under consideration. This shows that the integral affine structure defined by
$P/\ZZ p \simeq\ZZ^n$ is compatible with the affine structures on the maximal
cells of $\P$ containing $v$.

We have thus defined compatible charts for an affine structure covering the open
stars of vertices and interiors of maximal cells inside the barycentric
subdivision $\tilde \P$ of $\P$. The complement is the codimension two cells
complex $\Delta\subset B$ covered by all codimension two cells of $\tilde\P$
disjoint from vertices and interiors of maximal cells. For example, in dimension~$2$, $\Delta$ is the union of barycenters of $1$-cells of $\P$, while in dimension~$3$ $\Delta$ is the graph with edges connecting the barycenter of codimension one cells $\sigma\in\P$ with the barycenters of $1$-cells $\omega\in\P$.

We summarize the discussion in the following proposition.

\begin{proposition}
If $\shX\to\DD$ is a faithful toric degeneration, there is a natural integral affine structure on $B\setminus\Delta$ for $\Delta\subset B$ the largest $(n-2)$-dimensional subcomplex of the barycentric subdivision of $\P$ disjoint from vertices and interiors of maximal cells of $\P$.
\end{proposition}

In this situation, the statement of the topological description \eqref{Eqn: shX
minus X_0 to DD} of $\shX\setminus X_0 \to \DD\setminus\{0\}$ still holds true
in one important special case, that of \emph{simple singularities} introduced in
\cite[Def.1.60]{logmirror1}. The original definition of this concept is
by a certain indecomposability condition of the local affine monodromy of the
affine structure about $\Delta$. We will not embark on a discussion from this
perspective and rather use the characterization from \cite{logmirror2} in terms
of toric local models, which we now introduce.

\begin{construction}
\label{Constr: Local model simple sings}
Let $\Sigma$ be a complete fan in $\ZZ^{n-q}$ and assume $\psi_0,\psi_1,\ldots, \psi_q$ are
convex integral piecewise linear functions on $\Sigma$ with Newton polytopes
\[\Delta_0,\ldots,\Delta_q\subseteq (\RR^{n-q})^*,\] that is,
\[
\psi_i(m)=-\inf\big\{\langle n,m\rangle \,\big|\, n\in\Delta_i\big\}.
\]
We assume that $\psi_0$ is strictly convex, that is, $\Sigma$ is the normal fan
of a lattice polyhedron $\check\tau=\Delta_0$, while $\Delta_1,\ldots,\Delta_q$
are either a point or an elementary simplex (a lattice simplex with all integral
points vertices).
Given this data, define a monoid $P\subseteq \ZZ^{n+1}$ by
\begin{equation}
\label{Eqn: P}
\textstyle
P:=\Big\{ (m,a_0,\ldots,a_q)\in \ZZ^{n-q} \times 
\ZZ^{q+1}=\ZZ^{n+1} \,\Big|\, \mbox{$a_i\ge
\psi_i(m)$ for $0\le i\le q$}\Big\}.
\end{equation}
If we let $K$ be the cone in $(\RR^{n+1})^*$ generated by
\[
\bigcup_{i=0}^q \big(\Delta_i\times \{e_i^*\}\big) =
\big(\tau\times\{e_0^*\}\big)\cup \bigcup_{i=1}^q \big(\Delta_i\times \{e_i^*\}\big),
\]
where $e_0^*,\ldots,e_q^*$ is the dual basis of the standard basis
$e_0,\ldots,e_q$ in $\RR^{q+1}$, it is easy to check that $P=K^{\vee}\cap
\ZZ^{n+1}$. Now the map of monoids
\[
\NN^{q+1}\to P,\quad (a_0,a_1,\ldots,a_q)\longmapsto (0,a_0,\ldots,a_q)
\]
defines a ring homomorphism $\CC[t,s_1,\ldots,s_q]\to \CC[P]$. The composition
\begin{equation}
\label{Eqn: Simple local model}
\Specan\CC[P]\lra \CC\times\CC^q \lra \CC
\end{equation}
is a local model of a toric degeneration. The central fiber has irreducible
components $(\CC^*)^q\times\Spec \CC[C\cap\ZZ^{n-q}]$ indexed 
by the maximal cones $C \in \Sigma$ and pairwise glued along
their toric prime divisors as prescribed by $\Sigma$.
\end{construction}

We are now in position to define toric degenerations with simple singularities
based on the local models provided in Construction~\ref{Constr: Local model
simple sings}. 

\begin{definition}
\label{Def: simple singularities}
A toric degeneration $\shX\to \DD$ has \emph{semi-simple singularities} if
locally it is isomorphic to a model
\[
\pi:\Specan \CC[P]\lra \CC
\]
as defined in Construction~\ref{Constr: Local model simple sings}. If in
addition we can assume the Newton polytopes $\Delta_1,\ldots,\Delta_q$ are
standard rather than merely elementary simplicies, we say that $\shX\to\DD$ has
\emph{strongly semi-simple singularities}.
\end{definition}

Note that a toric degeneration with semi-simple singulartities is necessarily
faithful because $q$ in the local model is the dimension of the toric stratum
considered.

Our definition of semi-simple singularities generalizes the definition of simple
singularities in \cite[Def.1.60]{logmirror1}, which in addition requires a
mirror dual condition that amounts to a rigidity condition along a toric
stratum. For example, in dimension two, a toric degeneration is semi-simple iff
at each point of the log singular locus $Z\subset X_0$ it is locally isomorphic
to $uv=w\cdot t^\kappa$, while being simple imposes in addition that $Z$
intersects each irreducible component of $(X_0)_\sing$ in at most one point.

The local equivalence of both definitions follows in finite deformation order,
that is, modulo $t^{k+1}$, from the local models for the logarithmic central
fiber in \cite[Thm.2.6]{logmirror2} and the local rigidity result
concerning deformations \cite[Lem.2.15]{logmirror2}. The extension from finite
order deformations to analytic deformations follows from an analytic
approximation result for relatively smooth log morphisms, as
carried out in \cite[Prop.4.9]{RS2} for the case with simple singularities.

Toric degenerations with simple singularities are the most desirable from many
perspectives: Their central fibers have canonical (log-) versal deformations
constructed starting from combinatorial data on $(B,\P)$ by the smoothing algorithm in
\cite{affinecomplex}, see \cite[\S{A.2}]{theta} together with \cite[\S4.3]{RS2},
their Hodge theory is largely reflected by the affine geometry of $B$
\cite{logmirror2}, and mirror symmetry is a perfect duality on the
class of such toric degenerations, exhibited by a discrete version of Legendre
duality on $(B,\P)$, see \cite{logmirror1}. In the following sections we describe the general fiber of a toric degeneration with simple singularities, by studying the Kato--Nakayama space associated to the central fiber. 

\subsection{Lifted open gluing data and slab functions}
\label{Sec: open gluing data}
In the remainder of this section 
we review the \emph{(lifted) open gluing data} introduced in \cite[\S5.1]{logmirror1}, which plays a crucial role in the reconstruction of a toric degeneration from combinatorial data on a toric log Calabi--Yau space $X_0$, that plays the role of the central fiber \cite{affinecomplex}. Roughly speaking, gluing data is the data determining how the big torus orbits on the toric irreducible components of $X_0$ are assembled together. This assembling is determined by \emph{closed gluing data} -- see \cite{logmirror1}, Definition~2.3
and Definition~2.10. While closed gluing
data can be interpreted as changing the closed embeddings defined by the inclusion of toric strata on the central fibre of a toric degeneration, we also have a modification of it given by \emph{open gluing
data} \cite{affinecomplex,logmirror1}, where ``Open'' refers
to the fact that these gluing data modify open embeddings. 
We note that the reconstruction of a toric degeneration can be carried either on the \emph{cone side} (the intersection complex) or on the \emph{fan side} (the dual intersection complex) \cite[\S2]{logmirror1}. In what follows we follow conventions of \cite{affinecomplex}, and review ``open gluing data'' for the cone side.

Let $\sigma \in \mathcal{P}_{\mathrm{max}}$ be a maximal cell, and let $\tau \subseteq \sigma$. As explained in \cite[Construction~1.17]{affinecomplex}, there is a space $\mathrm{PM}(\tau)$ ($\breve{\mathrm{PM}}(\tau)$ in the notation of \cite{logmirror1})  of piecewise multiplicative functions along $\tau$, that depends only on the embedding of $\tau$ in $B$. The following definition, which can be found in \cite[Defn 1.18]{affinecomplex}, explains how such piece-wise multiplicative functions can be used to modify the gluing of affine pieces, which cover mutual intersections  of irreducible components of $X_0$.
\begin{definition}
\label{open_gluing_data}
\emph{Open gluing data} for $(B,\P)$ are data $s=(s_{ e})$, indexed by $e \colon \omega \to\tau$ inclusion of faces in $\P$, with the following properties: (1) $s_{ e}\in \PM(\tau)$ for
$e:\omega\to\tau$ (2) $s_{\id_\tau}=1$ for every $\tau\in\P$ (3) if
$e:\tau\to \tau'$, $f:\tau'\to \tau''$ then $s_{ f\circ e}=
s_{ f} \cdot s_{ e}$ wherever defined:
\[
s_{ f\circ e,\sigma}\ =\
s_{ f,\sigma}\cdot s_{ e,\sigma}\qquad
\text{for all $\sigma\in\P_\max$ with $\sigma\supseteq \tau''$}.
\]
\end{definition}
In the case when $B$ is positive and simple, we can recover open gluing data from lifted gluing data, defined as in \cite[Defn 5.1]{logmirror1} as a $\check{C}$ech cocycle.
Not all open gluing data arises in this way from lifted gluing data -- see \cite[Prop 4.25]{logmirror1}. However, if it does, it
arises from unique lifted gluing data \cite[Thm 5.2]{logmirror1}. Thus we view the set of lifted gluing data as a subset of the set of open gluing data. The following result is \cite[Thm 5.4]{logmirror1}, stated on the cone side. 
\begin{theorem}
\label{thm: lifted gluing}
In the case of positive and simple singularities, a polarized toric
log Calabi-Yau space with given intersection complex
$(B,\P)$ is defined uniquely up to isomorphism by
lifted gluing data $s\in H^1 (B,\iota_*\breve{\Lambda}
\otimes\CC^\times)$, where $\breve{\Lambda}$ denotes the sheaf of integral cotangent vectors. 
\end{theorem}
For the details on how to define the log structure on a toric log Calabi--Yau space from lifted open gluing data, see \cite[\S4]{logmirror1}. In the remaining part of this section we discuss the role of gluing data in connection with the reconstruction algorithm of \cite{affinecomplex}. 
One of the main ingredients in this algorithm is a wall-structure on $(B,\P)$.
Part of the wall-structure are the \emph{slab functions} 
$f_{\rho,x}$, which are algebraic functions indexed by a 
codimension $1$ cell $\rho$ of $\P$ and a connected component $x$ 
of $\rho \setminus \Delta$ \cite[Defn.~2.17]{affinecomplex}.
These functions are obtained from the data of the affine monodromy around the discriminant locus $\Delta$, and they control consistent
gluings of the local models of the deformation of the central fibre. As it is not the focus of the current article, we refer to \cite{affinecomplex} for a comprehensive study of the reconstruction via wall-structures. Our interest in the wall-structure in the context of this paper is based on its connection to lifted open gluing data. We review the connection between slab functions and lifted open gluing data in the following discussion.

\begin{discussion}
\label{Discussion: slabs}
At a general point of the double locus
$(X_0)_\sing$, exactly two irreducible components $X_\sigma$,
$X_{\sigma'}\subset X_0$ intersect and $\rho=\sigma_1\cap\sigma_2$ is an
$(n-1)$-cell. At such a point by Proposition \ref{Prop: description of shX in codim=1} there is a local description of $\shX$ of the form
\begin{equation}
\label{Eqn: Slab equation}
uv=f(z_1,\ldots,z_{n-1})\cdot t^\kappa,
\end{equation}
with $t$ a generator of the maximal ideal of $R$, $\kappa\in\NN\setminus\{0\}$,
$z_1,\ldots,z_{n-1}$ toric coordinates for the maximal torus of
$X_\rho= X_\sigma\cap X_{\sigma'}$ and $u,v$ restricting either to $0$
or to a monomial on $X_\sigma$, $X_{\sigma'}$. The statement that the restriction of the function $f$ is
well-defined after choosing $u,v$ and that this restriction classifies the log structure on the central fibre is one of the main results of \cite{logmirror1}. The zero locus of $f$ in $X_\rho$ specifies the log
singular locus $Z\subset (X_0)_\sing$ where the log structure $\M_{X_0}$ is
not fine. Thus $\M_{X_0}$ is fine outside a closed subset $Z\subset (X_0)_\sing$
of codimension two, a union of hypersurfaces on the irreducible components
$X_\rho$ of $X_0$, $\dim\rho=n-1$. Conversely, there is a sheaf $\shF$
on $X_0$ with support on $(X_0)_\sing$, an invertible $\O_{X_0}$-module on the
open dense subset where $X_0$ is normal crossings, with everywhere
locally non-zero sections classifying log structures arising from a local
embedding into a toric degeneration (see \cite{logmirror1}, Theorem~3.22 and
Definition~4.21). Thus the moduli space of log structures on $X_0$
over $O^\dagger$ that look like coming from a toric degeneration can be explicitly described by an open
subset of $\Gamma((X_0)_\sing,\shF)$ for some coherent sheaf $\shF$ on
$(X_0)_\sing$ -- see \cite[Thm 3.22]{logmirror1} for details. 
By \cite[Remark 2.18 2)]{affinecomplex}, an element of $\Gamma((X_0)_\sing,\shF)$ can be described in terms of
(order $0$) \emph{slab functions} $f_{\rho,x} \in \CC[\Lambda_\rho]$, indexed by $\rho$ a codim $1$ cell of $\mathcal{P}$, $x$ a connected component of 
$\rho \setminus \Delta$, and where 
$\Lambda_\rho$ is the space of integral tangent vectors to 
$\rho$. For an explicit description of these functions see \cite[Defn.~ 2.17]{affinecomplex}. 
\end{discussion}
The following result can be found in \cite[Thm ~5.2]{logmirror1}, under the assumption that $(B,\P)$ be positive and simple -- see \cite[\S 4.3]{logmirror1} and \cite[\S 5]{logmirror1} for a comprehensive discussion on positivity and simplicity. 
\begin{theorem}
\label{Thm: slab fncs determine loggy}
Let $s$ be lifted open gluing data. Then we have the following uniquely determined
data:
\begin{itemize}
\item[1)] The log singular locus, given by a closed subset $Z \subseteq X_0(B,\P, s)$, contained in the union of codimension
one strata and not containing any zero-dimensional strata.
\item[2)] Normalized slab functions $f_{\rho,x} \in \CC[\Lambda_\rho]$, which define a log structure on $X_0$. 
\end{itemize}
\end{theorem}
Being normalized means that if $f_\rho$ is the
slab function describing the log structure near a zero-dimensional
toric stratum $x\in X_0$ along a codimension one stratum
$X_\rho\subset X_0$ with $x\in X_\rho$, then $f_\rho(x)=1$. By the
discussion after Definition~4.3 in \cite{logmirror1}, there always
exist open gluing data normalizing a given toric log Calabi-Yau space, so
this assumption imposes no restriction. Note that while previously
we had viewed slab functions only as functions on (analytically)
open subsets of the big torus of $X_\rho$, hence as Laurent
polynomials, in the setup of \cite{logmirror1} or
\cite{affinecomplex} they extend as regular functions to the
zero-dimensional toric stratum they take reference to.


\section{The topology of toric degenerations and \texorpdfstring{$X_0^\KN$}{X0KN}}
\label{Subsect: Toric log Calabi-Yau spaces}

We now come to the first main result of this paper, which while a
consequence of results already in the literature, is our main motivation for the
study of $X_0^\KN$ and hence deserves to be spelled out. Similar statements to
the following result, in the case of affine hypersurfaces in smooth toric
varieties, can be found in \cite[Prop.5.12 and Cor.5.13]{rstz}.

\begin{theorem}
\label{Thm: KN for X0 in simple case}
Let $\delta:\shX\to \DD=\{z\in\CC\,|\, |z|<R\}$ be a toric degeneration with
strongly semi-simple singularities (Definition~\ref{Def: simple singularities}).
Then the restriction $\shX\setminus X_0\to \DD\setminus\{0\}$ is a topological
fiber bundle isomorphic to
\[
\delta^{KN}\times\id: X_0^\KN\times(0,r)\lra (O^\ls)^\KN\times (0,r)= U(1)\times(0,r).
\]
In particular, for $t= re^{i\theta}\in\DD$, the restriction of this isomorphism
to the fibre over $(e^{i\theta},r)$ induces a homeomorphism
\begin{equation}
\label{Eq: restriction to a phase}
\shX_t \simeq X_0^\KN(e^{i\theta})
\end{equation}
between the general fibre of $\delta$ and the fiber $X_0^\KN(e^{i\theta})$ of $\delta^{KN}$ over $e^{i\theta}\in U(1)$.
\end{theorem}

The proof generalizes to the semi-simple case. However, this generalization
being a somewhat technical point, which requires a detailed local analysis of
the orbifold structure of the log singular locus in such cases, we skip it in
this paper. For the interested reader the proof can be provided in some
unpublished notes of Mark Gross. Note also that semi-simple, but not strongly
semi-simple cases start appearing only in dimensions $>3$ with $\shX_t$ having
singularities in codimension~$4$.

For the proof we have to recall the notion of relative coherence and relative
log smoothness, as defined in \cite[Def.3.6]{NO} and \cite[p.399ff]{Ogus}.
Recall that a face $F$ of a monoid $P$ is a submonoid such that $x, y \in
P$, $x+y \in F$ implies $x, y \in F$. The corresponding notion of sheaf of faces
of a sheaf of monoids is defined by requiring the face condition stalkwise.

\begin{definition} 
\label{def: relatively coherent} 
Let $\mathcal{F} \to \mathcal{O}_X$ be a log structure on a complex analytic
space $X$. $\mathcal{F}$ is \emph{relatively coherent} if locally on $X$ there
exists a coherent log structure $\M$ containing $\mathcal{F}$ such that
$\mathcal{F}$ is locally generated as a sheaf of faces in $\M$ by a finite
number of sections of $\M$. In this case one says that 
$\mathcal{F}$ is relatively coherent in $\mathcal{M}$.
\end{definition}

Any sheaf of faces $\mathcal{F}$ of $\M_X$ of a coherent log structure $\M_X \to
\mathcal{O}_X$ is again a log structure, but the point of the definition is that
$\shF$ may not be coherent itself. The notion of log smoothness extends to this
class of non-coherent log structures.

\begin{definition}
\label{def:relatively smooth}
Let $(Y,\shM_Y)$ be a fine log space. A morphism of log spaces $f: (X,\shF) \to
(Y,\shM_Y)$ is \emph{relatively log smooth} if locally on $X$ there is a fine
log structure $\M$ on $X$
in which $\shF$ is relatively coherent and satisfying the following conditions:
\begin{itemize}
\item[i)] The composed map $(X,\shM) \to (X,\shF) \to (Y,\shM_Y)$ is log smooth.
\item[ii)] The stalks of the quotient monoid $\shM/\shF$ are free monoids.
\end{itemize}
\end{definition}

The following is Theorem $5.1$ in \cite{NO}. For the notion of exactness of a log morphism see \cite[Def.I.2.1.15]{Ogus}.

\begin{theorem} 
\label{Thm: Nakayama/Ogus}
Let $f:(X,\M_X)\to (Y,\M_Y)$ be a proper, separated, exact and relatively smooth
morphism of log analytic spaces, with $(Y,\M_Y)$
fine and $(X,\M_X)$ relatively coherent. Then $f^\KN : (X,\M_X)^\KN\to
(Y,\M_Y)^\KN$ is a topological fibre bundle, with fibres oriented topological
manifolds, possibly with boundary.
\qed
\end{theorem}

We are now in position to prove Theorem~\ref{Thm: KN for X0 in simple case}.

\begin{proof} (of Theorem~\ref{Thm: KN for X0 in simple case}.)
Since the log structure of $\shX$ is trivial outside $X_0$, the conclusion of
the theorem follows immediately from \cite{NO} once we
verify the hypothesis of Theorem~\ref{Thm: Nakayama/Ogus} for $(X_0,\M_{X_0})\to O^\ls$.

Since $\M_Y$ has stalks $\NN$, the condition of exactness follows from relative
smoothness.

By Definition~\ref{Def: simple singularities}, the toric degeneration $\shX\to \DD$ is locally isomorphic to the composition
\[
\Spec \CC[P]\lra \CC\times \CC^q\lra \CC
\]
defined by $t=z^{e_0}$, where $P$ is as in \ref{Eqn: P}. There are two natural
log structures on $\Spec \CC[P]$, the toric one $\M_P$ and the divisorial one
$\M_\delta$ defined by the divisor $(t=0)\subset \Spec\CC[P]$. Now $\M_\delta$
is not fine, but relatively coherent. In fact, adopting the notation of
Construction~\ref{Constr: Local model simple sings}, $\M_\delta$ is the sheaf of
faces defined by the minimal face $F\subset P$ containing $e_0$.

To verify Definition~\ref{def:relatively smooth},(ii), we need to verify that
the quotient monoid $P/F$ is free. To this end denote by $\ol K=e_0^\perp\cap K$
the cone in $(\RR^*)^{n+1}$ generated by
\begin{equation}
\label{Eqn: Monodromy cone}
\bigcup_{i=1}^q \big(\Delta_i\times \{e_i^*\}\big).
\end{equation}
Then $(P/F)^\vee=\Hom(P/F,\NN)$ is canonically isomorphic to the monoid of
integral points in $\Hom(\ol K,\RR_{\ge0})\subset F^\perp$. Since by assumption
all $\Delta_i$ are elementary simplices, $\ol K$ is a standard cone, and hence
its integral points define a free monoid. Thus $(P/F)^\vee$ is free and in turn
so is $P/F$. This verifies Definition~\ref{def:relatively smooth},(ii) at the
zero-dimensional toric stratum of $\Specan \CC[P]$.

At other points of the divisor $(t=0)\subset \Specan \CC[P]$, the inclusion
$(\ol\M_\delta)_x\subset (\ol\M_P)_x$ is given by taking the quotient of
$F\subset P$ by a face of $P/F$, thus verifying both the condition of relative
coherence and freeness of the quotient monoid. Hence, the log structure on
$(\shX, \M_\shX)$ and its restriction to $(X_0,\M_{X_0})$ are relatively
coherent and the morphism to $(\DD,\M_\DD)$ and to $O^\ls$, respectively, are
relatively log smooth.

We have thus verified the hypothesis of Theorem~\ref{Thm: Nakayama/Ogus}. Hence, the result follows.
\end{proof}


\section{Torus fibrations on Kato--Nakayama spaces}
\label{sec: The Topological Classification of Torus Bundles}
Let $B$ be a topological manifold and let $\mu: X\to B$ be a torus bundle of
relative dimension $r$ with transition
functions taking values in $U(1)^r\rtimes \GL(r,\ZZ)$. The topological classification of such torus bundles works in
analogy with the Lagrangian fibration case discussed in
\cite{Duistermaat}.  Note that $\Lambda= R^1\mu_* \ul\ZZ$ is a local
system with fibres 
\[ \Lambda_x= H^1(\mu^{-1}(x),\ZZ)\simeq \ZZ^r.\] 
We write $\Hom(\Lambda, U(1))$ for the associated torus
bundle. As a set, \[\Hom(\Lambda, U(1))= \coprod_{x\in B}
\Hom(\Lambda_x,U(1)),\] and the topology is generated by sets, for
$m\in\Gamma(U,\Lambda)$ with $U\subset X$ and $V\subset U(1)$
open,
\[
V_m =\big\{ (x,\varphi)\,\big|\, x\in U,\, \varphi
\in\Hom(\Lambda_x,U(1)),\, \varphi(m)\in V\big\}.
\]
The torus $\Hom(\Lambda_x,U(1))\simeq U(1)^r$ acts fibrewise by
translation on any trivialization $\mu^{-1}(U) \simeq U\times T^r$,
and this action does not depend on the trivialization. Hence $X\to
B$ can be viewed as a $\Hom(\Lambda, U(1))$-torsor. In
particular, if $\mu:X\to B$ has a section, then \[X\simeq
\Hom(\Lambda, U(1))\simeq \check\Lambda\otimes U(1)\] is
isomorphic to the trivial torsor. Here we write $\check \Lambda =
\shHom(\Lambda,\ul\ZZ)$.
The cohomology group $H^1(X,\check\Lambda\otimes U(1))$
classifies isomorphism classes of $\check\Lambda\otimes
U(1)$-torsors by the usual \v Cech description. Moreover, from the
exact sequence
\[
0\lra\check\Lambda \lra
\check\Lambda\otimes_\ZZ \RR\lra \check\Lambda\otimes_\ZZ U(1)\lra
0,
\] exhibiting $\check\Lambda\otimes_\ZZ U(1)$ fibrewise as a
quotient of a vector space modulo a lattice, we obtain the long
exact cohomology sequence
\begin{equation}
\label{Eq: delta as connecting hom}
\ldots\lra H^1 (B,\check\Lambda)\lra H^1(B,\check\Lambda\otimes\RR)
\lra H^1(B,\check\Lambda\otimes U(1))\stackrel{\delta}{\lra}
H^2(B,\check\Lambda)\lra\ldots
\end{equation}
The image of the class $[X]\in H^1(B,\check\Lambda\otimes U(1))$ of $\mu \colon X\to B$
under the connecting homomorphism, as
a $\check\Lambda\otimes U(1)$-torsor, defines
the obstruction $\delta([X])\in H^2(B,\check\Lambda)$ to the existence of a
continuous section of $\mu: X\to B$. The proof works as in the symplectic case
discussed in \cite{Duistermaat}, p.696f.

With this notation, we summarize the general discussion on torus
bundles with locally constant transition functions in the following
proposition. 

\begin{proposition}
\label{Prop: Topological torus fibrations}
Let $B$ be a topological manifold and let $\mu: X\to B$ be a torus bundle of
relative dimension $r$ with transition
functions taking values in $U(1)^r\rtimes \GL(r,\ZZ)$. Then up to isomorphism, $X\to B$ is given
uniquely by the local system 
\[\Lambda= R^1\mu_*\ul\ZZ\] with fibres
$\ZZ^r$ and a class 
\[[X]\in H^1(B,\check\Lambda\otimes U(1)).\]
Moreover, a continuous section of $\mu$ exists if and only if $\delta([X])\in H^2(B, \check\Lambda)$ vanishes, where $\delta$ is as in \eqref{Eq: delta as connecting hom}. In this case, we have an isomorphism $X\simeq \Hom(\Lambda, U(1))$ of torus bundles.
\qed
\end{proposition}

\begin{remark}
\label{Rem: Hom^0}
A torus bundle $X\to B$ with locally constant transition functions in $U(1)^r\rtimes\GL(r,\ZZ)$ as in the preceding proposition can be explicitly reconstructed from its class $[X]\in H^1(B, \check\Lambda\otimes U(1))$ as follows. Since $H^1(B, \check\Lambda\otimes U(1))= \Ext^1(\Lambda, U(1))$, any such class defines an extension
\[
1\lra U(1)\stackrel{i}{\lra} \shE\stackrel{q}{\lra} \Lambda\lra 0
\]
of $\Lambda$ by $U(1)$. Then the set $\Hom^\circ(\shE, U(1))$ of homomorphisms
$\varphi:\shE_x\to U(1)$, $x\in B$ with $\varphi\circ i= 1$ is naturally a
$\check\Lambda\otimes U(1)$-torsor by letting $\lambda:\Lambda_x\to U(1)$
act by $\varphi\mapsto (\lambda\circ q)\cdot \varphi$. By the
description of gluing trivial bundles via a \v Cech $1$-cocycle, it is then not
hard to construct an isomorphism of $\Hom^\circ(\shE, U(1))$ with $X\to B$ as
$\check\Lambda\otimes U(1)$-torsors.
\end{remark}

\begin{remark}
\label{Rem: Sections}
Let $X_0^\KN|_{B\setminus\shA} \lra B\setminus\shA$ be the torus bundle defined
by the central fibre of a proper toric degeneration. If $X_0$ is built using trivial gluing data
(\cite{logmirror1},Example $2.6$), then $[X_0^\KN|_{B\setminus\shA}]\in
H^1(B\setminus\shA,\check\Lambda\otimes U(1))$ is trivial and hence the
existence of a continuous section follows immediately from the above
proposition. More generally, we will see in Subsection~\ref{Subsect: X_0^KN
real} that for gluing data in $H^1(B,\iota_*\check\Lambda\otimes
\RR^\times)\subset H^1(B, \iota_*\check\Lambda\otimes\CC^\times)$, the
Kato-Nakayama space $X_0^\KN$ has a real structure (Corollary~\ref{Cor: Standard
real structure cohomological}) and the real locus forms a branched cover of
degree $2^{n+1}$ over $B$ (Proposition~\ref{Prop: real locus is a branched cover
of B}). Moreover, for certain gluing data, including the trivial one, this cover
contains a connected component mapping homeomorphically to $B$, thus defining a
section. See also Theorem~2.6 in \cite{GrossGeometry} for a related result on
the existence of Lagrangian sections in the context of SYZ-fibrations.
\end{remark}
In the remaining part we discuss the key point of this article, and provide 
a description of the Kato--Nakayama space over the central fibre of a toric degeneration as a topological torus bundle. We then discuss how to canonically classify torus fibrations on it by  controlling the gluing data in \S \ref{Sec: Topological Torus Bundles: a Study via Gluing Data}.

Denote by $Z^\KN= \pi^{-1}(Z)\subset X_0^\KN$ where $Z\subset X_0$ is the
log singular locus and 
\[ \shA=\mu(Z)\subset B \]
the amoeba image of it in $B$. Consider the composition
\begin{equation}
\label{Eq: fibration on KN}
\mu^\KN: X_0^\KN \stackrel{\pi}{\lra} X_0 \stackrel{\mu}{\lra} B
\end{equation} 
of the projection of the Kato-Nakayama space with the generalized momentum map. The main result of this
section, Theorem~\ref{Thm: (X_0)^KN->B is a torus bundle} gives a canonical description of the subset $(\mu^\KN)^{-1}(B\setminus\shA)\subset X_0^\KN$ as a torus bundle over $B \setminus \shA $. 
We will also provide a cohomological description of this torus
bundle in analogy with the case of Lagrangian fibrations treated in
\cite{Duistermaat}. See also \cite{GrossTopology}, \S3 and
\cite{GrossGeometry}, \S2 for a discussion in the context of SYZ-fibrations. For
$U\subset B\setminus A$ any subset, we write $X_0^{\KN}|_U= (\mu^\KN)^{-1}(U)$,
viewed as a topological torus bundle over $U$. Generally, topological $r$-torus
bundles are fibre bundles with structure group $\operatorname{Homeo} (T^r)$, the
group of homeomorphisms of the $r$-torus. There is the obvious subgroup
$U(1)^r\rtimes \GL(r,\ZZ)$ of homeomorphisms that lift to affine transformations
on the universal covering $\RR^{r+1}\to T^r= \RR^r/\ZZ^r$. In higher dimensions
(certainly for $r\ge 5$), there exist exotic homeomorphisms that are not
isotopic to a linear one \cite[Thm. ~4.1]{Hatcher}. However, in the present
situation such exotic transition maps do not occur, and we can even find a
system of local trivializations with transition maps induced by locally constant
affine transformations.

For $\sigma\in\P^\max$ denote by $\M_{X_\sigma}$ the toric log
structure for the irreducible component $X_\sigma\subset X_0$ and
by $X_\sigma^\KN$ its Kato-Nakayama space. By \cite{logmirror1},
Lemma~5.13, there is a canonical isomorphism
\begin{equation}
\label{Eqn: GS5.13}
\M_{X_0}^\gp|_{X_\sigma\setminus Z}\simeq
\M_{X_\sigma}^\gp|_{X_\sigma\setminus Z}\oplus\ZZ,
\end{equation}
the $\ZZ$-factor generated by the generator $t$ of $\maxid_R$ chosen
above.

First we show that the log singular locus can be dealt with by taking closures,
even stratawise. For $\tau\in\P$ denote by $X_\tau^\KN= \pi^{-1}(X_\tau)\subset
X_0^\KN$ and by $Z_\tau^\KN =Z^\KN\cap X_\tau^\KN$. 

\begin{lemma}
\label{Lem: Z^KN nowhere dense}
On each toric stratum $X_\tau\subset X_0$, the preimage
$Z_\tau^\KN\subset X_\tau^\KN$ of the log singular locus is a
nowhere dense closed subset.
\end{lemma}

\begin{proof}
It suffices to prove the statement over an irreducible component
$X_\sigma\subset X_0$, $\sigma\subset B$ a maximal cell. Let $x\in Z\cap
X_\sigma$ and $X_\tau\subset X_\sigma$ the minimal toric stratum containing $x$.
Since $Z$ does not contain zero-dimensional toric strata, $x$ is not the generic
point $\eta\in X_\tau$. We claim that the generization map $\chi_{\eta
x}:\ol\M_{X_0,x}\to \ol\M_{X_0,\eta}$ is injective. In fact, $\M_{X_0}$ is
locally the divisorial log structure for a toric degeneration. Hence the stalks
of $\ol\M_{X_0}$ are canonically a submonoid of $\NN^r$ with $r$ the number of
irreducible components of $X_0$ containing $X_\tau$, say $X_{\sigma_1},\ldots,
X_{\sigma_r}$. An element $a\in \NN^r$ lies in $\ol\M_{X_0,x}$ iff $\sum a_i
X_{\sigma_i}$ is locally at $x$ a Cartier divisor, in a local description as the
central fibre of a toric degeneration. In any case, both $\ol\M_{X_0,x}$ and
$\ol\M_{X_0,\eta}$ are submonoids of the same $\NN^r$, showing the claimed
injectivity of $\chi_{\eta x}$.

Now take a chart $\ol\M_{X_0,\eta}\to\Gamma(U,\M_{X_0})$ with $U$ a Zariski-open
neighbourhood of $\eta$ in $X_0\setminus Z$.
Shrinking $U$ if necessary, one can assume that $\ol\M_{X_0}|_{U\cap X_\tau}$
is a constant sheaf, and so Proposition~\ref{Prop: KN space from chart} yields a
canonical homeomorphism $\pi^{-1}(U\cap X_\tau)= (U\cap
X_\tau\times \Hom(\ol\M_{X_0,\eta}^\gp, U(1))$. By the definition of the
topology on $X_0^\KN$, the composition
\[
(U\cap X_\tau)\times \Hom(\ol\M_{X_0,\eta}^\gp, U(1)) \lra 
\Hom(\ol\M_{X_0,\eta}^\gp, U(1))  \lra \Hom(\ol\M_{X_0,x}^\gp, U(1)) 
\]
of the projection with pull-back by the generization map $\chi_{\eta
x}$ is continuous. By injectivity of $\chi_{\eta x}$ this
composition is surjective. Since $x\in \cl(U)$ we conclude that
$\pi^{-1}(x)$ is contained in the closure of $\pi^{-1}(U)$, showing
the density of the complement of $Z_{\tau}^{KN}$, which is
equivalent to the desired nowhere density.
\end{proof}

\begin{lemma}
\label{Lem: mu^KN over maximal cell}
For $\sigma\in\P_{\max}$ denote by $\Lambda_\sigma=
\Gamma(\Int\sigma,\Lambda)$ the group of integral tangent vector
fields on $\sigma$. Then there is a canonical continuous surjection
\[
\Phi_\sigma: \sigma\times \Hom(\Lambda_\sigma\oplus\ZZ, U(1)) \lra
(\mu^\KN)^{-1}(\sigma)\subset X_0^\KN,
\]
which is a homeomorphism onto the image over the complement of the
log singular locus $Z\subset X_0$.
\end{lemma}
\begin{proof}
By  Proposition~\ref{Prop: KN of toric variety}, we have a canonical homeomorphism between $\sigma \times \Hom(\Lambda_{\sigma},\ZZ)$ and $(X_\sigma, \shM_\sigma)^{KN}$.
Therefore, we have to construct a map 
\[ \Phi_\sigma \colon (X_\sigma, \shM_\sigma)^{KN} 
\times \Hom(\ZZ, U(1)) \lra (\mu^\KN)^{-1}(\sigma)\,.\]
We will first fix $x \in X_\sigma$ and describe 
$\Phi_{\sigma}$ above $x$. As the fiber of 
the natural projection 
$(X_\sigma, \shM_\sigma)^{KN}  \lra X$ is isomorphic to 
$\Hom(\ol\M_{X_{\sigma},x}^\gp, U(1))$, and the fiber of the natural projection 
$X_0^{KN} \lra X_0$ is isomorphic to 
$ \Hom(\ol\M_{X_{0},x}^\gp, U(1))$, we have to construct a map 
\[ \Phi_{\sigma,x} \colon \Hom(\ol\M_{X_{\sigma},x}^\gp \oplus \ZZ, U(1)) \lra  \Hom(\ol\M_{X_{0},x}^\gp, U(1))\,.\]
Let $X_\tau \subset X_\sigma$ be the minimal toric stratum of $X_\sigma$ containing $x$ and let $\eta$ be the generic point of 
$X_\tau$. We have $\eta \notin Z$: indeed, if one would have 
$\eta \in Z$, one would have 
$X_\tau = \overline{\eta} \subset Z$, in contradiction with the fact that $Z$ does not contain zero-dimensional toric strata.
By minimality of $\eta$, we have $\ol\M_{X_{\sigma},x}^\gp
=\ol\M_{X_{\sigma},\eta}^\gp$. On the other hand, as 
$\eta \notin Z$, we have  
by 
\eqref{Eqn:
GS5.13} a canonical isomorphism $\ol\M_{X_{\sigma},\eta}^\gp \oplus \ZZ
\simeq \ol\M_{X_{0},\eta}^\gp$. Therefore, we have to construct a map 
\[ \Phi_{\sigma,x} \colon \Hom(\ol\M_{X_{0},\eta}^\gp, U(1)) \lra  \Hom(\ol\M_{X_{0},x}^\gp, U(1))\,.\]
We take for $\Phi_{\sigma, x}$ the map induced by the generization map $\chi_{\eta x} \colon \ol\M_{X_{0},x}^\gp \lra \ol\M_{X_{0},\eta}^\gp$. By the proof of Lemma~\ref{Lem: Z^KN nowhere dense}, $\Phi_{\sigma,x}$ is surjective. 
Varying $x$ in $X_\sigma$, we define the map 
$\Phi_\sigma$, which is continuous by definition of the 
topology on Kato-Nakayama spaces. As the fiberwise maps 
$\Phi_{\sigma,x}$ are surjective, the map $\Phi_\sigma$
is also surjective.

Finally, for $x \notin Z$, we have by 
\eqref{Eqn:
GS5.13} a canonical isomorphism $
 \ol\M_{X_{0},x}^\gp \simeq \ol\M_{X_{\sigma},x}^\gp \oplus \ZZ$, hence $ \ol\M_{X_{0},x}^\gp \simeq \ol\M_{X_{0},\eta}^\gp$, and so $\Phi_{\sigma,x}$ is bijective. We conclude that $\Phi_\sigma$
 is an homeomorphism onto its image over the complement of $Z$.

\end{proof}
\begin{remark}
Let $\sigma\in\P_{\max}$ be a maximal cell as in Lemma \ref{Lem: mu^KN over maximal cell}. It follows immediately from the definitions that with respect to the product decomposition
\[
\sigma\times\Hom(\Lambda_\sigma\oplus\ZZ, U(1)) =
X_\sigma^\KN\times U(1),
\]
of the domain of the map $\Phi$ according to Proposition~\ref{Prop: KN of toric
variety}, the restrictions of $\pi: X_0^\KN\to X_0$ and $\delta^\KN: X_0^\KN\to
S^1$ to $(\mu^\KN)^{-1}(\sigma)\setminus Z^\KN$ are given by the projection to $X_\sigma^\KN$
followed by $X_\sigma^\KN\to X_\sigma$ and by the projection to $U(1)$,
respectively.
\end{remark}
\begin{theorem}
\label{Thm: (X_0)^KN->B is a torus bundle}
The map $\mu^\KN: X_0^\KN\to B$ in \eqref{Eq: fibration on KN} is a
bundle of $(n+1)$-tori over $B\setminus\shA$, where $n = \dim B$. Similarly, the fibre $X_0^\KN(\xi)$ of $\delta^{KN}:X_0^\KN \to S^1$, defined in \eqref{Eq: restriction to a phase}, is a bundle of $n$-tori over $B\setminus\shA$.
\end{theorem}
\begin{proof}
For $\sigma\in\P^\max$ denote by $T_\sigma =
\Hom(\Lambda_\sigma,U(1))\times U(1)$ the $(n+1)$-torus fibre of
$\mu^\KN$ over $\sigma$ in the description of Lemma~\ref{Lem: mu^KN
over maximal cell}. For $x\in B\setminus \shA$ let $\tau\in \P$ be
the unique cell with $x\in\Int\tau$. Let $n=\dim B$ and $k=\dim
\tau$. Then in an open contractible neighbourhood $U\subset
B\setminus\shA$ of $x$, the polyhedral decomposition $\P$ looks like
the product of $\Lambda_\tau\otimes_\ZZ \RR$ with an
$n-k$-dimensional complete fan $\Sigma_\tau$ in the vector space
with lattice $\Lambda_x/\Lambda_{\tau,x}$. Over each maximal cell
$\sigma$ containing $\tau$, we have the canonical homeomorphism of
$(\mu^\KN)^{-1} (\sigma)$ with $\sigma\times \Hom(\Lambda_\sigma,
U(1))\times U(1)$ provided by Lemma~\ref{Lem: mu^KN over maximal
cell}. Thus for any pair of maximal cells $\sigma,\sigma'
\supset\tau$ we obtain a homeomorphism of torus bundles
\[
\Phi_{\sigma'\sigma}:
(U\cap\sigma\cap\sigma')\times T_\sigma\lra
(U\cap\sigma\cap \sigma')\times T_{\sigma'}.
\]
We only claim a fibre-preserving homeomorphism of total spaces here,
$\Phi_{\sigma,\sigma'}$ does in general not preserve the torus
actions. In any case, these homeomorphisms are compatible over
triple intersections, hence provide homeomorphisms of torus bundles
also for maximal cells intersecting in higher codimension. This way
we have described $\pi^{-1}(U)$ as the gluing of trivial torus
bundles over a decomposition of $U$ into closed subsets, a clutching
construction.

To prove local triviality from this description of $\pi^{-1}(U)$,
replace $U$ by a smaller neighbourhood of $x$ that is star-like with
respect to a point $y\in \Int(\sigma)$ for some maximal cell
$\sigma\ni x$. By perturbing $y$ slightly, we may assume that the
rays emanating from $y$ intersect each codimension one cell $\rho$
with $\rho\cap U\neq\emptyset$ transversaly. To obtain a fibre-preserving
homeomorphism $\pi^{-1}(U)\simeq U\times T_\sigma$, connect any other
point $y\in U$ with $x$ by a straight line segment $\gamma$. Then
$\gamma$ passes through finitely many maximal cells $\sigma'$. At
each change of maximal cell apply the relevant
$\Phi_{\sigma'\sigma''}$ to obtain the identification of the fibre
over $y$ with $T_\sigma$.
\end{proof}

\begin{lemma}
\label{Lem: Locally constant transition fcts}
The topological torus bundle $X_0^\KN|_{B\setminus\shA}\to
B\setminus\shA$ has a distinguished atlas of local trivializations
with locally constant transition maps in $U(1)^{n+1}\rtimes
\GL(n+1,\ZZ)$.
\end{lemma}

\begin{proof}
It suffices to consider the attaching maps between the trivial pieces
\[(\mu^\KN)^{-1}(\sigma)= \sigma\times \Hom(\Lambda_\sigma \oplus\ZZ,U(1))\] of
Lemma~\ref{Lem: mu^KN over maximal cell} for maximal cells $\sigma,\sigma'$ with
$\rho=\sigma\cap\sigma'$ of codimension one. Let $V\subset \rho$ be a connected
component of $\rho\setminus\shA$. In Remark~\ref{Rem: Transition of X_^KN in
codim=1} we saw that the transition maps over $V$ are given by the equation
$\arg(v)=-\arg(u)+\arg(f)+\kappa_\rho\cdot \arg(t)$. Now $\mu^\KN|_V$ factors
over the Kato-Nakayama space of $X_\rho$, which can be trivialized as $V\times
\Hom(\Lambda_\rho, U(1))\times U(1)^2$. The last factor is given by
$(\Arg(u),\Arg(t))$, say, and the transition map transforms this trivialization
into the description with $(\Arg(v),\Arg(t))$. Thus this transition is the
identity on the first $n-1$ coordinates given by $\Lambda_\rho$ and on
$\Arg(t)$, while on the last coordinate it is given by $\Arg(u)^{-1}$ times the
phase of the algebraic function $f\cdot t^{\kappa_\rho}$. The homotopy class of
this map for constant $t\neq0$ is given by the winding numbers of a
generating set of closed loops in $\pi_1(T^{n-1})=\ZZ^{n-1}$. These
winding numbers define a monomial function $z^m$ on $V\times
\Hom(\Lambda_\rho,U(1))$ with $z^{-m}\cdot f$ homotopic to a constant
map. The transition function is therefore isotopic to
$(\id_{T^{n-1}},\Arg(z^m\cdot t^{\kappa_\rho}\cdot u^{-1}), \id_{U(1)})$,
fibrewise a linear transformation of $T^{n+1}=T^n\times U(1)$ with coordinates
$(\Arg(z),\Arg(u))$ for $T^n$ and $\Arg(t)$ for
$U(1)$, respectively.
\end{proof}
\begin{remark}
The translational factor $U(1)^{n+1}$ in Lemma \ref{Lem: Locally constant transition fcts} arises because non-trivial gluing data
change the meaning of monomials on the maximal cells by constants. This will be discussed in further detail in \S \ref{Sec: Topological Torus Bundles: a Study via Gluing Data}.
\end{remark}
We provide a more comprehensive discussion on the topological classification of torus bundles with transition maps in $U(1)^{n+1}\rtimes
\GL(n+1,\ZZ)$, as in Lemma \ref{Lem: Locally constant transition fcts}, in \S \ref{Sec: Topological Torus Bundles: a Study via Gluing Data}. As stated in Proposition \ref{Prop: Topological torus fibrations}, if $\mu: X\to B$ is such a a fibre bundle, then up to isomorphism, it is
uniquely given by the local system $\Lambda= R^1\mu_*\ul\ZZ$ with fibres $\ZZ^n$ and a class $[X]\in H^1(B,\check\Lambda\otimes U(1))$. For the Kato-Nakayama space $X_0^\KN|_{B\setminus\shA}$, the governing bundle $R^1\mu_*\ul\ZZ$ is identified as follows.

Recall that on the intersection complex, we fix a multivalued piecewise affine function
$\varphi$ as in \cite[Defn 1.10]{theta}. As in \cite{logmirror1}, we assume $\varphi$ takes values in $Q=\NN$. This data, by \cite[Construction~1.14]{theta}, defines an integral
affine manifold $\BB_\varphi$ with an integral affine action by
$(\RR,+)$, making $\BB_\varphi$ a torsor over $B=\BB_\varphi/\RR$. This torsor comes with a
canonical piecewise affine section locally representing $\varphi$.
The pull-back of $\Lambda_{\BB_\varphi}$ under this section defines
an extension
\begin{equation}
\label{Eqn: Aff^* extension}
0\lra \ul\ZZ\lra\shP \lra \Lambda\lra 0
\end{equation}
of $\Lambda$ by the constant sheaf $\ul\ZZ$ on $B\setminus\shA$. 
\begin{definition}
\label{Def: The chern class}
The extension
class of the sequence \eqref{Eqn: Aff^* extension} equals \[c_1(\varphi) \in \Ext^1(\Lambda,\ul\ZZ) =
H^1(B\setminus\shA, \check\Lambda),\] 
called the first Chern class of $\varphi$.
\end{definition}
For the mirror dual interpretation of the Chern class of $\varphi$ see \cite[Equation~(1.5)]{theta}.
\begin{remark}
\label{Rem: monodromy representation}
Much as in the discussion of the radiance obstruction in \cite{logmirror1},
\S1.1, the first Chern class $c_1(\varphi)$ can also be interpreted as an element in
group cohomology $H^1(\pi_1(B\setminus\shA,x), \check\Lambda_x)$. Here
$\check\Lambda_x\simeq\ZZ^n$ is a $\pi_1(B\setminus\shA,x)$-module by means of
parallel transport in $\check\Lambda$ along closed loops based at some fixed
$x\in B\setminus\shA$. As an element in group cohomology, $c_1(\varphi)$ is
given by a twisted homomorphism $\lambda: \pi_1(B\setminus\shA,x) \to
\check\Lambda_x$, $\gamma\mapsto \lambda_\gamma$, determining the monodromy of
$\shP$ around a closed loop $\gamma$ based at $x$ as follows:
\[
\Lambda_x\oplus \ZZ \lra \Lambda_x\oplus\ZZ,\quad
(v,a)\longmapsto \big(T_\gamma\cdot v,\lambda_\gamma\cdot
v+a\big).
\]
Here $T_\gamma\in\GL(\Lambda_x)$ is the monodromy of $\Lambda$
along $\gamma$ and we have chosen an isomorphism $\shP_x\simeq
\Lambda_x\oplus\ZZ$. Being a twisted homomorphism means that for a
composition $\gamma_1\gamma_2$ of two loops based at $x$,
\[
\lambda_{\gamma_1\gamma_2}= \lambda_{\gamma_2} \circ
T_{\gamma_1}+\lambda_{\gamma_1}.
\]
Here we use the convention that $\gamma_1\gamma_2$ runs through $\gamma_2$
first, and hence $T_{\gamma_1\gamma_2}=T_{\gamma_2}\circ T_{\gamma_1}$. This
interpretation is also compatible with the fact that under discrete Legendre
duality, the roles of $c_1(\varphi)$ and the radiance obstruction swap
(\cite{logmirror1}, Proposition~1.50,3).
\end{remark}

\begin{lemma}
\label{Lem: R^1mu_*ZZ for X_0^KN}
Writing $\breve\mu$ for the restriction of $\mu^\KN: X_0^\KN\to B$
to $B\setminus\shA$, there exists a canonical isomorphism of local
systems $R^1\breve\mu_*\ul\ZZ =\shP$, where $\shP$ is as in \eqref{Eqn: Aff^* extension}.
\qed
\end{lemma}
\begin{proof}
For each point $x\in B$ there is a chart for the log structure on
$X_0$ with monoid $\CC[\shP^+_x]$, where $\shP_x^+\subset \shP_x$ is
a certain submonoid of positive elements with $(\shP_x^+)^\gp
=\shP_x$ (\cite{theta}, \S2.2 and \cite{affinecomplex},
Construction~2.7). Hence from the local description of
$X_0^\KN|_{B\setminus\shA}$ in Lemma~\ref{Lem: mu^KN over maximal
cell} and Theorem~\ref{Thm: (X_0)^KN->B is a torus bundle},
the result follows immediately.
\end{proof}

In view of Lemma~\ref{Lem: R^1mu_*ZZ for X_0^KN}, an immediate corollary is a complete topological
description of $X_0^\KN|_{B\setminus\shA}$ over large open sets.

\begin{corollary}
\label{Cor: Locally constant transition fcts}
Denote by $\tilde\shA\subset B$ a closed subset containing $\shA$ and
such that $B\setminus\tilde\shA$ retracts to a one-dimensional cell
complex. Then as a topological torus bundle,
$X_0^\KN|_{B\setminus\tilde\shA}$ is isomorphic to $\Hom(\shP, \ul
U(1))$.
\end{corollary}
\begin{proof}
By Lemma~\ref{Lem: Locally constant transition fcts} we can treat
$X_0^\KN|_{B\setminus\tilde\shA}$ as a torus bundles with locally constant
transition functions in $U(1)^{n+1} \rtimes \GL(n+1,\ZZ)$. By
Proposition~\ref{Prop: Topological torus fibrations} the obstruction to the
existence of a continuous section then lies in $H^2(B\setminus\tilde\shA,
\shP^\vee)$. This cohomology group vanishes by the assumption on the existence
of a retraction.
\end{proof}

We should emphasize that in Corollary~\ref{Cor: Locally constant transition fcts}, we have first used Lemma~\ref{Lem:
Locally constant transition fcts} to reduce to the case of locally constant
transition functions. As discussed in Remark~\ref{Rem: Transition of X_^KN in
codim=1}, the transition functions for $X_0^\KN|_{B\setminus\shA} \to
B\setminus\shA$ between the canonical charts coming from toric geometry are not
locally constant, and hence this corollary makes a purely topological statement.


\subsection{Parametrizing torus fibrations by gluing data}
\label{Sec: Topological Torus Bundles: a Study via Gluing Data}
To obtain a canonical
description of the fibration on $X_0^\KN$, we study a refinement of
\eqref{Eqn: Aff^* extension}. 
Let $\tilde\shP$ be the sheaf of monomials of the form 
$a z^p$ with $a \in \CC$ and $p$ a local section of $\shP$.
Note that the multiples of the deformation parameter $t$ are well-defined on all charts and so define a constant subsheaf with
fibres $\ZZ\oplus\CC^\times$ of $\tilde\shP$, and so a
refinement of
\eqref{Eqn: Aff^* extension}:
\begin{equation}
\label{Eqn: tilde shP}
0\lra \ul\ZZ\oplus \ul\CC^\times\lra \tilde\shP\lra \Lambda\lra 0.
\end{equation}
The extension class
\[
\big(c_1(\varphi),s \big)\in\Ext^1 (\Lambda,\ul\ZZ\oplus
\ul \CC^\times) = H^1(B\setminus\shA, \check\Lambda)\oplus 
H^1(B\setminus\shA, \check\Lambda\otimes\CC^\times)
\]
has as second component the restriction to $B\setminus\shA$ of the gluing data
$s$, as discussed in \cite{theta}, Remark~5.16.\footnote{The discussion in
\cite{theta} is on the complement of a part $\tilde\Delta\subset B$ of the
codimension two skeleton of the barycentric subdivision. There is a retraction
of $\shA$ to a subset of $\tilde\Delta$. However, the discussion on monomials
works on any cell $\tau\in \P$ not contained in $\shA$ and with $\shX\to \Spec R$
locally toroidal at some point of $X_\tau$.} Furthermore, dividing out
$\RR_{>0}\subset\CC^\times$ defines an extension $\widehat\shP$ of $\Lambda$ by
$\ul\ZZ\oplus\ul U(1)$ with class
\begin{equation}
\label{extension-class}
\big(c_1(\varphi), \Arg(s)\big) \in \Ext^1(\Lambda,\ul\ZZ\oplus \ul U(1))=
H^1(B\setminus\shA, \check\Lambda)\oplus H^1(B\setminus\shA,
\check\Lambda\otimes U(1)).
\end{equation}
Taking this latter extension and the
extension of $\Lambda$ by $\ul\ZZ$ from \eqref{Eqn: Aff^* extension} as
two columns, we obtain the following commutative diagram with exact
rows and columns:\\[1ex]
\begin{equation}
\label{Eqn: X_0^KN as an extension}
\begin{CD}
@. @. 0@. 0\\
@. @. @VVV @VVV\\
1@>>> \ul U(1) @>>> \ul\ZZ\oplus\ul U(1) @>>> \ul{\ZZ} @>>>0\\
@. @V{=}VV @VVV @VVV\\
1@>>> \ul U(1) @>>> \widehat\shP @>>> \shP @>>>0\\
@. @. @VVV @VVV\\
@. @. \Lambda @>{=}>>\Lambda\\
@. @. @VVV @VVV\\
@. @. 0 @. 0\\[2ex]
\end{CD}
\end{equation}
Note that the extension of $\ul\ZZ$ by $\ul U(1)$ in the second row
is trivial by construction. The middle row now defines an extension
of $\shP$ by  $\ul U(1)$.


We use the extension $\widehat\shP$ of $\Lambda$, defined in \eqref{Eqn: X_0^KN as an extension}, to give a \emph{canonical}
description of $X_0^\KN$ on a large subset of $B\setminus\shA$,
assuming the open gluing data \emph{normalizes} the toric log
Calabi-Yau space $(X_0,\M_{X_0})$ (\cite{logmirror1},
Definition~4.3). Now, let
\[B'= \bigcup_{\sigma\in\P^\max}\Int\sigma \cup \big\{ v\in
\P^{[0]}\big\}\]
be the subset of $B$ covered by the interiors of
maximal cells and the vertices of $\P$. Assume that the toric log
Calabi-Yau space $(X_0,\M_{X_0})$ is normalized with respect to open
gluing data $s$ as discussed in \cite[\S4.4]{logmirror1}. Denote by $\widehat\shP$ the extension of $\Lambda$
by $\ZZ\oplus U(1)$ in \eqref{Eqn: X_0^KN as an extension}. Then, we obtain the following description of the Kato--Nakayama space of $X_0$, obtained from affine geometry.
\begin{remark}
In connection with the extension class $\widehat\shP$ defined in \eqref{extension-class}, we would like to emphasize the role of the normalization condition of slab functions as described in \cite[\S4.4]{logmirror1}. The canonical description of the Kato--Nakayama space over maximal cells of the central fiber in Lemma~\ref{Lem: mu^KN
over maximal cell} is based on toric monomials. To extend this
description over a point $x\in B\setminus\shA$ in a codimension one
cell $\rho$, we need the gluing equation~\eqref{Eqn: Gluing
equation} to be monomial along the fibres of the momentum map. This
condition means that the restriction of $f$ to a fibre of the
momentum map $X_\rho\to\rho$ is monomial. This is a strong condition
that in case the Newton polyhedron of $f$ is full-dimensional fails
everywhere except at the zero-dimensional toric strata of $X_\rho$.
The normalization condition then says that the non-trivial gluing of
the torus fibres over the vertices only comes from the gluing
data, hence is entirely determined by the extension class of
$\widehat\shP$ in \eqref{extension-class}.
\end{remark}
 
\begin{proposition}
\label{Prop: X_0^KN versus Hom(widehat shP,U(1))}
There is a
canonical homeomorphism
\[
\Hom^\circ(\widehat\shP, \ul U(1))|_{B'} \stackrel{\simeq}{\lra}
X_0^\KN\big|_{B'}
\]
of topological fibre bundles over $B'$,  where $\Hom^\circ(\widehat\shP,\ul U(1)) \subset \Hom (\widehat\shP, \ul
U(1))$ is the space of fibrewise homomorphisms restricting to the
identity on $\ul U(1)\subset \widehat\shP$.
\end{proposition}

\begin{proof}
We
obtain a canonical description of $\widehat\shP$ over $B'$ as follows. Over a
maximal cell $\sigma$, we have a canonical isomorphism of $\widehat\shP|_\sigma$
with the trivial bundle with fibre $\Lambda_\sigma\oplus\ZZ\oplus U(1)$. Then if
$\sigma,\sigma'\in\P^\max$ and $v\in\sigma\cap\sigma'$ is a vertex, the open
gluing data $s$ define a multiplicative function $s_{v,\sigma}:\Lambda_\sigma\to
\CC^\times$, and similarly for $\sigma'$ \cite[\S5.1]{logmirror1}. We glue the trivial bundles on
$\sigma,\sigma'$ by means of $\Arg \big(s_{v,\sigma'}\cdot
s_{v,\sigma}^{-1}\big)$ on the $U(1)$-factor. The gluing on the discrete part
$\Lambda_\sigma\oplus\ZZ$ is governed by a local representative of the MPL
function $\varphi$, to yield $\shP$ \cite[\S1.1]{RS2}. Accordingly, we obtain a description of
$\Hom^\circ(\widehat\shP, \ul U(1))$ by gluing trivial pieces $\sigma\times
\Hom(\Lambda_\sigma\oplus\ZZ, U(1))$. Note that under the \emph{normalization condition} on the slab functions $f_{\rho,x}$ \cite[\S4.4]{logmirror1}, the gluing of $X_0^\KN$
from the same canonical trivial pieces in Lemma~\ref{Lem: mu^KN over
maximal cell} is given by exactly the same procedure over vertices.
Indeed, given a codimension one cell $\rho$ and a vertex $v\in\rho$,
denoting by $f_{\rho,v}$ the slab function attached to the connected component of $\rho \setminus \Delta$ containing $v$,
in the formula $\arg(u)+\arg(v)= \arg(f_{\rho,v})
+\kappa_\rho\cdot\arg(t)$ given as in Equation \eqref{Eqn: Gluing equation} the term involving $f_{\rho,v}$ disappears
due to the normalization condition.
\end{proof}

\begin{remark}
We note that the class in $H^1(B', \shP^\vee\otimes U(1))$ defining $X_0^\KN|_{B'}$
as a topological torus bundle with locally constant transition functions in
$U(1)^{n+1}\rtimes \GL(n+1,\ZZ)$ according to Lemma~\ref{Lem: Locally constant
transition fcts}, agrees with the extension class $\widehat\shP$ in \eqref{Eqn: X_0^KN as an extension}. This is immediate from Remark~\ref{Rem: Hom^0}.
\end{remark}

To describe the fibres of $\delta^\KN: X_0^\KN\to (O^\ls)^\KN=S^1$, we
need another extension. For $\phi\in U(1)$ denote by
$\Psi_\phi:\ZZ\oplus U(1)\to U(1)$ the homomorphism mapping $(1,1)$ to
$\phi$ and inducing the identity on $U(1)$. We have a morphism of
extensions
\[
\begin{CD}
0@>>> \ul\ZZ\oplus \ul U(1) @>>> \widehat \shP @>>> \Lambda @>>>0\\
@. @V{\Psi_\phi}VV @VVV @VV{\id}V\\
0@>>> \ul U(1) @>>> \widehat \shP_\phi @>>> \Lambda @>>>0,
\end{CD}
\]
with the lower row having extension class $\Psi_\phi(s) \in \Ext^1(\Lambda, \ul
U(1))= H^1(B\setminus\shA, \check\Lambda\otimes \ul U(1))$.

\begin{corollary}
With the same assumptions as in Proposition~\ref{Prop: X_0^KN versus Hom(widehat
shP,U(1))}, the fibre of $\delta^\KN: X_0^\KN|_{B'}\to (\Spec O^\ls)^\KN=
S^1= U(1)$ over $\phi\in U(1)$ is isomorphic to the $n$-torus bundle
with local system $\Lambda$ and with extension class $\Psi_\phi(s) \in H^1(B',
\check\Lambda\otimes\ul U(1))$.
\end{corollary}

\begin{proof}
By Lemma~\ref{Lem: mu^KN over maximal cell}, the restriction of $\delta^\KN$ to
the canonical piece $\sigma\times
\Hom(\Lambda_\sigma\oplus\ZZ,U(1))\subset X_0^\KN$ is given by
composing with the inclusion $\ZZ\to \Lambda_\sigma\oplus\ZZ$. Thus
$(\delta^\KN)^{-1}(\phi)$ consists of homomorphisms $\Lambda_\sigma\oplus\ZZ\to
U(1)$ mapping $(0,1)$ to $\phi$. The statement now follows by tracing through
the gluing descriptions of $\widehat\shP_\phi$ and of the extension class
defined by $(\delta^\KN)^{-1} (\phi)$.
\end{proof}



\section{Real structures in log geometry}
\label{Sect: Real log spaces}
Recall that for a scheme $\bar X$ defined over $\RR$ the Galois group
$G(\CC/\RR)=\ZZ/2\ZZ$ acts on the associated complex scheme $X=\bar
X\times_{\Spec\RR}\Spec\CC$ by means of the universal property of the cartesian
product
\[\begin{CD}
X@>>> \bar X\\
@VVV@VVV\\
\Spec\CC@>>>\Spec\RR.
\end{CD}\]
The generator of the Galois action thus acts on $X$ as an involution
of schemes over $\RR$ making the following diagram commutative
\begin{equation}\label{Diag: iota}
\begin{CD}
X@>\iota >> X\\
@VVV@VVV\\
\Spec\CC@>\mathrm{conj}>>\Spec\CC.
\end{CD}
\end{equation}
Here $\mathrm{conj}$ denotes the $\RR$-linear
automorphism of $\Spec\CC$ defined by complex conjugation.

Conversely, a \emph{real structure} on a complex scheme $X$ is an involution
$\iota: X\to X$ of schemes over $\RR$ fitting into the commutative
diagram~\eqref{Diag: iota}. It is not hard to see that if $X$ is separated
and any two points are contained in an open affine subset (e.g.\ if $X$
is quasi-projective) then $X$ is defined over $\RR$ with $\iota$ the generator
of the Galois action (\cite{Hartshorne}, II Ex.4.7). We call pair $(X,\iota)$ a
\emph{real scheme}. By abuse of notation we usually omit $\iota$ when talking
about real schemes.

\begin{definition}\label{Def: Real log structure}
Let $(X,\M_X)$ be a log scheme over $\CC$ with a real structure
$\iota_{X} :X\to X$ on the underlying scheme. Then a \emph{real
structure} on $(X,\M_X)$ (\emph{lifting $\iota_{X}$)} is an
involution
\[
\tilde\iota_{X}= (\iota_X,\iota_X^\flat): (X,\M_X)\lra (X,\M_X)
\]
of log schemes over $\RR$ with underlying scheme-theoretic morphism
$\iota_{X}$. The data consisting of $(X,\M_X)$ and the involutions
$\iota_X$, $\iota_X^\flat$ is called a \emph{real log scheme}.
\end{definition}

In talking about real log schemes the involutions $\iota_X$,
$\iota^\flat_X$ are usually omitted from the notation. We also
sometimes use the notation $\iota_X$ for the involution of the log
space $(X,\M_X)$ and in this case write $\ul\iota_X$ if we want
to emphasize we mean the underlying morphism of schemes.

\begin{definition}
\label{Def: Category of real log schemes}
Let $(X,\M_X)$ and $(Y,\M_Y)$ be real log schemes. A morphism
$f:(X,\M_X)\to (Y,\M_Y)$ of real log schemes is called \emph{real} if the
following diagram is commutative.
\[\begin{CD}
f^{-1}\iota_Y^{-1}\M_Y @>{f^{-1}\iota^\flat_Y}>> f^{-1}\M_Y\\
@V{\iota_X^{-1}f^\flat}VV@VV{f^\flat}V\\
\iota_X^{-1}\M_X@>{\iota^\flat_X}>>\M_X.
\end{CD}\]
Here the left-hand vertical arrow uses the identification
$\iota_Y\circ f= f\circ\iota_X$.
\end{definition}

\begin{remark}
For a real morphism of real log schemes $f:(X,\M_X) \to (Y,\M_Y)$
the following diagram commutes.
\[\begin{xy}
  \xymatrix @!=0.8pc {
   &  f^{-1}\iota_Y^{-1}\O_Y \ar[rr] \ar[dd]  &  &  f^{-1}\O_Y \ar[dd]  \\
   f^{-1}\iota_Y^{-1}\M_Y \ar[rr] \ar[dd] \ar[ru]  &  &  f^{-1}\M_Y 
   \ar[dd] \ar[ru]  &  \\ 
   &  \iota_X^{-1}\O_X \ar[rr]  &  &  \O_X  \\
   \iota_X^{-1}\M_X \ar[rr] \ar[ru]  &  & \M_X \ar[ru]  &
  }
\end{xy}\]
In fact, commutativity on the (1)~bottom, (2)~top, (3)~right
(4)~left (5)~back and (6) front faces follows from the assumptions
that (1)~$(X,\M_X)$ is a real log scheme, (2)~$(Y,\M_Y)$ is a real
log scheme, (3)~$f$ is a morphism of log schemes, (4)~$\iota_X^{-1}$
applied to the right face plus the identity $f\circ \iota_X=
\iota_Y\circ f$, (5)~$f$ induces a real morphism on the underlying
schemes and (6)~$f$ is a real morphism of real log structures.
\end{remark}

Given a real log scheme $(X,\M_X)$ with $\alpha:\M_X\to \O_X$ the structure
homomorphism, for any geometric point $\bar x\to X$ we have a commutative
diagram
\[\begin{CD}
\M_{X,\bar x}@>{\iota^\flat}>>
\M_{X,\iota(\bar x)}@>{\iota^\flat}>>
\M_{X,\bar x}\\
@V\alpha_{\bar x}VV@V\alpha_{\iota(\bar x)}VV@V\alpha_{\bar x}VV\\
\O_{X,\bar x}@>{\iota^\sharp}>>
\O_{X,\iota(\bar x)}@>{\iota^\sharp}>>
\O_{X,\bar x}.
\end{CD}\]
The compositions of the maps in the two horizontal
sequences are the identity on $\M_{X,\bar x}$ and on $\O_{X,\bar
x}$, respectively. For the next result recall that if $X$ is a
pure-dimensional scheme and $D\subset X$ is a closed subset of
codimension one, then the subsheaf  $\M_{(X,D)}\subset \O_X$ of
regular functions with zeros contained in $D$ defines the
\emph{divisorial log structure} on $X$ associated to $D$.

\begin{proposition}\label{Prop: divisorial log structure}
Let $X$ be a pure-dimensional scheme, $D\subset X$ a closed
subset of codimension one and let $\M_X=\M_{(X,D)}$ be the
associated divisorial log structure. Then a real structure $\iota$
on $X$ lifts to $\M_X$ iff $\iota(D)=D$. Moreover, in this case the
lift $\iota^\flat$ is uniquely determined as the restriction of
$\iota^\sharp$ to $\M_{(X,D)}\subset \O_X$.
\end{proposition}

\begin{proof}
Let $\iota: X\to X$ be a real structure on $X$ with $\iota(D)=D$.
Then $\iota(X \setminus D)= X \setminus D$ and hence $\iota^\sharp$
restricts to an isomorphism $\varphi:\iota^{-1}\O_{X\setminus
D}^\times \to \O_{X\setminus D}^\times$. By definition of
$\shM_{(X,D)}$, $\varphi$ induces an isomorphism
$\iota^\flat:\iota^{-1}\shM_{(X,D)} \to \shM_{(X,D)}$. Hence we get
a real structure $(\iota,\iota^\flat): (X,\M_{(X,D)})  \to
(X,\M_{(X,D)})$ on $(X,\M_{(X,D)})$ lifting $\iota$.

Conversely, let the real structure $\iota: X\to X$ lift to
$(\iota,\iota^\flat): (X,\M_{(X,D)})\lra (X,\M_{(X,D)})$. In other words,
there exists a morphism $\iota^\flat: \iota^{-1}\M_{(X,D)} \lra
\M_{(X,D)}$ making the following diagram commute.
\begin{equation}
\label{Eqn: tilde iota exists}
\begin{CD}
\iota^{-1}\M_{(X,D)} @>{\iota^{-1}\alpha}>> \iota^{-1}\O_X\\
@V{\iota^\flat}VV @VV{\iota^\sharp}V \\
\M_{(X,D)} @>{\alpha}>>      \O_X 
\end{CD}
\end{equation}
Let $D=\bigcup_\mu D_\mu$ be the decomposition into irreducible components.
Since $\iota^2=\id_X$ it suffices to show $\iota(D)\subset D$, or
$\iota(D_\mu)\subset D$ for every $\mu$. Fix $\mu$ and let $U\subset X$ be an
affine open subscheme with $U\cap D_\mu\neq\emptyset$. Let $f\in \O_X(U)$
be a non-zero divisor with $D\subset V(f)$. Then $U\cap
D_\mu\subset U\cap D\subset V(f)$. Write $V(f)=(D_\mu\cap U)\cup E$ with
$E\subset V(f)$ the union of the irreducible components of $V(f)$ different from
$D_\mu$. Replacing $U$ by $U\setminus E$ we may assume $V(f)= U\cap D_\mu$. Note
that $U$ may not be affine anymore, but this is not important from now on.

Taking sections of Diagram~\eqref{Eqn: tilde iota exists} over
$\iota^{-1}(U)$ shows that $f\circ\iota= \iota^\sharp(f)$ lies in
$\M_{(X,D)}\big(\iota^{-1}(U)\big)\subset \O_X
\big(\iota^{-1}(U)\big)$. By the definition of $\shM_{(X,D)}$ this
implies $V(f\circ\iota)\subset D$. But also
\[
V(f\circ\iota)= \iota^{-1}\big(V(f)\big)
= \iota^{-1}(U\cap D_\mu) =\iota(U\cap D_\mu).
\]
Taken together this shows that $\iota(U\cap D_\mu)\subset D$.
Since $U$ is open with $U\cap D_\mu\neq\emptyset$ we obtain the
desired inclusion $\iota(D_\mu)\subset D$.
\end{proof}
\begin{proposition}\label{Prop: strict morphism}
Let $f:(X,\M_X)\to (Y,\M_Y)$ be strict and assume that the morphism
$\ul f$ of the underlying schemes is compatible with real structures
$\iota_X$ on $X$ and $\iota_Y$ on $Y$. Then for any real structure
$\iota^\flat_Y$ on $\M_Y$ lifting $\iota_Y$ there exists a unique
real structure $\iota^\flat_X$ on $\M_X$ lifting $\iota_X$ and
compatible with $f$.
\end{proposition}

\proof
By strictness we can assume the log structure $\shM_X$ on $X$ is the
pull-back log structure $f^*\shM_Y = f^{-1}\shM_Y
\oplus_{f^{-1}\O_Y^\times} \O_X^\times$. Hence,
\[
\iota^{-1}_X\shM_X =
\iota^{-1}_Xf^{-1}\shM_Y \oplus_{{\iota_X^{-1}}f^{-1}\O_Y^\times}
\iota^{-1}_X \O_X^\times
=f^{-1}\iota^{-1}_Y\shM_Y \oplus_{f^{-1}{\iota_Y^{-1}}\O_Y^\times}
\iota^{-1}_X \O_X^\times.
\]
Now for a lift $\iota_X^\flat: \iota_X^{-1}\M_X\to \M_X$ of
$\iota_X^\sharp$ compatible with $f$, the composition
\[
\varphi:f^{-1}\iota^{-1}_Y\shM_Y \stackrel{\iota_X^{-1}f^\flat}{\lra} \iota_X^{-1}\M_X
\stackrel{\iota_X^\flat}{\lra} \M_X=f^*\M_Y
\]
factors over $f^{-1}\iota_Y^\flat: f^{-1}\iota^{-1}_Y\shM_Y \to
f^{-1}\M_Y$ and is hence determined by $f$ and $\iota_Y^\flat$.
Similarly, the composition
\[
\psi:\iota_X^{-1}\O_X^\times \lra \iota_X^{-1}\M_X
\stackrel{\iota_X^\flat}{\lra} \M_X=f^*\M_Y
\]
factors over $\iota_X^\sharp: \iota_X^{-1}\O_X^\times\to  \O_X^\times$
and thus is known by assumption. Since $\ul f$ is a real morphism of
real schemes, $\varphi$ and $\psi$ agree on
$f^{-1}\iota_Y^{-1}\O_Y^\times$. Hence the unique existence of
$\iota_X^\flat$ with the requested properties follows from the
universal property of the fibered sum.
\qed
\medskip

Explicit computations are most easily done in charts adapted to the real
structure. For simplicity we provide the following statements for log structures
in the Zariski topology, the case sufficient for our main application to toric
degenerations. Analogous statements hold in the \'etale or analytic topology. 

\begin{definition}
\label{Def: Real charts}
Let $(X,\M_X)$ be a log scheme with a real structure
$(\iota_X,\iota_X^\flat)$. A  chart $\beta: P\to \Gamma(U,\M_X)$ for
$(X,\M_X)$ is called a \emph{real chart} if (1)~$\iota_X(U)=U$ and
(2)~there exists an involution $\iota_P:P\to P$ such that for all
$p\in P$ it holds $\beta\big(\iota_P(p)\big) =
\iota_X^\flat\big(\beta(p)\big)$.
\end{definition}

\begin{example}
An involution $\iota_P$ of a toric monoid $P$ induces an
antiholomorphic involution on $\CC[P]$ by mapping $\sum_p a_p z^p$
to $\sum_p \ol{a_p} z^{\iota_P(p)}$. The induced real structure on
the toric variety $X_P=\Spec\CC[P]$ permutes the irreducible
components of the toric divisor $D_P\subset X_P$ and hence, by
Proposition~\ref{Prop: divisorial log structure} induces a real
structure on $(X_P,\M_{(X_P,D_P)})$. We claim the canonical
toric chart
\[
\beta: P\lra \Gamma(X_P,\M_{(X_P,D_P)}),\quad
p\longmapsto z^p
\]
is a real chart. Indeed, for any $p\in P$ we have $\beta(\iota_P(p))
= z^{\iota(p)} = \iota_{X_P}^\sharp(z^p)= \iota_{X_P}^\flat(z^p)$,
the last equality due to Proposition~\ref{Prop: divisorial log
structure}. 
\end{example}

Real charts may not exist, a necessary condition being that $X$ has
a cover by affine open sets that are invariant under the real
involution $\iota_X$. This is the only obstruction:

\begin{lemma}
\label{Lem: Existence of real charts}
Let $(X,\M_X)$ be a real log scheme with involution $\iota_X$. Let
$U\subset X$ be a $\iota_X$-invariant open set supporting a chart
$\beta: P\to \Gamma(U,\M_X)$. Then there also exists a real chart
$\beta': P'\to \Gamma(U,\M_X)$ for $\M_X$ on $U$.
\end{lemma}
\begin{proof}
Let
\[
\tilde\beta: P\oplus P\lra \Gamma(U,\M_X),\quad
\tilde\beta(p,p')= \beta(p)\cdot \iota_X^\flat\big(\beta(p')\big) \,,
\]
and $\tilde{\beta}^\gp: (P \oplus P)^\gp \rightarrow \Gamma(U,\M_X^\gp)$ the induced map at the group level. Let 
$P' \subset (P \oplus P)^\gp$ be the submonoid of $(p,p') \in (P \oplus P)^{\gp}$
such that $\tilde{\beta}^\gp((p,p')) \in \Gamma(U,\M_X)$, and 
$\beta' \colon P' \rightarrow \Gamma(U,\M_X)$
the map given by the restriction of $\tilde{\beta}^{\gp}$
to $P'$. Since $\tilde\beta$ restricts to $\beta$ on
the first summand of $ P \oplus P$, the map $\tilde{\beta}^\gp$
induces a surjective map 
$(P \oplus P)^\gp \rightarrow \Gamma(U, \M_X^\gp /\O_X^{\times})$, and so $\beta'$ is a chart for 
$\M_X$ on $U$ by Lemma 2.10 of \cite{Kato}. 
We claim that $\beta'$ is a real chart.
As $\iota_X^{\flat}$ preserves 
$\M_X \subset \M_X^\gp$, and 
$\tilde{\beta}^\gp((p',p))=\beta(p') \cdot \iota_X^\flat(\beta(p))
=\iota_X^\flat(\beta(p) \cdot \iota_X^\flat(\beta(p'))=\iota_X^\flat (\tilde{\beta}^\gp(p,p'))$, the involution $(p,p') \mapsto (p',p)$ on $(P \oplus P)^{\gp}$
preserves $P'$ and so defines an involution 
$\iota_{P'} \colon P' \rightarrow P'$.
By construction, for every $(p,p')\in
P'$ we have $\beta'(\iota_{P'}(p,p'))
=\beta'((p',p))=\iota_X(\beta'(p,p'))$,
verifying the condition for a real chart.
\end{proof}

Note that if $X$ is a separated scheme, real charts always exist at
any point $x$ in the fixed locus of $\iota_X$. In fact, take any
chart defined in a neighbourhood $U$ of $X$, restrict to $U\cap
\iota_X(U)$, still an affine open set by separatedness, and apply
Lemma~\ref{Lem: Existence of real charts}.
\medskip

\begin{proposition}\label{Prop: compatibility with base change}
Cartesian products exist in the category of real log schemes.
\end{proposition}
\begin{proof}
Let $(X,\M_X)$, $(S,\M_S)$, $(T,\shM_T) $ be real log schemes
endowed with morphisms $f:(X,\shM_X) \to (T,\shM_T)$ and $g:
(S,\shM_S) \to (T,\shM_T)$. Then the fibre product in the category
of log schemes $( S \times_T X, \M_{S\times_T X})$
fits into the following cartesian diagram.
\begin{eqnarray}
    \xymatrix{
        ( S\times_T  X,\shM_{ S\times_T  X}) \ar[r]^-{p_X}
		\ar[d]^{p_S} &
		(X,\shM_X) \ar[d]^{f} \\
        (S,\shM_S) \ar[r]^{g}    & (T,\shM_T) }
\end{eqnarray}
The log structure on the fibre product  $S\times_T  X$ is given by
$\shM_{ S\times_T  X}=p^*_X\shM_X \oplus_{p^*_T\shM_T }
p^*_S\shM_S$. By the universal property of the fibered coproduct
the existence of real structures on $(X,\shM_X), (S,\shM_S)$ and
$(T,\shM_T) $ ensures the existence of a real structure on $(  S
\times_{T} X,\shM_{ S\times_T  X})$.
\end{proof}

Note that in general the fibred coproduct of fine log structures
$p^*_X\shM_X \oplus_{p^*_T\shM_T } p^*_Y\shM_Y$ is only coherent \cite[\S II, Prop.~2.1.1]{Ogus}, but not even
integral. To take the fibred product in the category of fine log schemes
requires the further step of integralizing $( S\times_T X,\shM_{ S\times_T X})$.
Given a monoid $P$ with integralization $P_{\mathrm{int}}$ and a chart $U\to
\Spec\ZZ[P]$ for a log scheme $(U,\M_U)$, the integralization of $(U,\M_U)$ is
the closed subscheme $U\times_{\Spec\ZZ[P]} \Spec\ZZ[P^{\mathrm{int}}]$ of $U$
with the log structure defined by the chart $U\to\Spec\ZZ[P]\to \Spec
\ZZ[P^{\mathrm{int}}]$. So integralization may change the underlying scheme. A similar
additional step is needed for staying in the category of saturated log schemes.
Fortunately, we are only interested in the case that $g$ is strict, and in this
case the fibre products in all categories agree. See \cite{Ogus}, Ch.III,
\S2.2.1, for details.

\begin{example}\label{Expl: relation to degenerations}
Let $S$ be the spectrum of a discrete valuation ring with residue field $\CC$
and $\pi:\shX \to S$ be a flat morphism. Let $0 \in S$ be the closed point,
$X_0= \pi^{-1}(0)$ and consider $\delta$ as a morphism of log schemes with
divisorial log structures $\pi: (\shX,\M_{(\shX, X_0)})\to (S,\M_{(S,0)})$. If
$\pi$ commutes with real structures on $\shX$ and $S$, then by Proposition
\ref{Prop: divisorial log structure}, the morphism $\pi$ is naturally a real
morphism of real log schemes. Taking the base change by the strict morphism
$(\Spec \CC, \NN\oplus\CC^\times) \to (S,\M_{(S,0)})$, Proposition \ref{Prop:
compatibility with base change} leads to a real log scheme $(X_0, \shM_{X_0})$
over the standard log point $O^\ls= (\Spec \CC, \NN\oplus\CC^\times)$
with its obvious real structure.
\end{example}

\section{Lifting real involutions to the Kato--Nakayama space}
\label{Sect: Real structure on KN}
In this section we study the additional structure on the Kato-Nakayama space of a log space
induced by a real structure.
\begin{definition}
\label{Def: Lifted real involution}
Let $(X,\M_X)$ be a real log space with
\[
\iota_X=(\ul\iota_X,\iota^\flat_X):(X,\M_X)\to (X,\M_X)
\]
denoting its real involution. Let $\Pi^\ls$ be the polar log point with real
structure $\iota_\Pi$ given by the involution
\[
\iota^\flat_\Pi(r,e^{i\varphi})= (r,e^{-i\varphi})
\]
lifting the standard conjugation involution on $\CC$. We call the
map
\begin{equation}
\label{Eqn: KN involution}
\iota_X^\KN: X^\KN\lra X^\KN,\quad
\big( f:\Pi^\ls\to (X,\M_X)\big) \lra \iota_X\circ f\circ
\iota_\Pi
\end{equation}
the \emph{lifted real involution} on the Kato-Nakayama space.
\end{definition}

\begin{proposition}
\label{Prop: The real involution lifts to X^KN}
Let $(X,\M_X)$ be a real log space.
\begin{itemize}
\item[(1)] The lifted real involution $\iota_X^\KN$ is continuous
and is compatible with the underlying real involution
$\iota_X$ of $X$ under the projection $\pi: X^\KN\to X$.
\item[(2)] The
real locus of $X$ has a canonical lift to $X^\KN$.
\end{itemize}
\end{proposition}
\begin{proof}
The first part is immediate from the definition \ref{Def: Lifted real involution}, and the second part is a direct consequence of the first part. 
\end{proof}
\begin{definition}
\label{Def: Real locus in X^KN}
Let $(X,\M_X)$ be a real log space and $\iota_X^\KN: X^\KN\to
X^\KN$ the lifted real involution. We call the fixed point set of
$\iota_X^\KN$ the \emph{real locus of $X^\KN$}, denoted
$X^\KN_\RR\subset X^\KN$.
\end{definition}

Studying the real locus $X^\KN_\RR\subset X^\KN$ allows us to understand real structures in toric degenerations, since under the presence of a real structure, Theorem~\ref{Thm: KN for X0 in simple case} generalizes to the following statement for the pair $\big(X_0^\KN, X_{0,\RR}^\KN\big)$.
\begin{proposition}
\label{Prop: pairs}
Let $(X_0,\shM_{X_0})$ be a real toric log Calabi--Yau, with a log smooth morphism $\delta:(X_0,\shM_{X_0}) \to \O^\ls$. Then the fibre bundle
structure for $\delta^\KN: X_0^\KN\to S^1$ can be chosen to respect the real
locus $X_{0,\RR}^\KN\subset X_0^\KN$ in the sense that there exists local trivializations in which the real involution becomes the pullback of the standard real involution on the base. The analogous statement holds for $\shX^\KN\to D^\KN$ assuming
there is a real structure on $(\shX,\shM_\shX)$ extending the real structure on
$(X_0,\shM_{X_0})$. In particular, we obtain a fibre bundle structure on $\shX^\KN_\RR \to D^\KN_\RR$. 
\end{proposition}

\begin{proof}
We have to verify that under the presence of a real structure, the proof of
\cite{NO}, Theorem~5.1, extends to the required stronger statement for pairs.
The proof first establishes a result for a toric model (\cite{NO}, Theorem~0.2)
and then globalizes by a result by Siebenmann (\cite{Siebenmann},
Corollary~6.14). Now the proof for toric models indeed works on the positive
real locus and leaves the phase (argument) untouched, thus readily gives the
required local trivialization respecting the real locus. For the globalization,
there does exist a stratified version of \cite{Siebenmann}, Corollary~6.14, the
so-called ``respectful version'' (\cite{Siebenmann}, First Complement~6.10 and
Remark~6.16). Applying this result to the pairs $X_{0,\RR}^\KN \subset X_0^\KN$
and $\shX_\RR^\KN \subset \shX^\KN$, respectively, then yields the assertions.
\end{proof}
\bigskip

In the remaining part of this section we study the real locus $X^\KN_\RR\subset X^\KN$. Recall from the description~\eqref{Eqn: (x,theta) description of KN space}, there is a canonical projection map $\pi: X^\KN\to X$. Let 
\[\pi_\RR: X^\KN_\RR\to X_\RR \]
denote its restriction to the real locus.

If $x\in X_\RR$ then the involution $\iota_X^\flat$ induces an involution on the stalk $\M_{X,x}$ as well as on the stalk of the ghost sheaf $\ol \M_{X,x}= \M_{X,x}/\O^\times_{X,x}$. We describe $\pi_\RR^{-1}(x)$ in the next proposition.

\begin{proposition}
\label{Prop: fibres of pi_RR}
Let $(X,\M_X)$ be a real log space and $x\in X_\RR$. Consider the Kato--Nakayama space $X^{\KN}$ described as in ~\eqref{Eqn: (x,theta) description of KN space} and let $\pi: X^\KN\to X$ be the projection map. Then,
\begin{enumerate}
\item 
A point $(x,\theta)$ of the fibre $\pi^{-1}(x)$,
where $\theta\in \Hom(\M_{X,x}^\gp,U(1))$ is as in \eqref{Eqn: (x,theta) description of KN space}, lies
in $X_\RR^\KN$ if and only if $\theta\circ\iota_{X,x}^\flat
=\ol\theta$, where $\ol\theta$ is the complex conjugate of $\theta$.
\item
If $\iota_X^\flat$ induces a trivial action on $\ol\M_{X,x}$, then
$\pi_\RR^{-1}(x)$ is canonically a torsor for the group
$\Hom(\ol\M_{X,x}^\gp, \ZZ/2\ZZ)$.
\end{enumerate}
\end{proposition}

\begin{proof}
(1)\ \ Let $\tilde x\in \pi^{-1}(x)$ be given by a log morphism $f:\Pi^\ls\to
(X,\M_X)$ with image $x$. By \eqref{Eqn: KN involution} we have $\tilde x\in
X_\RR^\KN$ if and only if 
\[ f\circ\iota_\Pi=\iota_X\circ f \]
Writing $\tilde
x=(x,\theta)$ as in \eqref{Eqn: (x,theta) description of KN space},
we have
\begin{eqnarray}
\label{compare}
\nonumber
f\circ\iota_\Pi & = & (x,\ol\theta) \\
\nonumber
\iota_X\circ f & = &
(x,\theta\circ\iota_{X,x}^\flat)
\nonumber
\end{eqnarray}
The statement follows by comparison of the above two equations.
\smallskip

\noindent
(2)\ ~Denote by $\kappa: \M_X\to \ol\M_X$ the quotient homomorphism. We
define the action of $\tau\in \Hom(\ol\M_{X,x}^\gp,\ZZ/2\ZZ)$ on
$\pi^{-1}(x)$ by
\begin{equation}
\label{Eqn: Hom(ol M^gp,ZZ/2)-action}
\big(\tau\cdot \theta\big)(s)= \tau( \kappa_x(s))\cdot\theta(s)
\end{equation}
for $s\in\M_{X,x}$ and $\theta\in \Hom(\M_{X,x}^\gp,U(1))$. In this definition
we take $\tau(\kappa_x(s))\in \O^\times_{X,x}$ by means of the identification
$\ZZ/2\ZZ=\{\pm1\}$. By (1), $\theta$ defines a point in $X_\RR^\KN$ iff
$\theta\circ\iota_{X,x}^\flat =\ol\theta$. If $\iota_X^\flat$ induces a trivial
action on $\ol\M_{X,x}$, hence on $\ol\M_{X,x}^\gp$, then this condition
is preserved by the action of $\Hom(\ol\M_{X,x}^\gp,\ZZ/2\ZZ)$. Indeed, by the triviality of the induced action on $\ol\M_{X,x}^\gp$ and since $\tau$ takes real values, it holds
\[
\tau\circ\kappa_x\circ\iota^\flat_{X,x}= \tau\circ \kappa_x = \ol{\tau\circ\kappa_x}.
\]
Thus $\tau\cdot\theta$ fulfills the condition stated in (1):
\[
(\tau\cdot\theta)\circ\iota_{X,x}^\flat =
\big((\tau\circ\kappa_x)\cdot \theta\big)\circ \iota_{X,x}^\flat=
\ol{(\tau\circ\kappa_x)}\cdot \ol\theta = \ol{\tau\cdot\theta}.
\]
Hence \eqref{Eqn: Hom(ol M^gp,ZZ/2)-action} defines an action of
$\Hom(\ol\M_{X,x}^\gp,\ZZ/2\ZZ)$ on $\pi_\RR^{-1}(x)$. To check the claimed
torsor property, let $\theta_1,\theta_2\in \Hom(\M_{X,x}^\gp,U(1))$ define
elements in $\pi_\RR^{-1}(x)\subset X_\RR^\KN$. Then by \eqref{Eqn: (x,theta)
description of KN space} we have for $h\in \O^\times_{X,x}$,
\[
\theta_1(h)=\frac{h}{|h|}=\theta_2(h).
\]
Hence $\theta_1\cdot\theta_2^{-1}$ takes trivial values on $\O^\times_{X,x}$.
Denote by $\tilde\tau:\ol\M_{X,x}^\gp\to U(1)$ the induced map with
$\tilde\tau\circ\kappa_x=\theta_1\cdot\theta_2^{-1}$. Note that $\tilde\tau$ takes real
values, hence defines a map $\tau:\ol\M_{X,x}^\gp\to \ZZ/2\ZZ$:
\[
\ol{\theta_1\cdot\theta_2^{-1}}= \ol\theta_1\cdot\ol\theta_2^{-1}=
(\theta_1\circ\iota_{X,x}^\flat)\cdot (\theta_2\circ\iota_{X,x}^\flat)^{-1}=
(\theta_1\cdot\theta_2^{-1})\circ\iota_{X,x}^\flat=
(\theta_1\cdot\theta_2^{-1}).
\]
The last equality follows since $\theta_1\cdot\theta_2^{-1}$ factors over
$\ol\M_{X,x}^\gp$, on which $\iota_{X,x}^\flat$ induces a trivial action. 

Now $\tau\in\Hom(\ol\M_{X,x}^\gp,\ZZ/2\ZZ)$ fulfills
$\theta_1=\tau\cdot\theta_2$ by definition, showing that the action is simply
transitive.
\end{proof}

\begin{corollary}
Let $(X,\M_X)$ be a fine saturated real log space. If the stalk of $\ol\M_X^\gp$ at $x\in X_\RR$ has rank~$r$,
and $\iota_x^\flat$ induces a trivial action on $\ol\M_{X,x}$, then
$\pi^{-1}(x)$ consists of $2^r$ points. 
\end{corollary}
\begin{proof}
This is an immediate consequence of Proposition~\ref{Prop: fibres of
pi_RR}.
\end{proof}

Without the assumption of a trivial action on the ghost sheaf
$\ol\M_{X,x}$, the fibre of $X_\RR^\KN\to X_\RR$ can be
non-discrete, as we will see in the following example.

\begin{example}
Let $X$ be a complex variety with a real structure $\ul\iota_X$ and a
$\ul\iota_X$-invariant simple normal crossings divisor $D$ with two irreducible
components $D_1$, $D_2$. Assume there is a real point $x\in D_1\cap D_2$ and
$\ul\iota_X$ exchanges the two branches of $D$ at $x$. Denote by $\iota_X^\flat$
the induced real structure on $\M_X=\M_{(X,D)}$ according to
Proposition~\ref{Prop: divisorial log structure}. Then $P=\ol\M_{X,x}=\NN^2$ and
$\ol\iota_{X,x}^\flat(a,b)= (b,a)$. The action extends to an
involution $\iota_M$ of $M=P^\gp=\ZZ^2$. In the present case there is a subspace
$M'\subset M$ with $M'\oplus\iota_M(M')= M$, e.g.\ $M'=\ZZ\cdot(1,0)$. Then
$\theta: M\to U(1)$ can be prescribed arbitrarily on $M'$ and extended uniquely
to $M$ by enforcing $\theta\circ\iota_M= \ol\theta$. Thus in the present case
$\pi_\RR^{-1} (x)= \Hom(\ZZ,U(1))= S^1$.

In the general case, when $\ol\M_{X,x}$ is a fine monoid and $\ol\M_{X,x}^{gp}$ is torsion-free, by \cite[II, Prop.~2.3.7]{Ogus} we can write
$\M_{X,x}^\gp= M\oplus \O_{X,x}^\times$ with $M$ a finitely generated abelian
group and such that $\iota_{X,x}^\flat$ acts by an involution $\iota_M$ on $M$
and by $\iota_{X,x}^\sharp$ on $\O_{X,x}^\times$. Then $\pi^{-1}(x)=
\Hom(M,U(1))$ is a disjoint union of tori, one copy of $\Hom(M/T, U(1))$ for
each element of the torsion subgroup $T\subset M$. The fibres $\pi_\RR^{-1}(x)$
for $x\in X_\RR$ are the preimage of the diagonal torus of the map
\[
\Hom(M,U(1))\lra \Hom(M,U(1))\times\Hom(M,U(1)),\quad
\theta\longmapsto (\theta\circ\iota_M, \ol\theta).
\]
\end{example}

Although the rank of $\ol\M_X^\gp$ varies, in toric situations the projection map $X^{KN} \to X$ interacts nicely with the momentum map description on $X$, ensuring the smoothness of $X^\KN_\RR$. We describe $X^\KN_\RR$ in such situations in the following Proposition.

\begin{proposition}
\label{Prop: X_RR^KN of toric}
Let $(X,\M_X)$ be a toric variety with its toric log structure and $\mu:X\to
\Xi\subset M_\RR$ a momentum map. Let $\iota_X$ be the unique real structure on
$(X,\M_X)$ lifting the standard real structure according to
Proposition~\ref{Prop: divisorial log structure}. Then there is a canonical
homeomorphism
\[
X_\RR^\KN \simeq \Xi\times \Hom(M,\ZZ/2\ZZ),
\]
with the projection to $\Xi$ giving the composition $\mu\circ\pi_\RR:
X_\RR^\KN\to \Xi$.
\end{proposition}
\begin{proof}
Recall the section $ s_0:\Xi\to X$ of the momentum map with image
$X_{\ge0}\subset X_\RR$ from Definition \ref{Def: Momentum map}. For $a\in\Xi$, Proposition~\ref{Prop: KN of
toric variety} identifies 
\[\pi^{-1}(\mu^{-1}(a))\subset X^\KN\] with pairs
$(\lambda\cdot s_0(a),\lambda)\in X\times \Hom(M,U(1))$. The
action of $\iota_X^\KN$ on this fibre is
\[
(\lambda\cdot s_0(a),\lambda)\longmapsto
(\ol\lambda\cdot s_0(a),\ol\lambda).
\]
Thus $(\lambda\cdot s_0(a),\lambda)$ gives a point in $X_\RR^\KN$
if and only if $\lambda=\ol\lambda$. This is the case iff
$\lambda$ takes values in $\RR\cap U(1)=\{\pm1\}$, giving the result.
\end{proof}


\section{Real loci in toric degenerations}
\label{Subsect: X_0^KN real}

Recall that
any complex toric variety is defined over $\ZZ$ and hence has a canonical real
structure by factoring the base change $\Spec\CC\to\Spec\ZZ$ over $\Spec\RR$.
Similarly, the standard log point has a canonical real structure. Let $(X_0,\M_{X_0})$ be a toric log Calabi-Yau space \cite[Defn.~4.3]{logmirror1}. We call $(X_0,\M_{X_0})$ \emph{standard real} if it has a
real structure compatible with the canonical real structure on the
standard log point and inducing the canonical real structure on its
toric irreducible components. Since the morphism $\delta: (X_0,\M_{X_0})\to
O^\ls$ is strict at the generic points of the irreducible components of $X_0$,
and since any section of $\M_{X_0}$ that is supported on higher codimensional
strata is trivial (constant~$1$), there is at most one such real structure on
$(X_0,\M_{X_0})$. Standard real structures appear to be the only class of real structures on toric
log-Calabi-Yau spaces that exist in great generality. While other real
structures, for example those lifting an involution on $B$, should be extremely
interesting in more specific situations, we therefore restrict the following
discussion to standard real structures, except for the example of a
toric degeneration of local $\PP^2$ discussed in \S\ref{Subsect: local PP2}.

A toric log Calabi--Yau space $(X_0,\M_{X_0})$ with given intersection complex $(B,\P)$, in good cases can be uniquely defined from additional combinatorial data given by \emph{lifted open gluing data} by Theorem \ref{thm: lifted gluing}. Lifted open gluing data, which we reviewed in \S\ref{Sec: open gluing data}, is given by elements $s\in H^1 (B,\iota_*\check\Lambda
\otimes\CC^\times)$.
\begin{proposition}
\label{Prop: Gluing data and slab function for the standard real structure}
Let $(X_0,\M_{X_0})$ be a polarized toric log Calabi-Yau space with intersection complex
$(B,\P)$. Then there is a standard real structure on
$(X_0,\M_{X_0})$ if and only if 
there exist lifted open gluing data $s=(s_v)_{v\in \tau}$ with
$X_0\simeq X_0(B,\P,s)$ such that $s_v \in \check{\Lambda}_\sigma \otimes_\ZZ \mathbb{R}^*$.
\end{proposition}
\begin{proof}
The proof is by inspection of the arguments in \cite{logmirror1}. If $s=(s_{\omega\tau})_{\omega,\tau}$ are open gluing data
taking values in $\RR^\times\subset\CC^\times$, the construction of
$X_0(B,\P,s)$ by gluing affine toric varieties in \cite{logmirror1},
Definition~2.28 readily shows that the real structures on the
irreducible components induce a real structure on $X_0(B,\P,s)$.
 
Conversely, given $(X_0,\M_{X_0})$ with a standard real structure, Theorem~4.14
in \cite{logmirror1} constructs open gluing data $s$ and an isomorphism
$X_0\simeq X_0(B,\P,s)$. The construction has two steps. First, $X_0$ being
glued from toric varieties, there exist \emph{closed} gluing data $\ol s$
inducing this gluing. If $X_0$ admits a standard real structure, $\ol s$
automatically takes real values. In a second step one shows that the closed
gluing data are the image of open gluing data as in \cite{logmirror1},
Lemma~2.29 and Proposition~2.32,2. This step uses a chart for the log structure
at a zero-dimensional toric stratum $x\in X_0$. In view of the given real
structure on $(X_0,\M_{X_0})$, this chart can be taken real (Lemma~\ref{Lem:
Existence of real charts}). With this choice of chart, the construction of open
gluing data in the proof of \cite{logmirror1}, Theorem~4.14, indeed produces
real open gluing data.
\end{proof}
We refer to lifted open gluing data $s$ as in Proposition \ref{Prop: Gluing data and slab function for the standard real structure} as \emph{real lifted open gluing data}. It follows that, in terms of lifted gluing data, the existence of a
standard real structure has a simple cohomological formulation.
\begin{corollary}
\label{Cor: Standard real structure cohomological}
Assuming $(B,\P)$ positive and simple, then the toric log
Calabi-Yau space $(X_0,\M_{X_0})$ defined by lifted gluing data $s\in
H^1(B,\iota_* \check\Lambda\otimes\CC^\times)$ is standard real if
and only if $s$ lies in the image of
\[
H^1(B, \iota_*\check\Lambda\otimes\RR^\times) \lra 
H^1(B,\iota_* \check\Lambda\otimes\CC^\times).
\]
\end{corollary}

\begin{proof}
This follows again by inspection of the corresponding results in
\cite{logmirror1}, here Theorems~5.2 and~5.4.
\end{proof}

Recall from \S \ref{Sec: open gluing data} that lifted open gluing data particularly defines slab functions describing the log structure on the central fiber. We arrive at the following characterization of real lifted open gluing data in terms of these functions.

\begin{proposition}
\label{Prop: real slab functions}
Let $(X_0,\M_{X_0})$ be a polarized toric log Calabi-Yau space with intersection complex
$(B,\P)$. Then $(X_0,\M_{X_0})$ is defined by real lifted open gluing data if and only if the slab functions $f_{\rho,x} \in \CC[\Lambda_\rho]$
describing the log structure $\M_{X_0}$ are
defined over $\RR$, that is $f_{\rho,x}\in \RR[\Lambda_\rho]$ for
any $\rho\in\P$ of codimension one and $x$ a connected component of $\rho \setminus \Delta$.
\end{proposition}
\begin{proof}
 The relation between the slab functions $f_{\rho,x}$ and charts for the log
structure is given in \cite{logmirror1}, Theorem~3.22. At a zero-dimensional
toric stratum $x\in X_0$ the description in terms of open gluing data yields an
isomorphism of an open affine neighbourhood in $\Spec X_0$ with
$\Spec\CC[P]/(z^{\rho_P})$, with $P=\ol\M_{X,x}$ and $\rho_P\in P$ corresponding
to the deformation parameter $t$. The facets of $P$ not
containing\footnote{This condition is non-trivial only if $x\in\partial
B$} $\rho_P$ are in one-to-one correspondence with the irreducible components
of $X_0$ containing $x$. Now charts for the log structure on this open subset
are of the form
\[
P\lra \CC[P]/(z^{\rho_P}),\quad p\longmapsto h_p\cdot z^p
\]
with $h_p$ an invertible function on $V(p)$, the closure of the open subset
$(z^p\neq 0)\subset \Spec\CC[P]$. The equation describing this chart in terms of
functions on codimension one strata expresses the slab function
(written $\xi_\omega(h)$ in \cite{logmirror1}) as a quotient of multiplicative
functions (written $g_v$ in \cite{logmirror1}) defined in terms of $h_p$. This
equation shows that describing a real chart via real open gluing data yields
real slab functions $f_{\rho,x}$.

Conversely, given real open gluing data and real slab functions, the real
structure on $\Spec \CC[P]$ induces the involution $\iota_{X_0}^\flat$ defining
a standard real structure on $(X_0,\M_{X_0})$.
\end{proof}

\begin{remark}
\label{Rem: Real smoothing algorithm}
It is worthwhile pointing out that real structures on
$(X_0,\M_{X_0})$ are compatible with the smoothing algorithm of
\cite{affinecomplex} in the following way. Assume that
$(X_0,\M_{X_0})$ is a toric log Calabi-Yau space for which the
smoothing algorithm of \cite{affinecomplex} works, for example with
associated intersection complex $(B,\P)$ positive and simple. Assume
that $(X_0,\M_{X_0})$ has a real structure, not necessarily
standard. The real involution then induces a possibly non-trivial
involution on the intersection complex $(B,\P)$. But in any case,
$(X_0,\M_{X_0})$ has a description by open gluing data
$s=(s_{\omega\tau})$ and slab functions $f_{\rho,x}$ with the real
involution lifting to an action on these data. By the strong
uniqueness of the smoothing algorithm it is then not hard to see
that the real involution extends to the constructed family $\shX\to
\Spec\CC\lfor t\rfor$.
\end{remark}
\subsection{The real locus in the Kato--Nakayama space}
\label{Subsect: Topology in standard case}
Let us now assume we have a standard real structure on $(X_0,\M_{X_0})$. We will
investigate the topology of the real locus $X_{0,\RR}^\KN\subset X_0^\KN$, the
fixed locus of the lifted real involution of $X_0^\KN$ from Definition~\ref{Def:
Lifted real involution}. First, since $(O^\ls)^\KN_\RR=\{\pm1\}$, the real locus
of $X_{0,\RR}^\KN$ decomposes into two parts, the preimages of $\pm1$ under
$\delta^\KN: X_0^\KN\to (O^\ls)^\KN= S^1$. Denote by $X_{0,\RR}^\KN(\pm1)$
these two fibres.

\begin{proposition}
\label{Prop: real locus is a branched cover of B}
The restriction of $\mu^\KN: X_0^\KN\to B$ to the real locus
$X_{0,\RR}^\KN$ is surjective and has finite fibres. Over
$B\setminus \shA$, this map is a topological covering map with
fibres of cardinality $2^{n+1}$.
\end{proposition}

\begin{proof}
Let $\sigma\in\P$ be a maximal cell. In the canonical identification
$\Phi_\sigma$ of $(\mu^\KN)^{-1}(\sigma)\subset X_0^\KN$ away from
$\shA=\mu(Z)$ given in Lemma~\ref{Lem: mu^KN over maximal cell}, the standard
real involution on $(\mu^\KN)^{-1}(\sigma)\subset X_0^\KN$ lifts to the
involution of $\sigma\times \Hom(\Lambda_\sigma \oplus\ZZ, U(1))$ that acts by
the identity on $\sigma$ and by multiplication by $-1$ on
$\Lambda_\sigma\oplus\ZZ$. The fixed point set of this involution over each
point in $\sigma$ is the set of two-torsion points $(\pm1,\ldots,\pm1)$ of
$U(1)^{n+1}$. In particular, away from $\shA=\mu(Z)\subset B$, the projection
$X_{0,\RR}^\KN\to B$ is a $2^{n+1}$-fold unbranched cover.

In any case, $\Phi_\sigma \big(\sigma\times \Hom(\Lambda_\sigma
\oplus\ZZ, \{\pm1\}) \big)$ is a closed subset in $X_0^\KN$ containing
$X_{\sigma,\RR}\setminus Z$ and projecting with fibres of
cardinality at most $2^{n+1}$ to $\sigma$. The statement on
finiteness of all fibres then follows if $Z$ is nowhere dense in
$X_{0,\RR}^\KN$. This statement follows as in Lemma~\ref{Lem: Z^KN
nowhere dense} noting that the generization maps between stalks of
$\M_{X_0}$ at real points are compatible with the real involution.
\end{proof}

We thus see that $X_{0,\RR}^\KN$ can be understood by studying
(a)~the unbranched covering over $B\setminus\shA$ and (b)~the
behaviour near the log singular locus by means of the canonical
uniformization map $\Phi_\sigma$ of Lemma~\ref{Lem: mu^KN over
maximal cell}. Sometimes, for instance in dimension two, the unbranched cover
together with the fact that $X_{0,\RR}^\KN$ is a topological
manifold, determines $X_{0,\RR}^\KN$ completely.

\begin{remark}
\label{Rem: Transition of X_^KN in codim=1}
We can describe explicitly the homeomorphism of torus bundles
$\Phi_{\sigma'\sigma}$ in the proof of Theorem~\ref{Thm:
(X_0)^KN->B is a torus bundle} locally around some $x\in
B\setminus\shA$, in terms of slab functions. We restrict to the basic case
$\sigma\cap\sigma'=\rho$ of codimension one. Let $f$ be the
function defining the log structure along $X_\rho$ according
to~\eqref{Eqn: Slab equation}.  The Kato-Nakayama space has an additional
$U(1)$-factor coming from the deformation parameter $t$. This
additional factor gets contracted in $X_0$ along $X_\rho$, but not
in $X_0^\KN$. Thus over $X_\rho$, the Kato-Nakayama space is a
$U(1)^2$-fibration. One factor captures the phase of the
deformation parameter $t$, the other the phase of the monomial
$u$ (or $v$) describing $X_\rho$ as a divisor in $X_\sigma$ and
$X_{\sigma'}$, respectively. In these coordinates for $X_0^\KN$ over
$\sigma$ and $\sigma'$, the gluing $\Phi_{\sigma'\sigma}$ is
determined by taking the argument of~\eqref{Eqn: Slab equation}:
\begin{equation}
\label{Eqn: Gluing equation}
\arg(u)+\arg(v)= \kappa_\rho\cdot\arg(t) +\arg(f).
\end{equation}
For the unbranched cover, Lemma~\ref{Lem: mu^KN over maximal cell} together with
the gluing equation~\eqref{Eqn: Gluing equation} in Remark~\ref{Rem: Transition
of X_^KN in codim=1} provide a full description of $X_{0,\RR}^\KN$. Note also
that the gluing equation involves the term $\kappa_\rho\cdot\arg(t)$, which for
$\kappa_\rho$ odd and $\Arg(t)=-1$ leads to a difference in the identification
of branches over neighbouring maximal cells $\sigma\subset B$. Recall by Theorem \ref{Thm: (X_0)^KN->B is a torus bundle} that the Kato--Nakayama space over each maximal cell, away from the discriminant locus, is given by $\sigma\times \Hom(\Lambda_\sigma \oplus\ZZ,
U(1))\setminus \Phi_\sigma^{-1}(Z^\KN)$.
Hence, for any real log Calabi--Yau space  $(X_0,\M_{X_0})$ endowed with a standard real structure,
defined by real lifted open gluing data, $X_0^\KN$ is obtained by gluing
trivial pieces $\sigma\times \Hom(\Lambda_\sigma \oplus\ZZ,
U(1))\setminus \Phi_\sigma^{-1}(Z^\KN)$ via Equation~\eqref{Eqn:
Gluing equation}. This exhibits the real locus $X_{0,\RR}^\KN\setminus
Z^\KN$ as glued from the finite covers
\[\big(\sigma\times \Hom(\Lambda_\sigma
\oplus\ZZ, \{\pm1\})\big)\setminus \Phi_\sigma^{-1}(Z^\KN)\]
over the maximal cells $\sigma$.
This gluing is determined by studying possible identifications of each of these bundles over maximal cells, which is determined by the choice of gluing data.
In general the specific choice of such identifications changes the
topology of $X_{0,\RR}^\KN$, and hence has to be studied case by case. Assuming
without loss of generality that the toric log Calabi-Yau space $(X_0,\M_{X_0})$
is normalized by the open gluing data, we can however give a neat global
description over the large subset $B'\subset B$ considered in
Proposition~\ref{Prop: X_0^KN versus Hom(widehat shP,U(1))}. In the simple case,
there is a retraction of $B\setminus\shA$ to $B'$ and this result is strong
enough to understand the unbranched cover over $B\setminus\shA$ completely. In
the general case, this result can be complemented by separate studies along the
interior of codimension one cells to gain a complete understanding of
the real locus over $B\setminus\shA$.
\end{remark}

As a preparation, we need to discuss the effect of the real
involution on Diagram~\eqref{Eqn: X_0^KN as an extension}, and in
particular on the middle vertical exact sequence
\[
0\lra \ul\ZZ\oplus\ul U(1)\lra \widehat\shP\lra\Lambda\lra 0.
\]
The action on the discrete part $\ul\ZZ$ and $\Lambda$ is induced
by the action on the cohomology of the torus fibres, which is
multiplication by $-1$. Similarly, we can act by multiplication by
$-1$ on each entry of the sequence defining $\shP$, forming the next
to rightmost column in~\eqref{Eqn: X_0^KN as an extension}. For the
extension by $U(1)$, however, taking the pushout with
complex conjugation on $U(1)$, maps the extension class $s\in
\Ext^1(\Lambda, \ul U(1))$ to its complex conjugate $\ol s$. Thus
only if this class is real, reflected in a real choice of open
gluing data (or slab functions as in Proposition~\ref{Prop: Gluing data and slab function
for the standard real structure}), there is an involution on
$\widehat\shP$ inducing multiplication by $-1$ on $\ul\ZZ$ and
$\Lambda$ and the conjugation on $U(1)$. Note also that the
extension class is real if and only if it lies in the image of
$\Ext^1(\Lambda,\ul\ZZ\oplus\{\pm1\})$ under the inclusion
$\{\pm1\}\to U(1)$. In this case, the extension of $\Lambda$ by $\ul
\ZZ\oplus U(1)$ is obtained by pushout from an extension by
$\ul\ZZ\oplus\{\pm1\}$. We now assume such an involution
$\iota_{\widehat\shP}$ of $\widehat\shP$ exists.

\begin{theorem}
\label{Thm: X_{0,RR}^KN in terms of extensions}
Let $(X_0,\M_{X_0})$ be a toric log Calabi-Yau space defined by real lifted open gluing data. Then the real locus in $X_0^\KN$ over $B'\subset B$ is given by
\[
X_{0,\RR}^\KN = \Hom^{\circ}(\widehat\shP,\{\pm1\})|_{B'} \subset
\Hom^{\circ}(\widehat\shP,U(1))|_{B'}.
\] 
where $\Hom^{\circ}(\widehat\shP,U(1))|_{B'}$ is defined as in Proposition~\ref{Prop: X_0^KN versus Hom(widehat shP,U(1))}.
\end{theorem} 
\begin{proof}
Recall the trivialization with fibres $\Lambda_\sigma\oplus\ZZ\oplus
U(1)$ of $\widehat\shP$ over the interior of a maximal $\sigma\in\P$
used in the proof of Proposition~\ref{Prop: X_0^KN versus
Hom(widehat shP,U(1))}. In this trivialization, the involution
$\iota_{\widehat\shP}$ acts by $-1$ on $\Lambda_\sigma\oplus\ZZ$ and
by conjugation on $U(1)$. Taking homomorphisms to $U(1)$ and
restricting to those homomorphisms inducing the identity on the
$U(1)$-factor, identifies the fibres of $X_0^\KN$ over $\Int\sigma$
with $\Hom(\Lambda_\sigma\oplus\ZZ, U(1))$. The fixed point locus of
the induced action of $\iota_{\widehat\shP}$ is then the set of
homomorphisms to the two-torsion points of $U(1)$, that is,
$\Hom(\Lambda_\sigma\oplus \ZZ,\{\pm1\})$, as claimed.
\end{proof}

\begin{remark}
\label{Rem: Monodromy description of X_{0,RR}^KN}
The topology of the $2^{n+1}$-fold cover of $B'$ can
also be described in terms of the monodromy representation as
follows. Analogously to the discussion for $\shP$ in 
Remark~\ref{Rem: monodromy representation}, the monodromy
representation of $\widehat\shP$ is given by viewing
$(c_1(\varphi),s)\in \Ext^1(\Lambda,\ul\ZZ\oplus\ul U(1))$
as a pair of twisted homomorphisms,
\[
(\lambda, \theta):
\pi_1(B',x)\lra \Hom(\Lambda_x, \ZZ\oplus U(1)).
\]
Explicitly, for a closed loop $\gamma$ at $x$, the action of
$(\lambda,\theta)(\gamma) = (\lambda_\gamma,\theta_\gamma)$ on the
fibre of $\widehat\shP_x \simeq\Lambda_x\oplus\ZZ\oplus U(1)$ is
\[
\Lambda_x\oplus\ZZ\oplus U(1)\ni (v,a,\beta)\longmapsto
(T_\gamma\cdot v, \lambda_\gamma\cdot v+a, \theta_\gamma(v)\cdot
\beta).
\]
Here $T_\gamma\in \GL(\Lambda_x)$ is from parallel transport in
$\Lambda$. If the open gluing data $s$ are real, $\theta$ takes
values in $\{\pm1\}\subset U(1)$. Thus in the real case,
$(\lambda,\theta)$ is a twisted homomorphism with values in
$\Hom(\Lambda_x,\ZZ\oplus \{\pm1\})$.

In view of Theorem~\ref{Thm: X_{0,RR}^KN in terms of
extensions}, the monodromy representation of $X_{0,\RR}^\KN$ over
$B'$ is given by the induced action on $\Hom^\circ\big(\Lambda_x
\oplus\ZZ\oplus \{\pm1\}, \{\pm1\} \big)
=\Hom(\Lambda_x,\{\pm1\}) \oplus \Hom(\ZZ, \{\pm1\})$. Note that the
last summand in $\Lambda_x \oplus\ZZ\oplus \{\pm1\}$ does not
contribute to the right-hand side, since we restricted to those
homomorphisms inducing the identity on $0\oplus 0 \oplus\{\pm1\}$.
The action of a closed loop $\gamma$ on
$\Hom(\Lambda_x,\{\pm1\}) \oplus \Hom(\ZZ,\{\pm1\})$ is now readily
computed as
\begin{equation}
\label{Eqn: Monodromy formula for real locus}
(\varphi,\mu)\longmapsto  \big(\varphi\circ T_\gamma+
\mu\circ\lambda_\gamma+ \theta_\gamma,\mu\big),
\end{equation}
Here we wrote the group structure on $\{\pm1\}$ additively.
This formula gives an explicit description of $X_{0,\RR}^\KN$ over
$B'$ in terms of a permutation representation of $\pi_1(B',x)$ on
the set $\check\Lambda_x/2\check\Lambda_x\oplus \{\pm1\}$ of
cardinality $2^{n+1}$.

In this description, the map to the real part
$(O^\ls)^\KN_\RR=\{\pm1\}$ of the Kato-Nakayama space $U(1)$ of the
standard log point, is induced by the inclusion $\ZZ \oplus\{\pm1\}
\to \Lambda_x\oplus\ZZ\oplus \{\pm1\}$. Thus to describe the
fibres over $\{\pm1\}\subset (O^\ls)^\KN$ in $X_{0,\RR}^\KN$
simply amounts to restricting to $\mu=\pm1$ in \eqref{Eqn:
Monodromy formula for real locus}. In particular, $c_1(\varphi)$
only becomes relevant for the fibre over $-1$. This fact can also be
seen from the gluing description of~\eqref{Eqn: Gluing equation},
where $\kappa_\rho$ is the only place for $c_1(\varphi)$ to enter.
\end{remark}



\section{Examples}
\label{Examples}
\subsection{A degeneration of K3 surfaces with non-simple singularities}
\label{Subsect: quartic K3}
In this section we study a toric degeneration of real quartic K3 surfaces, and investigate the topology of the real locus, based on the general results discussed in \S\ref{Subsect: X_0^KN real}. We first study an example of a degeneration of $K3$ surfaces where the associated intersection complex does not have simple singularities. It is
nonetheless locally rigid, and the smoothing algorithm of \cite{affinecomplex} is generalized to this case in dimension two -- for details see \cite[Appendix A4]{theta}.

\begin{figure}
\center{\input{tetrahedron.pspdftex}}
\caption{$(B,\P)$ for a quartic K3 surface}
\label{Fig: tetrahedron}
\end{figure}

Consider $(B,\P)$ the polyhedral affine manifold that as an integral cell
complex is the boundary of a $3$-simplex, and is given by the convex hull of the set of points $\{(1; 0; 0); (0; 1; 0); (0; 0; 1); (-1,-1,-1)\}$, with four focus-focus singularities on
each edge 
as illustrated in Figure~\ref{Fig: tetrahedron}. Note that to ensure simplicity we would require to have only one focus-focus singularity on each edge of $\P$, therefore in this case $(B,\P)$ is not simple. There are four maximal cells, each isomorphic
to the standard simplex in $\RR^2$ with vertices $(0,0)$, $(1,0)$ and $(0,1)$.
The edges have integral length $1$ and are identified pairwise to yield the
boundary of a tetrahedron. On each of the six edges there are four singular
points of the affine structure, with monodromy conjugate to
$\left(\begin{smallmatrix}1& 0\\1& 1\end{smallmatrix}\right)$. In this example $(B,\P)$ does not have simple singularities \cite[\S $5.1$]{logmirror1} as there are four focus-focus singularities on each edge, but it is
locally rigid in the sense of \cite{affinecomplex}, Definition~1.26. Thus the
smoothing algorithm of \cite{affinecomplex} works, yielding a one-parameter
smoothing of $X_0$, one for each choice of slab functions.

Consider the trivial
(``vanilla'') gluing data, that is, $s_{\omega\tau}=1$ for all $\omega,\tau\in
\P$, $\omega\subset\tau$. Then $X_0$ is isomorphic as a scheme to $Z_0 Z_1 Z_2
Z_3=0$ in $\PP^3$, a union of four copies of $\PP^2$. As the MPL-function
defining the ghost sheaf $\ol\M_{X_0}$ we take the function with kink
$\kappa_\rho=1$ on each of the edges. The moduli space of toric log Calabi-Yau
structures on $X_0$ with the given $\ol\M_{X_0}$ is described by the space of
global sections of an invertible sheaf $\shLS$ on the double locus
$(X_0)_\sing$ (\cite{logmirror1}, \S5). This line bundle has degree~$4$ on each of the six
$\PP^1$-components.
The section is explicitly described by the
functions $f_{\rho,y}$. Let $v$ and $v'$ the two vertices adjacent to $\rho$ and let $f_{\rho,v}$ and $f_{\rho,v'}$ be slab functions attached to the connected components of $\rho \setminus \Delta$
containing $v$ and $v'$.
They are 
related by the equation $f_{\rho,v}(x)= x^4 f_{\rho,v'}(x^{-1})$ for $x$ the
toric coordinate on $\PP^1$ (see e.g.\ \cite{affinecomplex}, Equation~1.11). The slab functions $f_{\rho,y}$ for $y$ another connected component of $\rho \setminus \Delta$ are similarly determined in terms of $f_{\rho,v}$.
Explicitly, restricting to the normalized case, we have $f_{\rho,v}=1+a_1 x+a_2
x^2+a_3 x^3+x^4$, the highest and lowest coefficients being~$1$ due to the
normalization condition at the two zero-dimensional toric strata of the
projective line $X_\rho\subset X_0$. The other coefficients $a_i\in\CC$ are
free, to give a total of $6\cdot 3=18$ parameters. Taking into account the
additional deformation parameter $t$, this number is in agreement with the
$19$~dimensions of projective smoothings of $X_0$ (indeed, we obtain a versal family as shown in \cite{RS2}). See the appendix of
\cite{theta} for a discussion of projectivity in this context. 

By Proposition~\ref{Prop: real slab functions} this smoothing is real if and only if all slab functions are real,
that is, if all coefficients $a_i\in\RR$. To obtain $4$~focus-focus
singularities on each edge of $\P$ as drawn in Figure~\ref{Fig: tetrahedron}, we
need to choose the $a_i$ in such a way that the $4$~zeros of $f_{\rho,v}= 1+a_1
x+a_2 x^2+a_3 x^3+x^4$ have pairwise different absolute values. These are then
also all real (if not, as the $a_i$ are real, there is at least one pair of complex conjugated roots, which then have the same absolute value, contradiction). This condition is open in the Euclidean topology, but the closure
is a proper subset of $\RR^3$, the space of tuples $(a_1,a_2,a_3)$ fulfilling
this condition. The precise choice does not matter for the following discussion
and we assume such a choice has been made for each edge.

\begin{proposition}
\label{Prop: Standard real quartic}
Let $(X_0,\M_{X_0})$ be the union of four copies of $\PP^2$ with the real log
structure as described. Assume that the zeros of the all the slab functions $f_{\rho,y}$ are all
negative. Denote by $\shA\subset B$ the pairwise different images of the $24$
singular points of the log structure (the zero loci of the slab functions).

Then the fibre $X_{0,\RR}^\KN(1)$ of $\delta_\RR^\KN: X_{0,\RR}^\KN\to
\{\pm1\}\subset U(1)$ has two connected components, one mapping
homeomorphically to $B$, the other a branched covering of degree
$3$, unbranched over $B\setminus\shA$ and with a simple branch
point over each point of $\shA$. In particular, the latter
component is a closed orientable surface of genus~$10$.
\end{proposition}

\begin{proof}
In the present case of trivial gluing data with real slab functions
$f_{\rho,y}$ without zeros on the positive real axis, the positive real sections
$\sigma\times \{1\}\subset \sigma\times \Hom(\Lambda_\sigma,\{\pm1\})$ of the
pieces of $X_0^\KN$ over maximal cells (Lemma~\ref{Lem: mu^KN over
maximal cell}) are compatible to yield a section of $X_{0,\RR}^\KN(1)\to B$. In
fact, all arguments in Equation~\ref{Eqn: Gluing equation} describing the gluing
vanish. The image of this section is thus a $2$-sphere mapping homomorphically
to $B$. 

To describe the remaining three branches of $X_{0,\RR}^\KN(1)\to B$,
observe that at each point of $\shA$ there are local analytic coordinates
$x,y,w$, defined over $\RR$, with $(X_0,\M_{X_0})$ isomorphic to the central
fibre of the degeneration $xy=t(w+1)$ discussed in Example~\ref{Expl:
Focus-focus singularity II}. In this example we checked that the real locus of the
Kato-Nakayama space has three connected components, with two being sections and
one a two-fold branched cover with one branch point over $B$. Thus
$X_{0,\RR}^\KN(1)\to B$ is a cover of degree~$4$, simply branched at the
$24$~points of $\shA$ and having at least one section.

To finish the proof we have to study the global monodromy representation
$\pi_1(B\setminus \shA)\to S_4$ into the permutations of the four fixed points
of the real involution on a fibre and show it has at most two irreducible
subrepresentations. We compute a part of the affine monodromy representation and
then use Remark~\ref{Rem: Monodromy description of X_{0,RR}^KN} and notably
Equation~\eqref{Eqn: Monodromy formula for real locus} to obtain the induced
monodromy representation in $S_4$.

\begin{figure}
\center{\input{chart_quartic.pspdftex}}
\caption{Chart at a vertex of $(B,\P)$ from Figure~\ref{Fig: tetrahedron}}
\label{Fig: chart at vertex}
\end{figure}

Figure~\ref{Fig: chart at vertex} depicts a chart at a vertex $v$ with its three
adjacent maximal cells. The chart gives the affine coordinates in the union of
the three triangles minus the dotted lines. The locations of the $12$~singular
points on the outer dotted lines are irrelevant in this chart and are hence
omitted. We look at the part of the fundamental group spanned by the three loops
$\gamma_1$, $\gamma_2$, $\gamma_3$. Each loop encircles one
focus-focus singularity on an edge containing the vertex and hence the affine
monodromy along any such loop is conjugate to
$\left(\begin{smallmatrix}1& 0\\1& 1\end{smallmatrix}\right)$. In particular,
the translational part vanishes. Concretely, in standard coordinates of $\RR^2$,
the monodromy matrices $T_i$ along $\gamma_i$ are
\begin{equation}
\label{Eqn: monodromy matrices}
T_1=\left( \begin{matrix} 1&0\\1&1\end{matrix}\right),\quad
T_2=\left( \begin{matrix} 2&-1\\1&0\end{matrix}\right),\quad
T_3=\left( \begin{matrix} 1&-1\\0&1\end{matrix}\right).
\end{equation}
Now while the $\gamma_i$ are not loops inside $B'$ as treated in
Remark~\ref{Rem: Monodromy description of X_{0,RR}^KN}, it is not
hard to see that \eqref{Eqn: Monodromy formula for real locus}
still applies in the present case. We have $\mu=1$ since we look at
$X_{0,\RR}^\KN(1)$ and $\theta_{\gamma_i}=1$ also for the
translational parts. Thus \eqref{Eqn: Monodromy formula for real
locus} says that the branches transform according to the linear part
of the affine monodromy. Now indeed a slab function $f_{\rho,x}$
changes signs locally along the real locus over an edge whenever
crossing a focus-focus singularity. For $X_{0,\RR}^\KN(1)$ this
means that the two branches given by $\Hom(\Lambda_\rho,\{\pm1\})
\subset \Hom(\Lambda_\sigma,\{\pm1\})$ have trivial monodromy around
any focus-focus singularity on $\rho$, while the two other branches
swap.

It thus remains to compute the action of the transpose $T_i^t$ on the two-torsion
points $\ZZ^2/2\ZZ^2$ of the real two-torus
$\Hom(\Lambda_v,U(1)) \simeq\RR^2/\ZZ^2$. These are the four vectors
\[
u_0=\left(\begin{matrix}0\\0\end{matrix}\right),\quad
u_1=\left(\begin{matrix}1\\0\end{matrix}\right),\quad
u_2=\left(\begin{matrix}0\\1\end{matrix}\right),\quad
u_3=\left(\begin{matrix}1\\1\end{matrix}\right),
\]
Here $u_0$ is the point in the positive real locus, yielding the section already treated at the beginning of the proof. The permutation of the indices of the other three
vectors yield the three transpositions $(23)$, $(12)$, $(13)$ for
$T_1^t$, $T_2^t$ and $T_3^t$, respectively. These transpositions act
transitively on $\{u_1,u_2,u_3\}$, showing connectedness of the cover of
degree~$3$. This component is a genus $10$ surface by the Riemann Hurwitz
formula.
\end{proof}
\subsection{K3 surfaces with simple
singularities}

As a second, related family of examples we consider toric degenerations of K3
surfaces such that the associated intersection complex $(B,\P)$ has simple
singularities. We take a subdivision of the polyhedral decomposition $\P$ of the tetrahedra $B$ studied in \S\ref{Subsect: quartic K3}, so that every point in the discriminant locus of the affine structure lies between different vertices of $\P$. By taking this subdivision we ensure that $(B,\P)$ has actually (strongly) semi-simple
singularities as discussed in \S\ref{Subsect: Faithful and simple toric degenerations}. Recall that each toric irreducible component of the central fiber of a toric degeneration has momentum map image corresponding to a maximal cell of the polyhedral decomposition $\P$, hence in this case the central fiber of the degeneration has several more irreducible components, unlike in the non-simple case. We illustrate a face of the tetrahedra $B$, along with the subdivision $\P$ obtained by taking barycentric subdivisions in Figure \ref{Fig: simple}.

\begin{figure}
\center{\input{Simple.pspdftex}}
\caption{The polyhedral decomposition on a triangular face of the tetrahedra in case of simple singularities}
\label{Fig: simple}
\end{figure}

In this case the possible topologies of $X_{0,\RR}^\KN$ are
determined by Theorem~\ref{Thm: X_{0,RR}^KN in terms of extensions}.
Interestingly, for the fibre $X_{0,\RR}^\KN(1)$ over $1\in (O^\ls)^\KN=U(1)$, the question becomes
a purely group-theoretic one. In fact, according to \eqref{Eqn: Monodromy
formula for real locus}, for $\mu=1$ the translational part $\lambda$ of the
affine monodromy representation does not enter in the computation. Moreover, by
a classical result of Livn\'e and Moishezon, the linear part of the monodromy
representation for an affine structure on $S^2$ with $24$ focus-focus
singularities is unique up to equivalence \cite[pg.~179]{Moishezon}. The result
says that there exists a set of standard generators $\gamma_1,\ldots,
\gamma_{24}$ of $\pi_1(S^2\setminus\text{$24$ points},x)$, closed loops mutually
only intersecting at $x$ and with composition $\gamma_1 \cdot\ldots
\cdot\gamma_{24}$ homotopic to the constant loop, such that the monodromy
representation takes the form
\[
T_{\gamma_i}=\begin{cases} T_3,& i\ \text{odd}\\
T_1,& i\ \text{even},\end{cases}
\]
with $T_1$, $T_3$ as in \eqref{Eqn: monodromy matrices}. As in \S\ref{Subsect:
quartic K3}, the corresponding monodromy of the four elements in
$\Hom\big(\Lambda_x, \{\pm1\}\big) \simeq \ZZ^2/2\ZZ^2$ are $(23)$ and $(13)$,
respectively. Thus the computation only depends on the choice of the twisted
homomorphism $\theta\in H^1(B',\check\Lambda\otimes \{\pm1\})$. Now each
$\theta_\gamma$ acts by translation on the fibre $\ZZ^2/2\ZZ^2$. If
$\theta_\gamma$ is non-trivial, the permutation is a double transposition. But
any double transposition together with $(23)$ and $(13)$ act transitively on
the $4$-element set. Thus $X_{0,\RR}^\KN(1)$ is connected as soon as $\theta\neq0$;
otherwise we have two connected components as in Proposition~\ref{Prop: Standard
real quartic}.

\begin{proposition}
Let $(X_0,\M_{X_0})$ be a toric log K3 surface with intersection
complex $(B,\P)$ having simple singularities and endowed with a
standard real structure. Denote by $\theta\in
H^1(B,\check\Lambda\otimes\{\pm1\})$ the argument of the 
associated lifted real gluing data according to Corollary~\ref{Cor:
Standard real structure cohomological}.
Then $X_{0,\RR}^\KN(1)$ has a connected component mapping
homeomorphically to $B\simeq S^2$ if and only if $\theta=0$, and is
otherwise connected.
\qed
\end{proposition}

For $X_{0,\RR}^\KN(-1)$, the translational part of the affine monodromy enters
in \eqref{Eqn: Monodromy formula for real locus}. The action is also by
translation, hence lead to a double transposition if non-trivial. A similar
analysis then shows that if $X_{0,\RR}^\KN(-1)$ is not connected, one connected
component maps homeomorphically to $B\simeq S^2$.

\subsection{A non-standard real structure on local $\PP^2$}
\label{Subsect: local PP2}

As one example with a non-standard real structure we discuss the toric
degeneration of the canonical bundle of $\PP^2$ from \cite[Ex.~5.1]{invitation}. This example features a finite, real wall structure and thus yields
an algebraic family over $\AA^1$ defined over $\RR$, which we view as an
analytic toric degeneration $\shX\to \CC$. Denote by $\iota$ the real involution
of $\shX$. The central fibre has six irreducible components with a single
compact component, $\PP^2\times\PP^1$. The affine manifold $B$ is sketched in Figure~\ref{fig.LocalP2}, with the dotted lines indicating the singular locus of the affine structure.
\begin{figure}[ht]
\input{LocalP2.pspdftex}
\caption{The tropical manifold for a toric degeneration of $K_{\PP^2}$.}
\label{fig.LocalP2}
\end{figure}

The labels at the vertices refer to the corresponding theta functions of
level~$1$ \cite{theta}. There is a global vertical integral vector field with
corresponding affine projection $\pi: B\to\RR^2$ contracting the three vertical
line segments of the central prism. The theta function for the corresponding
asymptotic monomial defines a global holomorphic function without zeroes or
poles, denoted $s$ in \cite{invitation}. Such a function does not exist on
$K_{\PP^2}$ and indeed, the general fibre of $\shX\to \CC$ turns out to be the
complement $K_{\PP^2}\setminus\Gamma_t$ of the graph $\Gamma_t$ of a meromorphic
section without zeros and with poles along the three coordinate lines of
$\PP^2$. In an affine chart containing the zero-dimensional toric stratum
labelled $U$ we have coordinates $z=W/U= Z/X$, $v=V/U= Y/X$ and $r$, fulfilling
the equation
\begin{equation}
\label{Eqn: Chart local PP2}
rvzs=(1+s)t.
\end{equation}
Here $r$ is also the restriction of a canonically defined global function, the theta
function for the primitive integral vector field along the unbounded edge
emanating from the vertex labelled $U$. In this affine chart the embedding of
the fibre $X_t$ over $t$ into $K_{\PP^2}$ is defined set-theoretically by
mapping $(v,z,r,s)$ to the $2$-form $r dz\wedge dv$ in the fibre of $K_{\PP^2}$
over $[1,v,z]\in\PP^2$. In particular, $r$ restricts to a linear function on
each fibre of $K_{\PP^2}$ and \eqref{Eqn: Chart local PP2} expresses $s$ in
terms of $r$ and the local coordinates $z,v$ of $\PP^2$:
\[
s=\frac{t}{vzr-t}.
\]
Thus the hypersurface $\Gamma_t$ removed in this chart is given by $r=t/vz$, the
pull-back of the standard hyperbola via $r$ and $zv$. Note also that for
$t\neq0$ the restriction of $s$ to the zero section $r=0$ or to the coordinate
lines $zv=0$ is constant $-1$.

Now the bounded cell, the prism, has a reflection symmetry inverting the
vertical lines and acting by swapping the vertices $X$ with $U$, $Y$ with $V$
and $Z$ with $W$. It can be checked that this involution extends uniquely to an
involution $\kappa$ of $B$ acting by affine isomorphisms with integral linear
part (and half-integral translational part) on each cell and respecting the
projection $\pi:B\to\RR^2$. The fixed locus $B^\kappa\subset B$ is a copy of
$\RR^2$ embedded in $B$. Taking the discriminant locus $\Delta$ barycentric, it
holds $\Delta\subset B^\kappa$ and then $\kappa$ induces an orientation
reversing involution of the integral affine manifold $B$ with its polyhedral
decomposition $\P$.

This involution respects the initial slab functions and in turn the whole wall
structure and hence induces an involution $\tilde\kappa$ of the toric
degeneration $\shX\to \CC$. Explicitly, $\tilde\kappa$ swaps the two lines of
six defining local equations displayed in \cite[Ex.~5.1]{invitation}. The
action on our coordinates in \eqref{Eqn: Chart local PP2} is as follows:
\[
\tilde\kappa^*(z)=z,\quad \tilde\kappa^*(v)=v,\quad
\tilde\kappa^*(s)=s^{-1},\quad \tilde\kappa^*(r)=rs.
\]
Thus away from the coordinate lines of $\PP^2$, the involution $\tilde\kappa$ is
the involution of $\CC\setminus \{t/vz\}\simeq\CC^*$ that interchanges the pole
of $s$ at $r=t/vz$ with the zero at $\infty$ in such a way that $s$ turns into
$s^{-1}$. For fixed $t\neq 0$, taking $vz\to 0$ in Equation~\eqref{Eqn: Chart
local PP2} implies that $s=-1$ over the coordinate lines. Now $s=-1$ is a fixed
point of the involution swapping $s$ and $s^{-1}$. Hence $\tilde\kappa$ acts
trivially over the coordinate lines of $\PP^2$.

Observe that $\tilde\kappa$ commutes with the standard
real involution $\iota$ of $\shX$, given by complex conjugation. The composition
\[
\tilde\iota=\iota\circ\tilde\kappa=\tilde\kappa\circ\iota
\]
thus defines another antiholomorphic involution of $\shX$. This non-standard real
involution still commutes with the canonical real structure on $\PP^2$ and hence
lies over the real locus $\RR\PP^2$ of $\PP^2=\CC\PP^2$. But for
$t\in\RR\setminus\{0\}$, in the above embedding of $X_t$ into $K_{\PP^2}$, the
non-standard real involution acts fibrewise by $\tilde\iota^*(s)= \ol s^{-1}$.
Thus the fibre of $(X_t)_\RR$ over $[1,z,v]\in\RR\PP^2$ with $vz\neq0$ is the
circle in the $r$-plane with center $r=t/vz\in\RR\setminus\{0\}$ and passing
through the origin, that is, with radius $|t/vz|$. Indeed, for any
$z,v\in\RR\setminus\{0\}$ and $r=\frac{t}{vz}\big(1+ e^{\varphi i}\big)\in\CC$
we can take $s=e^{-\varphi i}$ to obtain a point in $(X_t)_\RR$. The radius
$|t/vz|$ of the circles tend to infinity as we approach the union of coordinate
lines $vz=0$ in $X_t$, but is constant for fixed $t/vz$. Thus the limits of
these circles cover all of the $r$-plane as illustrated in Figure \ref{Fig: r-plane}. Summarizing, the fibre of $(X_t)_\RR$
over $[1,z,v]\in \RR\PP^2$ is a circle for $vz\neq 0$, while over the union of
coordinate lines $\bigcup_3\RR\PP^1\subset\RR\PP^2$ it is the whole
$K_{\PP^2}$-fibre.

\begin{figure}
\center{\scalebox{.4}{\input{circles.pspdftex}}}
\caption{}
\label{Fig: r-plane}
\end{figure}

Let us compare this picture with the Kato-Nakayama space of $X_0$. We are only
interested in the restriction of $X_0^\KN$ to the fixed locus $B^\kappa\subset
B$. Define $\Lambda_{B^\kappa}$ as the sheaf of integral tangent vectors on
$B^\kappa$ fixed under $\kappa_*$. Then since $\ker\pi_*$ is spanned by a global
vertical tangent vector field $\xi$, we have a direct sum decomposition
\[
\iota_*\Lambda|_{B^\kappa}= \Lambda_{B^\kappa}\oplus\ul\ZZ
\]
with the $\ZZ$-factor spanned by $\xi$. The sheaf $\Lambda_{B^\kappa}$ consists
of those integral vectors in $\pi^*\Lambda_{\RR^2}|_{B^\kappa}$ that are invariant under
local monodromy. Thus $\Lambda_{B^\kappa}$ is a subsheaf of the trivial sheaf
$(\pi|_{B^\kappa})^*\Lambda_{\RR^2} =\ul{\ZZ^2}$ with quotient a sheaf
$\Lambda_\Delta$ supported on the discriminant locus $\Delta\subset B^\kappa$:
\[
0\lra \Lambda_{B^\kappa}\lra \ul{\ZZ^2}\lra \Lambda_\Delta\lra 0.
\]
The quotient $\Lambda_\Delta$ has stalks $\ZZ^2$ at the three vertices of
$\Delta$, with generization maps to $\ZZ$ on the interior of each of the edges of
$\Delta$. Said differently, the stalk of $\Lambda_{B^\kappa,x}$ at $x\in\Delta$
consists of the tangent vectors of the minimal cell containing $x$ in the
polyhedral decomposition $\P_\kappa$ of $B^\kappa=\RR^2$ induced by $\P$.

In the canonical description of $X_0^\KN$ over each maximal cell in
Lemma~\ref{Lem: mu^KN over maximal cell}, for each maximal cell $\sigma\in
\P_\kappa$, we obtain a canonical continuous surjection
\[
\Phi_\sigma:\sigma\times\Hom\big(\Lambda_\sigma,U(1)\big)\times U(1)\times U(1)\lra (\mu_\KN)^{-1}(\sigma)\subset X_0^\KN.
\]
The two $U(1)$-factors correspond to the phases of $s$ and $t$ respectively.
Over the $1$-skeleton of $\P_\kappa$ exactly those directions from
$\Hom(\Lambda_\sigma,U(1))\simeq U(1)^2$ get contracted by
$\Phi_\sigma$ that factor over $\Lambda_\Delta$. Thus $\Phi_\sigma$ induces a canonical bijection $\Hom(\Lambda_{B_\kappa,x},U(1))\times U(1)^2\simeq \mu_\KN^{-1}(x)$. Now
\[
\bigcup_{x\in\sigma} \Hom(\Lambda_{B_\kappa,x},U(1))
\]
is canonically in bijection with the toric surface $X_\sigma$ obtained by
viewing $\P_\kappa$ as a polyhedral decomposition of $\RR^2$. For example, for
$\sigma$ the interior triangle, we obtain $X_\sigma=\PP^2$. Thus
$\mu_\KN^{-1}(B^\kappa)$ is covered canonically by the product with $U(1)^2$ of
a union of the four toric surfaces defined by the four maximal cells of the
polyhedral decomposition $\P_\kappa$ of $B^\kappa=\RR^2$. Explicitly, these four
toric surfaces are $X_{\sigma_0}=\PP^2$ and the three restrictions
$X_{\sigma_1}$, $X_{\sigma_2}$, $X_{\sigma_3}$ of $K_{\PP^2}\simeq
\O_{\PP^2}(-3)$ to the three coordinate lines in $\PP^2$. The four pieces are
glued non-trivially using $\Arg(1+s)$, as in the description of Remark~\ref{Rem:
Transition of X_^KN in codim=1} with slab function $f=1+s$.

Our real involution $\tilde\iota$ induces the canonical real involution on the
four toric surfaces $X_{\sigma_0},\ldots,X_{\sigma_3}$ and on the $U(1)$-factor
for the phase of $t$. On $s$ the action is $s\mapsto \ol{s}^{-1}$, which acts
trivially on its phase. Thus $X_{0,\RR}^\KN(\pm 1)$ is a union of
$U(1)\times\RR\PP^2$, with $\RR\PP^2$ covering $4:1$ the triangle $\sigma_0$,
and the restrictions of $K_{\PP^2}$ to the real loci $\RR\PP^1$ in the three
coordinate lines of $\PP^2$, each a $U(1)$-bundle over the remaining cells
$\sigma_1,\sigma_2,\sigma_3$. As expected from Theorem~\ref{Thm: KN for X0 in simple case}
and Proposition \ref{Prop: pairs}, this is compatible with the above discussion for the real
locus of $X_t$. Note that since the present $B$ is not compact, we have to first
compactify the situation to apply Theorem~\ref{Thm: KN for X0 in simple case} and Proposition
\ref{Prop: pairs}. In the present example this compactification can easily be
achieved by replacing the unbounded cells of $\P$ by appropriate bounded cells.
Details are left to the reader.
\medskip

We have made no serious attempt to uncover the general structure of
$X_{0,\RR}^\KN$ involving non-trivial involutions $\kappa$ of $B$. It seems that
it would be easiest to restrict to cases where $B^\kappa$ is already covered by
cells of $\P$. This situation requires the more general setup of \cite{theta} to
admit vertices of $\P$ to lie in the singular locus $\Delta$ of the affine
structure. Then $B^\kappa$ defines an analytic subspace $\shY\subset\shX$.
Moreover, the restriction of $\mu_\KN$ to the fixed locus $B_\kappa$ then
factors over $\shY^\KN$ and we expect $\mu_\KN^{-1}(B^\kappa)\to Y_0^\KN$ to be
a $U(1)^r$-fibration with $r$ the local codimension of $B^\kappa$ in $B$. In the
real situation, the real involution should act trivially on these
$U(1)^r$-fibres, exhibiting $X_{0,\RR}^\KN$ as a $U(1)^r$-fibration over
$Y_{0,\RR}^\KN$.

\sloppy

\fussy

\end{document}